\documentclass[11pt,reqno]{amsart}
\usepackage{amsmath,amssymb,mathtools,amsfonts,amsthm,enumerate,natbib,color,ifthen,hyperref}
\usepackage{xargs}
\usepackage[textwidth=4cm, textsize=footnotesize]{todonotes}
\usepackage[latin1]{inputenc}
\usepackage{wrapfig}

\usepackage{float}
\usepackage[caption = false]{subfig}
\usepackage{graphicx}

\usepackage{aliascnt,bbm}
\def\rset{\mathbb R}
\def\zset{\mathbb Z}

\def\eqsp{\;}

\newcommand{\pscal}[2]{\left\langle#1,#2\right\rangle}
\newcommand{\KL}[2]{\mathsf{KL}\left(#1 \vert #2\right)}

\newcommand{\eqdef}{\ensuremath{\stackrel{\mathrm{def}}{=}}}

\def\Xset{\mathcal{X}} 
\def\Zset{\mathcal{Z}} 
\def\F{\mathcal{F}} 
\def\B{\mathcal{B}} 
\def\cB{\mathsf{B}} 
\def\e{\mathcal{E}}

\def\E{\mathbb{E}}

\def\M{\mathcal{M}}

\def\G{\mathcal{G}}

\def\A{\mathcal{A}}

\def\H{\mathcal{H}}

\newcommandx\sequence[3][2=t,3=\zset]
{\ifthenelse{\equal{#3}{}}{\ensuremath{\{ #1_{#2}\}}}{\ensuremath{\{ #1_{#2}, \eqsp #2 \in #3 \}}}}

\def\PP{\mathbb{P}} 
\newcommand{\CPP}[3][]
{\ifthenelse{\equal{#1}{}}{{\mathbb P}\left(\left. #2 \, \right| #3 \right)}{{\mathbb P}_{#1}\left(\left. #2 \, \right | #3 \right)}}
\def\PE{\mathbb{E}} 
\newcommand{\CPE}[3][]
{\ifthenelse{\equal{#1}{}}{{\mathbb E}\left[\left. #2 \, \right| #3 \right]}{{\mathbb E}_{#1}\left[\left. #2 \, \right | #3 \right]}}

\def\L{\mathcal{L}} 

\def\tv{\mathrm{tv}}

\def\Cset{\mathcal{C}} 

\def\q{\textsf{q}}

\def\r{\textsf{r}}
\def\S{\mathcal{S}}

\usepackage{bm}

\theoremstyle{plain}
\newtheorem{theorem}{Theorem}
\newtheorem{assumption}{H\hspace{-3pt}}
\newtheorem{assumptionB}{B\hspace{-3pt}}
\newtheorem{assumptionC}{C\hspace{-3pt}}

\newaliascnt{proposition}{theorem}

\aliascntresetthe{proposition}

\newaliascnt{lemma}{theorem}
\newtheorem{lemma}[lemma]{Lemma}
\aliascntresetthe{lemma}
\newaliascnt{corollary}{theorem}
\newtheorem{corollary}[corollary]{Corollary}
\aliascntresetthe{corollary}

\theoremstyle{definition}
\newaliascnt{definition}{theorem}

\aliascntresetthe{definition}

\newtheorem{algorithm}{Algorithm}

\newaliascnt{remark}{theorem}
\newtheorem{remark}[remark]{Remark}
\aliascntresetthe{remark}

\newaliascnt{example}{theorem}
\newtheorem{example}[example]{Example}
\aliascntresetthe{example}

\def\rmd{\mathrm{d}}

\def\1{\mathbbm{1}}

\DeclareMathOperator*{\argmin}{{\mathsf{Argmin}}}
\DeclareMathOperator*{\argmax}{{\mathsf{Argmax}}}

\begin{document}

\title[Large-scale Quasi-Bayesian inference with spike-and-slab priors]{An approach to large-scale Quasi-Bayesian inference with spike-and-slab priors}\thanks{This work is partially supported by the NSF grant DMS 1513040}

\author{Yves Atchad\'e}\thanks{ A. Y. Atchad\'e: Boston University, 111 Cummington Mall, Boston, 02215 MA, United States. {\em E-mail address:} atchade@bu.edu}
\author{Anwesha Bhattacharyya}\thanks{ A. Bhattacharyya: University of Michigan, 1085 South University, Ann Arbor,
  48109, MI, United States. {\em E-mail address:} anwebha@umich.edu}



\maketitle

\begin{center} (Aug. 2019) \end{center}

\begin{abstract}
We propose a general framework using spike-and-slab prior distributions to aid with the development of high-dimensional Bayesian inference. Our framework allows inference with  a general quasi-likelihood function. We show that highly efficient and scalable Markov Chain Monte Carlo (MCMC) algorithms can be easily constructed to sample from the resulting quasi-posterior distributions.

We study the large scale behavior of the resulting quasi-posterior distributions as the dimension of the parameter space grows, and we establish several convergence results. In large-scale applications where computational speed is important, variational approximation methods are often used to approximate posterior distributions.  We show that  the contraction behaviors of the quasi-posterior distributions can be exploited  to  provide theoretical guarantees  for  their variational approximations.  We illustrate the theory with some simulation results from Gaussian graphical models, and sparse principal component analysis.
\end{abstract}

\setcounter{secnumdepth}{3}

\section{Introduction}\label{sec:intro}
We consider the problem of estimating a $p$-dimensional parameter using a dataset $z\in\Zset$, and a likelihood or quasi-likelihood function $\ell:\;\rset^p\times \Zset\to\rset$, where $\Zset$ denote a sample space equipped with a reference sigma-finite measure $\rmd z$. We assume that the quasi-likelihood function $(\theta,z)\mapsto \ell(\theta,z)$ is a  jointly measurable function on $\rset^p\times \Zset$, and thrice differentiable in the parameter $\theta$ for any $z\in\Zset$. We take a Bayesian approach with a spike-and-slab prior for $\theta$. The prior requires the  introduction of a new parameter $\delta\in \Delta\eqdef\{0,1\}^p$ with prior  distribution  $\{\omega(\delta),\;\delta\in\Delta\}$ which can be used for variable selection. The components of $\theta$ are  then assumed to be conditionally independent given $\delta$, and $\theta_j\vert \delta$ has a mean zero Gaussian distribution with precision parameter $\rho_1>0$  if $\delta_j=1$ (slab prior),  or a mean zero Gaussian distribution with precision parameter $\rho_0>0$ if $\delta_j=0$ (spike prior). Spike-and-slab priors have been popularized by the seminal works  \cite{mitchell:beauchamp:88,george:mcculloch97} among others. Versions with a point-mass at the origin are known to have several optimality properties in high-dimensional problems (\cite{johnstone:04,castillo:etal:12,castillo:etal:14,atchade:15b}),  but are computationally difficult to work with. In this work we follow \cite{george:mcculloch97,narisetti:he:14} and others, and  replace the point-mass at the origin by a small-variance Gaussian distribution. 
We then propose to study the following quasi-posterior distribution  on $\Delta\times \rset^p$,
\begin{equation}\label{post:Pi}
\Pi(\delta,\rmd\theta\vert z) \propto  e^{\ell(\theta_\delta,z)} \omega(\delta) \left(\frac{\rho_1}{2\pi}\right)^{\frac{\|\delta\|_0}{2}} \left(\frac{\rho_0}{2\pi}\right)^{\frac{p-\|\delta\|_0}{2}} e^{-\frac{\rho_1}{2}\|\theta_\delta\|_2^2}  e^{-\frac{\rho_0}{2}\|\theta-\theta_\delta\|_2^2}\rmd\theta,
\end{equation}
assuming that it is well-defined, where for $\theta\in\rset^p$, and $\delta\in\Delta$, $\theta_\delta$ denote their componentwise product. A distinctive feature of (\ref{post:Pi}) is that we have also replaced the quasi-likelihood $\ell(\theta;z)$ by a sparsified version $\ell(\theta_\delta;z)$. In other words, even if $\ell$ is a standard log-likelihood, (\ref{post:Pi}) would still be different from the Gaussian-Gaussian spike-and-slab posterior distribution of \cite{george:mcculloch97,narisetti:he:14}. To the best of our knowledge this sparsification trick has not been explored in the literature. It has the effect of  bringing (\ref{post:Pi}) closer to the point-mass spike-and-slab posterior distribution  in terms of statistical performance, while  at the same time providing tremendous computational speed as we will see.

By working with a general quasi-likelihood function this work also contributes to a growing Bayesian literature where  non-likelihood functions are combined with prior distributions for  the sake of tractability and scalability (\cite{chernozhukov:hong03,jiang:tanner08,liao:jiang:11,yang:he:2012,kato:13,li:jiang:14,atchade:15b,atchade:15:c}). Non-likelihood functions (also known as quasi-likelihood,  pseudo-likelihood or composite likelihood functions) are routine in frequentist statistics, particular to deal with large scale problems (\cite{meinshausen06,zou:etal:06,shen:huang:08,ravikumaretal10,varin:etal:11,lei:vu:15}).  In semi/non-parametric statistics and econometrics, the idea is closely related to moments restrictions inference (\cite{ichumura:93,chernozhukov:etal:07,atchade:15:c}). 

At a high-level, our main contribution  can be described as follows: given a  log-quasi-likelihood function $\ell$ and a random sample $Z$ such that $\ell(\cdot; Z)$ is (locally) strongly concave with maximizer located near some parameter value of interest $\theta_\star\in\rset^p$, we show that the distribution  (\ref{post:Pi}) puts most of its probability mass around $(\delta_\star,\theta_\star)$, where $\delta_\star$ is the support of $\theta_\star$.  Precise statements can be found in Theorem \ref{thm:0} and Theorem \ref{thm1}. The parameter value $\theta_\star$ is typically (but not necessarily) defined as  the maximizer of the population version of the log-quasi-likelihood function:
\[\theta_\star = \argmax_{\theta\in\rset^p} \;\PE_\star\left[\ell(\theta;Z)\right].\]
We use Theorem \ref{thm:0} to argue in Section \ref{subsec:algo1} that the sparcification trick used in (\ref{post:Pi})  significantly speeds up MCMC computation compared to the state of the art. 

For sufficiently strong signal $\theta_\star$, we show that $\Pi$ actually behaves like a product of a point mass at $\delta_\star$ and the Gaussian approximation of the conditional distribution of $\theta$ given $\delta=\delta_\star$ in $\Pi$ (Bernstein-von Mises approximation). Precise statements can be found in Theorem \ref{thm:KL}. The results have implications for  variational approximation methods, and  as an application of the main results, we derive some  sufficient conditions under which  variational approximations of $\Pi$ are consistent. We illustrate the theory with  examples from Gaussian graphical models (Section \ref{sec:lin:reg}), and sparse principal component analysis (Section \ref{sec:spca}).

The paper is organized as follows. We study the sparsity and statistical properties of $\Pi$ in Section \ref{sec:sparsity} and \ref{sec:contraction} respectively. The Bernstein-von  Mises theorem and the behavior of their variational approximations are considered in Section \ref{sec:post:approx}.  We illustrate these results by considering the problem of inferring Gaussian graphical models in Section \ref{sec:lin:reg}, and sparse principal component estimation in Section \ref{sec:spca}. All the proofs are collected in the appendix.

\subsection{Notation}
Throughout we equip the Euclidean space $\rset^p$ ($p\geq 1$ integer) with its usual Euclidean inner product $\pscal{\cdot}{\cdot}$ and norm $\|\cdot\|_2$, its Borel sigma-algebra, and its Lebesgue measure. All vectors $u\in\rset^p$ are column-vectors unless stated otherwise. We also use the following norms on $\rset^p$: $\|\theta\|_1\eqdef \sum_{j=1}^p|\theta_j|$, $\|\theta\|_0\eqdef\sum_{j=1}^p \textbf{1}_{\{|\theta_j|>0\}}$, and $\|\theta\|_\infty\eqdef \max_{1\leq j\leq p}|\theta_j|$.

We set $\Delta\eqdef\{0,1\}^p$. For $\theta,\theta'\in\rset^p$, $\theta\cdot\theta'\in\rset^p$ denotes the component-wise product of $\theta$ and $\theta'$. For $\delta\in\Delta$, we set $\rset^p_\delta\eqdef\{\theta\cdot\delta:\,\theta\in\rset^p\}$, and we write $\theta_\delta$ as a short for $\theta\cdot\delta$. 
For $\delta,\delta'\in\Delta$, we write $\delta\supseteq\delta'$ to mean that for any $j\in\{1,\ldots,p\}$, whenever $\delta_j'=1$, we have $\delta_j=1$. Given $\theta\in\rset^p$, and $\delta\in\Delta\setminus\{0\}$, we write $[\theta]_\delta$ to denote the  $\delta$-selected components of $\theta$ listed in their order of appearance: $[\theta]_\delta=(\theta_j,\;j\in\{1\leq k\leq p:\;\delta_k=1\})\in\rset^{\|\delta\|_0}$. Conversely, if $u\in\rset^{\|\delta\|_0}$, we write $(u,0)_\delta$ to denote the element of $\rset^p_\delta$ such that $[(u,0)_\delta]_\delta=u$.


If $f(\theta,x)$ is a real-valued function that depends on the parameter $\theta$ and some other argument $x$, the notation $\nabla^{(k)} f(\theta,x)$, where $k$ is an integer, denotes the $k$-th partial derivative with respect to $\theta$ of the map $(\theta,x)\mapsto f(\theta,x)$, evaluated at $(\theta,x)$. For $k=1$, we write $\nabla f(\theta,x)$ instead of $\nabla^{(1)} f(\theta, x)$.

A continuous function $\r:\;[0,+\infty)\to [0,+\infty)$ is called a rate function if $\r(0)=0$, $\r$ is increasing and $\lim_{x\downarrow 0} \r(x)/x=0$. 


All constructs  and other constants in the paper (including the sample size $n$)  depend a priori on the dimension $p$.  And we carry the asymptotics  by letting $p$ grow to infinity. We say that  a term $x\in\rset$ is an absolute constant if $x$ does not depend on $p$. Throughout the paper $C_0$ denotes some generic absolute constant whose actual value may change from one appearance to the next.

\section{Main assumptions and Posterior sparsity}\label{sec:sparsity}

We introduce here our two main assumptions. 
We set
\[\L_{\theta_1}(\theta;z) \eqdef \ell(\theta;z) - \ell(\theta_1;z) -\pscal{\nabla\ell(\theta_1;z)}{\theta-\theta_1},\;\;\theta\in\rset^p,\]
and we assume that the following holds.

\begin{assumption}\label{H1}
We observe a $\Zset$-valued random variable $Z\sim f_\star$, for some probability density $f_\star$ on $\Zset$. Furthermore there exists $\delta_\star\in\Delta$, $\theta_\star\in\rset^p_{\delta_\star}$, $\theta_\star\neq \textbf{0}_p$, finite positive constants $\bar\rho,\bar\kappa$,  such that $\PP_\star(Z\in \e_0)>0$, where 
\begin{multline*}
\e_{0} \eqdef \left\{z\in\Zset:\; \Pi(\cdot\vert z) \mbox{ is well-defined, }\;\; \|\nabla\ell(\theta_\star;z)\|_\infty\leq \frac{\bar\rho}{2},\;\mbox{ and }\;\;\right.\\
\left. \L_{\theta_\star}(\theta;z) \geq -\frac{\bar{\kappa}}{2}\|\theta-\theta_\star\|_2^2,\;\mbox{ for all }\theta\in\rset^p_{\delta_\star} \right\}.\end{multline*}
Furthermore, we assume that the prior parameter $\rho_1$ satisfies $32\rho_1 \|\theta_\star\|_\infty \leq \bar \rho$, and we write $\PP_\star$ and $\PE_\star$ to denote probability and expectation operator under $f_\star$.
\end{assumption}

\medskip
\begin{remark}
H\ref{H1} is very mild. Its main purpose  is to introduce the data generating process, the true value of the parameter,  and their relationship to the quasi-likelihood function. Specifically, since $\nabla\ell(\cdot;z)$ is null at the maximizer of $\ell(\cdot;z)$, having $z\in\e_0$ implies that the maximizer of $\ell(\cdot;z)$ is close to $\theta_\star$ in some sense, and  the largest restricted (restricted to $\rset^p_{\delta_\star}$) eigenvalue of the second derivative of $-\ell(\cdot;z)$ is bounded from above by $\bar\kappa$.  
The assumption that $\theta_\star\neq {\bf 0}_p$   is made only out of mathematical convenience. All the results below continue to hold when $\theta_\star = {\bf 0}_p$ albeit with minor adjustments. 
\vspace{-0.6cm}
\begin{flushright}
$\square$
\end{flushright}
\end{remark}

For convenience we will write $s_\star\eqdef\|\theta_\star\|_0$ to denote the number of non-zero components of the elements of $\theta_\star$.  We assume next that the prior on $\delta$ is a product of independent Bernoulli distribution with small probability of success.

\begin{assumption}\label{H2}
We assume that
\[\omega(\delta)= \q^{\|\delta\|_0}(1-\q)^{p-\|\delta\|_0},\;\;\;\delta\in\Delta,\]
where $\mathsf{q}\in (0,1)$ is such that $\frac{\mathsf{q}}{1-\mathsf{q}} = \frac{1}{p^{u+1}}$, for some  absolute constant $u>0$. Furthermore we will assume that $p\geq 9$, $p^{u/2}\geq 2e^{2\rho_1}$.
\end{assumption}
\medskip

Discrete priors as in H\ref{H2} and generalizations  were introduced by \cite{castillo:etal:12}. This is a very strong prior distribution that is well-suited for high-dimensional problems with limited sample where the signal is believed to be very sparse. It should be noted that this prior can perform poorly if these conditions are not met. We show next that  the resulting posterior distribution is also typically sparse. 

\begin{theorem}\label{thm:0}
Assume H\ref{H1}-H\ref{H2}.   Suppose that there exists a rate function $\r_0$ such that for all $\delta\in\Delta$,
\begin{multline}\label{eq:curvature:cond:thm0}
\log\PE_\star\left[\textbf{1}_{\e}(Z)e^{\L_{\theta_\star}(u;Z)  + \left(1-\frac{\rho_1}{\bar\rho}\right)\pscal{\nabla\ell(\theta_\star;Z)}{u-\theta_\star}} \right] \\
\leq \left\{\begin{array}{ll} -\frac{1}{2}\r_0(\|\delta_\star\cdot(u-\theta_\star)\|_2) & \mbox{ if } \;\|\delta_\star^c\cdot(u-\theta_\star)\|_1 \leq 7\|\delta_\star\cdot(u-\theta_\star)\|_1\\ 0 & \mbox{ otherwise}\end{array}\right.,\end{multline}
for some measurable subset $\e\subseteq\e_0$. Let $\mathsf{a}_0\eqdef -\min_{x>0}\left[\r_0(x) -4\rho_1 s_\star^{1/2}x\right]$.  If for some absolute constant $c_0$ we have
\begin{equation}\label{eq:cond:thm:0}
s_\star \left(\frac{1}{2} +2\rho_1\right)  + \frac{s_\star}{2}\log\left(1+\frac{\bar \kappa}{\rho_1}\right) +\frac{\mathsf{a_0}}{2} +2\rho_1\|\theta_\star\|_2^2 \leq c_0s_\star\log(p),\end{equation}
then it holds that for all $j\geq 1$
\[\PE_\star\left[\textbf{1}_{\e}(Z)\Pi\left(\|\delta\|_0\geq s_\star\left(1 + \frac{2(1+c_0)}{u}\right)+  j\;\vert Z\right)\right] \leq  \frac{2}{p^{\frac{uj}{2}}}.\]
\end{theorem}
\begin{proof}
See Section \ref{sec:proof:thm:0}.
\end{proof}
\medskip

Theorem \ref{thm:0}  is analogous to Theorem 1 of \cite{castillo:etal:14}, and Theorem 3 of \cite{atchade:15b}, and says that the quasi-posterior distribution $\Pi$ is automatically sparse in $\delta$ (of course $\theta$ is never sparse). The main contribution here is the fact that this behavior  holds with Gaussian  slab priors. The condition in (\ref{eq:cond:thm:0}) implies that the precision parameter of the slab density (that is $\rho_1$)  should be of order $\log(p)$ or smaller. Simulation results (not reported here) show indeed that the method performs poorly if $\rho_1$ is taken too large. 

Roughly speaking, the condition (\ref{eq:curvature:cond:thm0}) is expected to hold if
\[\textbf{1}_{\e_0}(Z)\L_{\theta_\star}(u;Z) \leq - \log  \PE_\star\left[e^{\left(1-\frac{\rho_1}{\bar\rho}\right)\pscal{\nabla\ell(\theta_\star;Z)}{u-\theta_\star}} \right],\]
for all $u$ in the cone $\Cset = \{u\in\rset^p:\;\|\delta_\star^c\cdot(u-\theta_\star)\|_1 \leq 7\|\delta_\star\cdot(u-\theta_\star)\|_1\}$. If the quasi-score $\nabla\ell(\theta_\star;Z)$ is sub-Gaussian, then  the right-hand side of the last display is lower bounded by $-c_0(1-\rho_1/\bar\rho)^2\|u-\theta_\star\|_2^2$, for some positive constant $c_0$. In this case  (\ref{eq:curvature:cond:thm0}) will hold if
\[\textbf{1}_{\e_0}(Z)\L_{\theta_\star}(u;Z) \leq -c_0(1-\rho_1/\bar\rho)^2\|u-\theta_\star\|_2^2,\]
for all $u\in\Cset$. Hence (\ref{eq:curvature:cond:thm0}) is a form  restricted strong concavity of $\ell$ over $\Cset $.  We refer the reader to   \cite{negahbanetal10} for more details on restricted strong concavity.

\subsection{Implications for Markov Chain Monte Carlo sampling}\label{subsec:algo1}
Theorem \ref{thm:0} has  implications  for Markov Chain Monte Carlo (MCMC) sampling. To show this we consider a  Metropolized-Gibbs strategy  to sample from $\Pi$ whereby we update $\theta$ keeping $\delta$ fixed, and then update $\delta$ keeping $\theta$ fixed -- we refer the reader to (\cite{robertetcasella04}) for an introduction to basic MCMC algorithms. Note that given $\delta$, $[\theta]_\delta$ and $[\theta]_{\delta^c}$ are conditionally independent, and $[\theta]_{\delta^c}\stackrel{i.i.d.}{\sim}\textbf{N}(0,\rho_0^{-1})$, whereas  $[\theta]_\delta$ can be updated using either its full conditional distribution when available, or using an  extra MCMC update. For each $j$, given $\theta$ and $\delta_{-j}$, the variable $\delta_j$ has a closed-form Bernoulli distribution. However, we choose to update $\delta_j$ using an Independent Metropolis-Hastings kernel with a $\textsf{Ber}(0.5)$ proposal. Putting these steps together yields the following algorithm.

\begin{algorithm}\label{algo:basic}
Draw $(\delta^{(0)},\theta^{(0)}) \in\Delta\times\rset^p$ from some initial distribution. For $k=0,\ldots,$ repeat  the following. Given $(\delta^{(k)},\theta^{(k)})= (\delta,\theta)\in\Delta\times\rset^p$:
\begin{description}
\item [(STEP 1)] For all $j$ such that $\delta_j=0$, draw $\theta_j^{(k+1)}\sim \textbf{N}(0,\rho_0^{-1})$. Using $[\theta]_{\delta}$, draw jointly $[\theta^{(k+1)}]_{\delta}$ from some appropriate MCMC kernel on $\rset^{\|\delta\|_0}$ with invariant distribution proportional to
\[u\mapsto e^{\ell\left((u,0)_\delta;z\right) -\frac{\rho_1}{2}\|u\|_2^2}.\]
\item [(STEP 2)] Given $\theta^{(k+1)}=\bar\theta$, set $\delta^{(k+1)} =\delta^{(k)}$  and do the following for $j=1,\ldots,p$. Draw $\iota\sim\textbf{Ber}(0.5)$.  If $\delta^{(k+1)}_j =0$, and $\iota=1$, with probability  $\min(1,A_j)/2$ change $\delta^{(k+1)}_j$ to $\iota$. If $\delta^{(k+1)}_j=1$, and $\iota=0$, with probability $\min(1,A_j^{-1})/2$, change $\delta^{(k+1)}_j$ to $\iota$; where
\begin{equation}\label{def:Aj}
A_j  \eqdef \frac{\mathsf{q}}{1-\mathsf{q}} \sqrt{\frac{\rho_1}{\rho_0}} e^{-\left(\rho_1-\rho_0\right)\frac{\bar\theta_j^2}{2}} e^{\ell(\bar\theta_\delta^{(j,1)};z) - \ell(\bar\theta_\delta^{(j,0)};z)},\end{equation}
where $\bar\theta_\delta^{(j,1)},\bar\theta_\delta^{(j,0)}\in\rset^p$ are defined as $(\bar\theta_\delta^{(j,1)})_k = (\bar\theta_\delta^{(j,0)})_k =(\bar\theta_\delta)_k$, for all $k\neq j$, and $(\bar\theta_\delta^{(j,1)})_j = \bar\theta_j$, $(\bar\theta_\delta^{(j,0)})_j = 0$.
\end{description}
\vspace{-0.6cm}
\begin{flushright}
$\square$
\end{flushright}

\end{algorithm}
\medskip

We have left unspecified the MCMC kernel on $\rset^{\|\delta\|_0}$ used in STEP 1, since it can be set up in many ways.  Let us call $C_1(\delta^{(k)})$  the computational cost of that part of STEP 1, and let $C_2(\delta)$ denote the cost of computing the quasi-likelihood $\ell(\theta_\delta;z)$ which is the dominant term  in (\ref{def:Aj}). Then  as $p$ grows, the total per-iteration cost of Algorithm \ref{algo:basic} is of order
\[ O\left(C_1(\delta^{(k)} ) +  pC_2(\delta^{(k)})\right).\]
 Since Theorem \ref{thm:0} implies that a typical  draw $\delta^{(k)}$ from the quasi-posterior distribution is sparse and satisfies $\|\delta^{(k)}\|_0 =O(s_\star)$, we can conclude that  the per-iteration cost of the algorithm  is accordingly reduced  in problems where the sparsity of $\delta$ reduces the cost of the MCMC update in STEP 1, and the cost of computing the sparsified pseudo-likelihood $\ell(\theta_\delta;z)$. For instance, in a linear regression model (see Algorithm \ref{algo:gibbs} in Appendix \ref{append:algo} for a detailed presentation), if the Gram matrix $X'X$ is pre-computed then $ C_1(\delta^{(k)} )  =O(\|\delta^{(k)}\|_0^3) = O(s_\star^3)$ (the cost of Cholesky decomposition), and $C_2(\delta^{(k)}) =O(\|\delta^{(k)}\|_0) = O(s_\star)$. As a result the per-iteration cost of Algorithm \ref{algo:gibbs} grows with $p$ as $O(s_\star^3 + s_\star p)=O(s_\star p)$, which is substantially faster than $O(\min(n,p)p^2)$ as needed by most MCMC algorithms for high-dimensional linear regression (\cite{bhattacharya:etal:16}). We refer the reader to Section \ref{sec:lin:reg} for a numerical illustration.

\section{Contraction rate and model selection consistency}\label{sec:contraction}
If in addition to the assumptions above, the restrictions of $\ell$ to the sparse subsets $\rset^p_\delta$  are  strongly concave then one can show that a draw $\theta$ from  $\Pi$ is typically close to $\theta_\star$. To elaborate on  this,  let $\bar s\geq s_\star$ be some arbitrary integer and set $\Delta_{\bar s}\eqdef\{\delta\in\Delta:\;\|\delta\|_0\leq \bar s\}$, and
\[\e_1(\bar s) \eqdef\e_0 \cap \left\{z\in\Zset:\; \L_{\theta_\star}(\theta;z) \leq -\frac{1}{2}\r(\|\theta-\theta_\star\|_2) ,\;\;\mbox{ for all } \delta\in\Delta_{\bar s},\;\theta\in\rset^p_{\delta}\right\},\]
for some rate function $\r$.  Hence $z\in\e_1(\bar s)$ implies that the function $u\mapsto \ell(u;z)$ behaves like a strongly concave function when restricted to $\rset^p_\delta$, for all $\delta\in\Delta_{\bar s}$, but with a general rate function $\r$. Here also, checking that $Z\in\e_1(\bar s)$ boils down to checking a strong restricted  concavity of $\ell$, which can be done using similar methods as in \cite{negahbanetal10}.  The use of a general rate function $\r$ allows to handle problems  that are not strongly convex in the usual sense  (as for instance with logistic regression). Our main result in this section states that  when $z\in\e_1(\bar s)$, we are automatically guaranteed  a minimum rate of contraction for  $\Pi$ given by
\begin{equation}\label{eq:rate}
\epsilon \eqdef \inf\left\{z>0:\; \r(x) -2(s_\star + \bar s)^{1/2} \bar \rho x\geq 0,\;\mbox{ for all } x\geq z\right\}.\end{equation}
 To gain some intuition on $\epsilon$, consider  a linear regression model where $\ell(\theta;z) = -\|z-X\theta\|_2^2/(2\sigma^2)$. Then  we have
 \[\L_{\theta_\star}(\theta;z) = -\frac{n}{2\sigma^2}(\theta-\theta_\star)'\left(\frac{X'X}{n}\right)(\theta-\theta_\star).\] 
If $\theta\in\rset^p_\delta$ for some $\delta\in\Delta_{\bar s}$, then $\L_{\theta_\star}(\theta;z)\leq -n\underline{v}(\bar s+s_\star)\|\theta-\theta_\star\|_2^2/(2\sigma^2)$,  where $\underline{v}(\bar s+s_\star)$ is the restricted smallest eigenvalue of $X'X/n$ over $(\bar s+s_\star)$-sparse vectors. Hence,  we can take the rate function $\r(x) = n\underline{v}(\bar s+s_\star)x^2/\sigma^2$, In that case the contraction rate in (\ref{eq:rate}) gives  $\epsilon = 2\sigma^2(\bar s+s_\star)^{1/2} \bar \rho/ (n\underline{v}(\bar s+s_\star))$. The final form of the rate depends on $\bar\rho$ (in H\ref{H1}) which is determined by the tail behavior of the quasi-score $\nabla\ell(\theta_\star;Z)$. In the sub-Gaussian case $\bar\rho\propto \sqrt{n\log(p)}$, and this gives $\epsilon\propto \sqrt{(\bar s+s_\star)\log(p)/n}$. We refer the reader to the proof of Corollary \ref{coro:lm} for more details.

We set 
\begin{equation}\label{def:set:B}
\cB \eqdef \bigcup_{\delta\in\Delta_{\bar s}}\; \{\delta\}\times \cB^{(\delta)},\;
\end{equation}
where 
\begin{equation} \cB^{(\delta)}\eqdef \left\{\theta\in\rset^p:\; \|\theta_\delta-\theta_\star\|_2 \leq C\epsilon,\;\|\theta-\theta_\delta\|_2\leq \sqrt{(1+C_1)\rho_0^{-1} p},\right\},\end{equation}
for some absolute constants $C,C_1\geq 3$, where $\epsilon$ is as defined in (\ref{eq:rate}).  Our next result says that if $(\delta,\theta)\sim\Pi(\cdot\vert Z)$ and $Z\in\e_1(\bar s)$, then with high probability we have $\theta\in\cB^{(\delta)}$ for some $\delta\in\Delta_{\bar s}$: $\theta_\delta$ is close to $\theta_\star$, and  $\theta-\theta_\delta$ is small.
 
\begin{theorem}\label{thm1}
Assume H\ref{H1}-H\ref{H2}. Let $\bar s\geq s_\star$ be some arbitrary integer, and take $\e\subseteq\e_1(\bar s)$. If 
\begin{equation}\label{tech:cond:thm1}
C\bar\rho(s_\star +\bar s)^{1/2}\epsilon  \geq 32\max\left[\bar s\log(p), \;(1+u)s_\star\log\left(p+\frac{p\bar\kappa}{\rho_1}\right)\right],\end{equation}
then for all $p$ large enough,
\begin{equation}\label{thm1:eq:main:bound}
\PE_\star\left[\textbf{1}_\e(Z)\Pi\left(\cB^c\vert Z\right)\right] \leq  \PE_\star\left[\textbf{1}_{\e}(Z)\Pi\left(\|\delta\|_0 > \bar s\;\vert Z\right)\right]  + 8e^{-\frac{C}{32}\bar\rho(s_\star+ \bar s)^{1/2}\epsilon}  + 2e^{-p}\end{equation}
where  $\cB^c \eqdef (\Delta\times\rset^p)\setminus \cB$.
\end{theorem}
\begin{proof}
See Section \ref{sec:proof:thm1}.
\end{proof}

\medskip
\begin{remark}
The result implies that for $j$ such that $\delta_j=0$, $|\theta_j| =O(\sqrt{\rho_0^{-1}})$ under $\Pi$. As a result we recommend scaling  $\rho_0^{-1}$ in practice as
\[\rho_0^{-1} = \frac{C_0}{n},\;\;\mbox{ or }\;\; \rho_0^{-1} = \frac{C_0}{p}.\]

When the posterior distribution is known to be sparse one can choose $\bar s$ appropriately to make the first term on the right hand side of (\ref{thm1:eq:main:bound}) small. For instance under the assumptions of Theorem \ref{thm:0}, we can take 
\[\bar s = s_\star\left(1 + \frac{2(1+c_0)}{u}\right)+  k.\]
If in addition  $\PP_\star(Z\notin\e_1(\bar s))\to 0$ as $p\to\infty$, we can deduce from (\ref{thm1:eq:main:bound}) that $\PE_\star[\Pi(\cB^c\vert Z)]\to 0$, as $p\to\infty$. If Theorem \ref{thm:0} does not apply, one  can  modify H\ref{H2} to impose the sparsity constraint $\|\delta\|_0\leq \bar s$ directly in the prior distribution. In this case  the first term on the right hand side of (\ref{thm1:eq:main:bound}) automatically vanishes. The main drawback in this approach is that an a priori knowledge of $\bar s\geq s_\star$ is needed in order to use the quasi-posterior distribution with a possible risk of misspecification.
\vspace{-0.6cm}
\begin{flushright}
$\square$
\end{flushright}
\end{remark}

We now show that when the non-zero components of $\theta_\star$ are sufficiently large, $\Pi$  achieves perfect model selection. 
Given $\delta\in\Delta_{\bar s}$ we define the function $\ell^{[\delta]}(\cdot;z):\;\rset^{\|\delta\|_0}\to\rset$ by $\ell^{[\delta]}(u;z)\eqdef \ell((u,0)_\delta;z)$. We then introduce the estimators
\begin{equation}\label{freq:est}
\hat\theta_\delta(z)\eqdef  \argmax_{u\in\rset^{\|\delta\|_0}}\; \ell^{[\delta]}(u;z),\;\;\;z\in\Zset.\end{equation}
When $\delta=\delta_\star$ we write $\hat\theta_\star(z)$.  At times, to shorten the notation we will omit the data $z$ and write $\hat\theta_\delta$ instead of $\hat\theta_\delta(z)$. Recall for $z\in\e_1(\bar s)$ the functions $\ell^{[\delta]}(\cdot;z)$ are strongly concave. Therefore for $z\in\e_1(\bar s)$,  the estimators $\hat\theta_\delta$ are well-defined for all $\delta\in\Delta_{\bar s}$. Omitting the data $z$, we will write $\mathcal{I}_{\delta}\in \rset^{\|\delta\|_0\times \|\delta\|_0}$ to denote the negative of the matrix of second derivatives of $u\mapsto \ell^{[\delta]}(u;z)$ evaluated at $\hat\theta_\delta(z)$. That is 
\[\mathcal{I}_\delta \eqdef -\nabla^{(2)}\ell^{[\delta]}(\hat\theta_\delta;z) \in\rset^{\|\delta\|_0\times \|\delta\|_0}.\]
Note that $\mathcal{I}_\delta$ is simply the sub-matrix of $\nabla^{(2)}\ell((\hat\theta_\delta,0)_\delta;z)$ obtained by taking the rows and columns for which $\delta_j=1$.  When $\delta=\delta_\star$, we will  write $\mathcal{I}$ instead of $\mathcal{I}_{\delta_\star}$. 
For $a>0$, and $\delta\in\Delta\setminus \{0\}$, we define
\[\varpi(\delta,a;z)\eqdef \;\;\sup_{u\in\rset^{\|\delta\|_0}:\;\|u-\hat\theta_\delta\|_2\leq a}\;\;\max_{1\leq i,j,k\leq \|\delta\|_0} \left|\frac{\partial^3\ell^{[\delta]}(u;z)}{\partial u_i\partial u_j\partial u_k}\right|.\]
$\varpi(\delta,a;z)$ measures the deviation of the log-quasi-likelihood from its quadratic approximation around $\hat\theta_\delta$. 
With the rate $\epsilon$ as  in (\ref{eq:rate}), we will make the assumption that 
\begin{equation}\label{eq:strong:signal}
\min_{j:\;\delta_{\star j}=1} |\theta_{\star j}| >C\epsilon.\end{equation}
Clearly this assumption is unverifiable in practice since $\theta_\star$ is typically not known. However a strong signal assumption such as (\ref{eq:strong:signal}) is needed in one form or the other for exact model selection (\cite{narisetti:he:14,castillo:etal:14,yang:etal:15}).  Furthermore as we show in Section \ref{sec:lin:reg}, in specific models  (\ref{eq:strong:signal}) translates into a condition on the sample size $n$, which in some cases can help the user evaluates in practice whether (\ref{eq:strong:signal}) seems reasonable or not. An understanding of the behavior of $\Pi$ when (\ref{eq:strong:signal}) does not hold remains an interesting problem for future research.

One can readily observe that when (\ref{eq:strong:signal}) holds,  then the set $\cB^{(\delta)}$ introduced above is necessarily empty when $\delta$ does not contain the true model $\delta_\star$. In other words, when  (\ref{eq:strong:signal}) holds, the set $\cB$ defined in (\ref{def:set:B}) can be written as 
\[\cB= \bigcup_{\delta\in\A_{\bar s}} \{\delta\}\times \cB^{(\delta)},\]
where
\[ \A_{\bar s}\eqdef\{\delta\in\Delta:\;\|\delta\|_0\leq \bar s,\;\mbox{ and }\delta\supseteq\delta_\star\},\]
and we recall that the notation $\delta\supseteq\delta'$ means that $\delta_j=1$ whenever $\delta_j'=1$ for all $j$. More generally, for $j\geq 0$, we set
\[ \A_{s_\star +j} \eqdef\{\delta\in\Delta:\;\|\delta\|_0\leq s_\star +j,\;\delta\supseteq\delta_\star\},\;\mbox{ and }\;\; \cB_j= \bigcup_{\delta\in\A_{s_\star +j}} \{\delta\}\times \cB^{(\delta)}.\]
In particular $\cB_0 =\{\delta_\star\}\times \cB^{(\delta_\star)}$, and $(\delta,\theta)\in\cB_j$  implies that $\delta$ has at most $j$ false-positive (and no false-negative). We set 
\begin{multline*}
\e_{2}(\bar s) \eqdef \e_1(\bar s) \cap \bigcap_{j=1}^{\bar s-s_\star} \left\{z\in\Zset:\;\max_{\delta\in\A_{\bar s}:\;\|\delta\|_0=s_\star + j}\ell^{[\delta]}(\hat\theta_\delta;z) -\ell^{[\delta_\star]}(\hat\theta_\star;z)\leq \frac{ju}{2}\log(p) \right\},\end{multline*}
 which imposes a growth condition on the log-quasi-likelihood ratios of sparse sub-models. 
 
\begin{theorem}\label{thm:sel}
Assume H\ref{H1}-H\ref{H2}, and (\ref{eq:strong:signal}). Let $\bar s\geq s_\star$ be some arbitrary integer, and take $\e\subseteq\e_{2}(\bar s)$. For some constant $\underline{\kappa}>0$, suppose that for all $z\in\e$,
\begin{equation}\label{def:kappa:min}
\min_{\delta\in\A_{\bar s}}\; \inf_{u\in\rset^{\|\delta\|_0}:\;\|u-\hat\theta_\delta\|_2\leq 2\epsilon}\;\inf\left\{ \frac{v'\left(-\nabla^{(2)}\ell^{[\delta]}(u;z)\right)v}{\|v\|_2^2},\; v\in\rset^{\|\delta\|_0},\;v\neq 0\right\}\geq \underline{\kappa},\end{equation}
and 
\begin{equation}\label{def:kappa:bar:bis}
\max_{\delta\in\A_{\bar s}}\; \sup_{u\in\rset^{\|\delta\|_0}}\;\sup\left\{ \frac{v'\left(-\nabla^{(2)}\ell^{[\delta]}(u;z)\right)v}{\|v\|_2^2},\; v\in\rset^{\|\delta\|_0},\;v\neq 0\right\}\leq \bar \kappa,\end{equation}
where $\bar\kappa$ is as in H\ref{H1}.  Then it holds that for any $j\geq 1$
\begin{multline}\label{maineq:thm:sel}
\textbf{1}_{\e}(z)\left(1 - \Pi\left(\cB_j\vert z\right)\right)\\
\leq  8 e^{C_0(\rho_1\|\theta_\star\|_\infty\bar s^{1/2}\epsilon + \mathsf{a}_2\bar s^{3/2}\epsilon^3)} e^{\frac{2\mathsf{a}_2\bar s^{3}\epsilon}{\underline{\kappa}}} \left(\sqrt{\frac{\rho_1}{\underline{\kappa}}} \frac{1}{p^{\frac{u}{2}}}\right)^{j+1} +\textbf{1}_{\e}(z) \Pi(\cB^c\vert z),\end{multline}
provided  that $\underline{\kappa} p^u\geq 4\rho_1 $, and $(C-1)\epsilon \underline{\kappa}^{1/2} \geq 2(s_\star^{1/2}+1)$, where
$\mathsf{a}_2\eqdef \max_{\delta\in\A_{\bar s}}\;\varpi(\delta,(C+1)\epsilon;z)$, and $C_0$ some absolute constant.
\end{theorem}
\begin{proof}
See Section \ref{sec:proof:thm:sel}.
\end{proof}

We note that $\cB_0 = \{\delta_\star\} \times \cB^{(\delta_\star)}\subset \{\delta_\star\}\times\rset^p$. Hence by choosing $j=0$,  (\ref{maineq:thm:sel}) provides a lower bound on the probability of perfect model selection $\Pi(\delta_\star\vert z)$. 

\begin{remark}
The left hand sides of (\ref{def:kappa:min}) and (\ref{def:kappa:bar:bis}) are restricted eigenvalues. We note that the infimum on $u$ in (\ref{def:kappa:min}) is taken over a small neighborhood of $\hat\theta_\delta$, which is an important detail that facilitates the application of the result. The main challenge  in using this result is bounding the probability of the event $\e_{2}(\bar s)$ (which deals with the behavior of the quasi-likelihood ratio statistics). For linear regression problems, this boils down to deviation bounds for projected Gaussian distributions as we show in Section \ref{sec:lin:reg}. An extension to generalized linear models via the Hanson-Wright inequality seems plausible although not pursed here.
\vspace{-0.4cm}
\begin{flushright}
$\square$
\end{flushright}
\end{remark}

\section{Posterior approximations}\label{sec:post:approx}
We show here that a  Bernstein-von Mises approximation holds in the  KL-divergence sense. 
We consider the  distribution 
\begin{equation}\label{Pi:inf}
\Pi^{(\infty)}_\star(\delta,\rmd\theta\vert z)  \propto \textbf{1}_{\delta_\star}(\delta) e^{-\frac{1}{2}([\theta]_{\delta_\star}-\hat\theta_\star)'\mathcal{I}([\theta]_{\delta_\star}-\hat\theta_\star) -\frac{\rho_0}{2}\|\theta-\theta_{\delta_\star}\|_2^2}  \rmd\theta,
\end{equation}
which puts probability one on $\delta_\star$, and draws independently $[\theta]_{\delta_\star}\sim \textbf{N}(\hat\theta_\star,\mathcal{I}^{-1})$, and $[\theta]_{\delta_\star^c}\stackrel{i.i.d.}{\sim}\textbf{N}(0,\rho_0^{-1})$. Our version of the Bernstein-von Mises theorem  says that $\Pi$ behaves like  $\Pi_\star^{(\infty)}$.  If $\mu,\nu$ are two probability measures on some measurable space we define the Kulback-Leibler divergence (KL-divergence) of $\mu$ respect to $\nu$ as
\[\KL{\mu}{\nu} \eqdef \left\{ \begin{array}{ll}\int \log\left(\frac{\rmd\mu}{\rmd\nu}\right)\rmd\mu, & \mbox{ if } \mu\ll\nu \\ +\infty & \mbox{ otherwise}.\end{array}\right.\]
A Bernstein-von Mises approximation in the KL-divergence sense -- unlike the analogous result in the total variation metric --  requires a control of the tails of the log-quasi-likelihood. To limit the technical details we will focus  on the case where those tails are quadratic.

\begin{theorem}\label{thm:KL}
Assume H\ref{H1}-H\ref{H2}.  For some integer $\bar s\geq s_\star$, and some constant $\underline{\kappa}>0$, let $\e$ be some measurable subset of $\Zset$ such that  for all $z\in\e$, $\Pi(\delta_\star\vert z)\geq 1/2$,  (\ref{def:kappa:bar:bis}) holds with $\bar\kappa$ as in H\ref{H1}, and 
\begin{equation}\label{def:kappa:min:KL}
\min_{\delta\in\A_{\bar s}}\; \inf_{u\in\rset^{\|\delta\|_0}}\;\inf\left\{ \frac{v'\left(-\nabla^{(2)}\ell^{[\delta]}(u;z)\right)v}{\|v\|_2^2},\; v\in\rset^{\|\delta\|_0},\;v\neq 0\right\}\geq \underline{\kappa}.\end{equation}
Then there exists an absolute constants $C_0$ such that 
\begin{multline}\label{eq:main:bvm}
\textbf{1}_{\e}(z) \KL{\Pi_\star^{(\infty)}}{\Pi}  \leq   C_0 \left(\rho_1\bar s^{1/2}\epsilon + \mathsf{a}_2\bar s^{3/2}\epsilon^3\right)
+\frac{3\rho_1^2(\epsilon + \|\theta_\star\|_2)^2}{2(\rho_1+\bar \kappa)} \\
+C_0(\rho_1 +\bar\kappa)\epsilon^2\left(\frac{\bar\kappa}{\underline{\kappa}}\right)^{\frac{s_\star}{2}} e^{-\frac{(C-1)^2\epsilon^2\underline{\kappa}}{32}} + C_0(\rho_1 +\bar\kappa)e^{-p} +2\textbf{1}_{\e}(z) (1-\Pi(\delta_\star \vert z)), \end{multline}
provided that  $\underline{\kappa}(C-1)\epsilon \geq 4\max(\sqrt{s_\star\underline{\kappa}},\rho_1(\epsilon+ s_\star^{1/2}\|\theta_\star\|_\infty))$, where  $C$ is as in Theorem \ref{thm1}.
\end{theorem}
\begin{proof}
See Section \ref{sec:proof:thm:KL}.
\end{proof}

\begin{remark}
The upper bound in (\ref{eq:main:bvm}) implies an upper bound on the total variation distance between $\Pi$ and $\Pi_\star^{(\infty)}$ via Pinsker's inequality (see e.g. \cite{boucheron:eta:13}~Theorem 4.19).  The leading term in (\ref{eq:main:bvm}) is typically $C_0(\rho\bar s^{1/2}\epsilon + \mathsf{a}_2\bar s^{3/2}\epsilon^3)$ which gives a non-trivial convergence rate in the Bernstein-von Mises approximation. 
\vspace{-0.6cm}
\begin{flushright}
$\square$
\end{flushright}
\end{remark}


\subsection{Implications for variational approximations}
When dealing with very large scale problems, practitioners often turn to variational approximation methods to obtain fast approximations of $\Pi$. We explore some implications of  Theorem \ref{thm:KL}   on the behavior of variational approximation methods in the high-dimensional setting. Let $\S\in\{0,1\}^{p\times p}$ be a symmetric matrix, and let $\M^+_p(\S)$ be the set of all $p\times p$ symmetric positive definite (spd) matrices with sparsity pattern $\S$ (that is $M\in\M_p^+(\S)$ means that $\S\cdot M= M$, where $A\cdot B$ is the component-wise product of $A,B$).  We assume in addition that $\S$ is such that if $M$ is spd then $\S\cdot M$ is also spd. We consider the  family $\mathcal{Q}\eqdef\{Q_\Psi,\;\Psi\}$ of probability measures on $\Delta\times\rset^p$, indexed by $\Psi = (q,\mu,C)\in (0,1)^p\times \rset^p \times \M_p^+(\S)$, where
\begin{equation}\label{def:fam:Q}
Q_\Psi(\rmd \delta,\rmd \theta) = \prod_{j=1}^p \textbf{Ber}(q_j)(\rmd\delta_j)   \textbf{N}_p(\mu,C)(\theta)\rmd\theta,\end{equation}
 In these definitions $\textbf{Ber}(\alpha)(\rmd x)$ is the probability measure on $\{0,1\}$ that assigns probability $\alpha$ to $1$, and $\textbf{N}_p(m,V)(\cdot)$ is the density of $p$-dimensional Gaussian distribution $\textbf{N}_p(m,V)$. Let $Q$ be the minimizer of the KL-divergence $\KL{Q}{\Pi}$ over the family $\mathcal{Q}$:
\begin{equation}\label{Q:star}
Q \eqdef \argmin_{Q\in\mathcal{Q}}\;  \KL{Q}{\Pi}.\end{equation}
We call $Q$ the variational approximation of $\Pi$ over the family $\mathcal{Q}$. Although not shown in the notation, $Q$ depends on the data $z$. We will consider the following examples.

\begin{example}[Skinny variational approximation]
If $\S=I_p$, then $Q$ corresponds to a  mean-field variational approximation of $\Pi$. We will refer to this approximation below as the skinny variational approximation (skinny-VA) of $\Pi$.
\end{example}

\begin{example}[full and midsize variational approximations]
If $\S$ is taken as the full matrix with all entries equal to $1$, we will refer to $Q$ as the full variational approximation (full-VA) of $\Pi$.  More generally let $\delta^{(\textsf{i})}$ be some element of $\{0,1\}^p$ that we call a template. Ideally we want $\delta^{(i)}$ to be sparse and to contain the true model, but this needs not be assumed. We then define $\S$ as follows:  $\S_{ij}=1$ if $i=j$, and $\S_{ij}=\delta^{(\textsf{i})}_i\delta^{(\textsf{i})}_j$ if $i\neq j$. If $\delta^{(\textsf{i})}$ is sparse, matrices $M\in \M_p^+(\S)$ are also sparse. In that case we call $Q$ a midsize variational approximation (midsize-VA) of $\Pi$. We note that we also recover the skinny-VA by taking $\delta^{(\textsf{i})}= {\bf 0}_p$, and we recover the full-VA by taking $\delta^{(\textsf{i})}$ as the vector with components equal to $1$.
\end{example}

The appeal of variational approximation methods is that  $Q$ can be approximated using algorithms that are order of magnitude faster than MCMC. We note however that the optimization problem in (\ref{Q:star}) is non-convex in general. Hence,  convergence guarantees for these algorithms are difficult to establish.  We do not address these issues here. Instead we would like to explore the behavior of $Q$ in view of Theorem \ref{thm:KL}.  Let us rewrite the distribution $\Pi^{(\infty)}_\star$ in (\ref{Pi:inf}) as
\begin{equation*}
\Pi^{(\infty)}_\star(\delta,\rmd\theta\vert z)  \propto \textbf{1}_{\delta_\star}(\delta) e^{-\frac{1}{2}(\theta - \hat\theta_\star)' \bar{\mathcal{I}}_\gamma (\theta-\hat\theta_\star)}  \rmd\theta,
\end{equation*}
where we abuse notation to write $(\hat\theta_\star,0)_{\delta_\star}$ as $\hat\theta_\star$, and $\bar{ \mathcal{I}}_\gamma\in\rset^{p\times p}$ is such that $[\bar{ \mathcal{I}}_\gamma]_{\delta_\star,\delta_\star} = \mathcal{I}$, $ [\bar{ \mathcal{I}}_\gamma]_{\delta_\star,\delta_\star^c} = [\bar{ \mathcal{I}}_\gamma]_{\delta_\star^c,\delta_\star}' = 0$, and $[\bar{ \mathcal{I}}_\gamma]_{\delta_\star^c,\delta_\star^c} = (1/\gamma)I_{p-s_\star}$. Then we set 
\begin{equation}\label{Pi:inf:tilde}
\tilde\Pi^{(\infty)}_\star(\delta,\rmd\theta\vert z)  \propto \textbf{1}_{\delta_\star}(\delta) e^{-\frac{1}{2}(\theta - \hat\theta_\star)' \left(\S\cdot \bar{\mathcal{I}}_\gamma\right) (\theta-\hat\theta_\star)}  \rmd\theta.
\end{equation}
 The total variation metric between two probability measure is defined as
\[\|\mu-\nu\|_{\tv} \eqdef \sup_{A\;\mbox{ meas.}}\; \left(\mu(A) -\nu(A)\right).\]

\begin{theorem}\label{thm:VA}
Assume H\ref{H1}-H\ref{H2}. For all $z\in\Zset$ such that $\Pi(\cdot\vert z)$ and $\Pi_\star^{(\infty)}(\cdot\vert z)$ are well-defined we have 
\begin{equation}\label{thm:VA:eq}
\|Q - \tilde\Pi_\star^{(\infty)}\|_\tv^2 \leq  8\zeta + 16 \int_{\delta_\star\times\rset^p} \log\left(\frac{\rmd \Pi^{(\infty)}_\star}{\rmd \Pi}\right)\rmd \tilde\Pi_\star^{(\infty)}, \end{equation}
where
 \begin{equation}\label{def:zeta1:zeta2}
 \zeta =\log\left(\frac{\det(\bar{\mathcal{I}}_\gamma)}{\det(\S\cdot\bar{\mathcal{I}}_\gamma)}\right)  + \textsf{Tr}\left(\bar{\mathcal{I}}_\gamma^{-1}(\S\cdot \bar{\mathcal{I}}_\gamma)\right) - p.
 \end{equation}
\end{theorem}
\begin{proof}
See Section \ref{sec:proof:thm:VA}.
\end{proof}

\begin{remark}
As we show below in the proof of Theorem \ref{thm:KL}, the integral on the right size of (\ref{thm:VA:eq}) behaves like $\KL{\Pi_\star^{(\infty)}}{\Pi}$, which can be shown to vanish using the Bernstein-von Mises theorem (Theorem \ref{thm:KL}) under appropriate regularity conditions. In this case, whether $Q$ behaves like $\tilde\Pi_\star^{(\infty)}$ can be deduced from the behavior of $\zeta$, a term that is easier to analyze.  For instance for the full-VA $\zeta=0$. More generally for any midsize-VA such that $\delta^{(\textsf{i})}\supseteq\delta_\star$, we have $\zeta=0$. In the case of the skinny-VA (mean field variational approximation), $\zeta>0$ in general, but  $\zeta=o(1)$  when the off-diagonal elements of the information matrix $\mathcal{I}$ are $o(1)$.
\vspace{-0.6cm}
\begin{flushright}
$\square$
\end{flushright}
\end{remark}

\begin{remark}
Theorem \ref{thm:VA}  gives an approximation (in total variation sense) of the variational approximation. To the exception of (\cite{wang:blei:18}) most of the theoretical work on variational approximation methods have focused on concentration: whether the variational approximation put most of its probability mass around the true value (see e.g. \cite{alquier:ridgeway:17} for some recent results, and \cite{wang:blei:18} for an overview of the literature), without addressing whether other aspects of the distribution are recovered well. One important limitation of \cite{wang:blei:18} which makes the extension of their approach to high-dimension problematic is  their reliance on a) local asymptotic normality assumptions, and b) the assumption that the variational family can be viewed as a re-scaled version of some sample-size independent family.
\vspace{-0.5cm}
\begin{flushright}
$\square$
\end{flushright}
\end{remark}

\section{Examples}\label{sec:ex}
\subsection{Gaussian graphical models via Linear regressions}
\label{sec:lin:reg}
Fitting large sparse graphical models in the Bayesian framework is computationally challenging (\cite{dobra:11a,dobra:11b,khondkeretal13,peterson:etal:14,banerjee:ghosal13}). A quasi-Bayesian approach based on the neighborhood selection of \cite{meinshausen06}  offers a simple, yet effective alternative. The idea was explored in \cite{atchade:15:c} using point-mass spike and slab priors. The approach  proposed in this paper yields a highly scalable quasi-posterior distribution  with equally strong theoretical backing. We make   the following data generating assumption.

\begin{assumptionB}\label{H:lin:mod}
$Z\in\rset^{n\times (p+1)}$ is a random matrix with i.i.d. rows from $\textbf{N}_{p+1}(0,\vartheta_\star^{-1})$ for some positive definite matrix $\vartheta_\star$.  We set $\Sigma\eqdef \vartheta^{-1}_\star$ and also assume that as $p\to\infty$,
\begin{equation}\label{lm:eig:assump}
\frac{1}{\lambda_{\textsf{min}}(\Sigma)} + \lambda_{\textsf{max}}(\Sigma) = O(1).\end{equation}
\end{assumptionB}
\medskip

\begin{remark}
The assumption in (\ref{lm:eig:assump}) restricts our focus to  problems that in some sense do not  become intrinsically harder as $p$ increases. It	 can be relaxed by tracking more carefully the constants  in the proofs.
\vspace{-0.6cm}
\begin{flushright}
$\square$
\end{flushright}
\end{remark}

Given the data matrix $Z\in\rset^{n\times (p+1)}$, we wish to estimate the precision matrix $\vartheta_\star$. Instead of a full likelihood approach (explored in the references cited above), we consider a pseudo-likelihood approach that estimates each column of $\vartheta_\star$ separately. Given $1\leq j\leq p+1$, we partition the data matrix $Z$ as $Z = [Y^{(j)},X^{(j)}]$, where $Y^{(j)}\in\rset^n$ denotes the $j$-th column of $Z$,  and $X^{(j)}\in\rset^{n\times p}$ collects the remaining columns.  In that case the conditional distribution of $Y^{(j)}$ given $X^{(j)}$ is 
\[\textbf{N}_n\left(X^{(j)}\theta_\star^{(j)},\frac{1}{[\vartheta_{\star}]_{jj}}I_n\right),\] 
where $\theta_\star^{(j)} \eqdef (-1/[\vartheta_{\star}]_{jj})[\vartheta_{\star}]_{-j,j}\in\rset^p$.  Therefore, for some user-defined parameters $\sigma_j>0$, $\rho_{0,j}>0$, and $\rho_{1,j}$ the quasi-posterior distribution on $\Delta\times \rset^p$ given by
\begin{multline}\label{post:linmod}
\Pi^{(j)}(\delta,\rmd\theta\vert Z)
\propto  \\
e^{-\frac{1}{2\sigma_j^2}\|Y^{(j)}-X^{(j)}\theta_\delta\|_2^2} \omega(\delta) \left(\frac{\rho_{1,j}}{2\pi}\right)^{\frac{\|\delta\|_0}{2}} \left(\frac{\rho_{0,j}}{2\pi}\right)^{\frac{p-\|\delta\|_0}{2}} e^{-\frac{\rho_{1,j}}{2}\|\theta_\delta\|_2^2}e^{-\frac{\rho_{0,j}}{2}\|\theta-\theta_\delta\|_2^2} \rmd \theta,
\end{multline}
can be used to estimate $\theta_\star^{(j)}$, and hence the $j$-th column of $\vartheta_\star$, if an estimate of $[\vartheta_{\star}]_{jj}$ is available\footnote{A full Bayesian approach can be adopted to estimate both $\theta_\star^{(j)}$ and $[\vartheta_{\star}]_{jj}$. But for simplicity's sake we will not pursue this here}. This is basically the quasi-Bayesian analog of the neighborhood selection of \cite{meinshausen06}. The same procedure can be repeated -- possibly in parallel --  to recover the entire matrix $\vartheta_\star$. We use the theory of Section \ref{sec:sparsity}-\ref{sec:post:approx} to describe the behavior of this approach to infer $\vartheta_\star$.  We focus on the case where $n=o(p)$, and we recall that $C_0$ is an absolute constant whose value may be different from one expression to the other. Let $\Pi_\star^{(j,\infty)}$ be the corresponding limiting distribution of $\Pi^{(j)}$ as defined in (\ref{Pi:inf}), and let $\tilde\Pi_\star^{(j,\infty)}$ be the corresponding approximation given in (\ref{Pi:inf:tilde}). In this particular case, $\Pi_\star^{(j,\infty)}$ is the probability measure on $\Delta\times\rset^p$ that  puts probability one on $\delta_\star^{(j)}$ (the support of $\theta_\star^{(j)}$),  draws $[\theta]_{\delta_\star^{(j)}}\sim \textbf{N}\left(\hat\theta_\star^{(j)},\sigma_j^2(X_{\delta_\star^{(j)}}'X_{\delta_\star^{(j)}})^{-1}\right)$, and draws independently  all other components i.i.d. from $\textbf{N}(0,\rho_0^{-1})$, where $\hat\theta_\star^{(j)}$ is the OLS estimator $(X_{\delta_\star^{(j)}}X_{\delta_\star^{(j)}})^{-1}X_{\delta_\star^{(j)}}'Y^{(j)}$. We set $s_\star^{(j)} \eqdef \|\theta_\star^{(j)}\|_0$. Let $Q^{(j)}$ denote the variational approximation of $\Pi^{(j)}$ based on the family (\ref{def:fam:Q}) with sparsity pattern $\S^{(j)}$, and let $\zeta_j$ denote the corresponding term in (\ref{def:zeta1:zeta2}).

\begin{corollary}\label{coro:lm}
Assume H\ref{H2}, B\ref{H:lin:mod}, and  suppose that $s_\star^{(j)}>0$, $\max_j\|\theta_\star^{(j)}\|_\infty=O(1)$, and $\max_j s_\star^{(j)}=O(\log(p))$ as $p\to\infty$. Suppose also that $u>2$, and  $u\sigma_j^2[\vartheta_{\star}]_{jj}\geq 16$. Choose  the prior parameter $\rho_{1,j}$  as
\[ \rho_{1,j} = \sqrt{\frac{\log(p)}{n}}.\]
Set
\[\bar s^{(j)}  \eqdef s_\star^{(j)}\left( 1 + \frac{6}{u}\right) + \frac{u}{4},\;\;\;\epsilon^{(j)} \eqdef  C_0\sqrt{\frac{(\bar s^{(j)}+s_\star^{(j)})\log(p)}{[\vartheta_\star]_{jj}\;n}},\;\;\mbox{ and }\;\;\; \bar s = \max_j \bar s^{(j)}.\]
Suppose that the sample size  $n$ satisfies $n=o(p)$, as $p\to\infty$, and
\[
 n\geq C_0 \bar s \log(p),\]
and the strong signal assumption 
\begin{equation}\label{tech:cond:lm}
\min_{k:\;|\theta^{(j)}_{\star,k}| > 0} |\theta^{(j)}_{\star,k}| > C_0 \epsilon^{(j)}\end{equation}
holds. Then there exists a measurable set $\G$ with $\PP_\star(Z\notin\G)\to 0$ as $p\to\infty$ such that  
 \begin{equation}\label{eq:thm:KL:lm}
\PE_\star\left[\textbf{1}_{\G}(Z) \max_{1\leq j\leq p+1}\KL{\Pi_\star^{(j,\infty)}}{\Pi^{(j)}}\right] \leq \frac{C_0\max_j(\bar s^{(j)} + s_\star^{(j)})}{\min_j [\vartheta_\star]_{jj}} \frac{\log(p)}{n} + \frac{C_0}{p^{1\wedge\left(\frac{u}{2}-1\right)}}.\end{equation}
Furthermore the  variational approximation $Q^{(j)}$ satisfies
\begin{multline}
\PE_\star\left[\textbf{1}_{\G}(Z) \max_{1\leq j\leq p+1} \|Q^{(j)} - \tilde\Pi_\star^{(j,\infty)}\|_\tv^2\right] \leq 8\PE_\star\left[\textbf{1}_{\G}(Z)\max_{1\leq j\leq p+1} \zeta^{(j)} \right]  \\
+ \frac{C_0\max_j(\bar s^{(j)} + s_\star^{(j)})}{\min_j [\vartheta_\star]_{jj}} \frac{\log(p)}{n} + \frac{C_0}{p^{1\wedge\left(\frac{u}{2}-1\right)}}.\end{multline}
\end{corollary}
\begin{proof}
See Section \ref{proof:coro:lm}. 
\end{proof}

\begin{remark}
\begin{enumerate}
\item We have focused in the Corollary on the Bernstein-von Mises approximation and the behavior of the VA approximation. Other results, and generally more precise results are given in the proof. In particular we show that the rate of contraction of $\Pi^{(j)}$ is $\epsilon^{(j)}$, and that $\Pi^{(j)}$ achieves perfect model selection.
\item One cannot easily remove the indicator $\textbf{1}_\G$ from (\ref{eq:thm:KL:lm}). However by Pinsker's inequality we get
\begin{multline*}
2\PE_\star\left[\max_{1\leq j\leq p+1} \|\Pi_\star^{(j,\infty)} - \Pi^{(j)}\|_\tv^2\right] \leq 2\PP_\star[Z\notin\G] \\
+ \frac{C_0\max_j(\bar s^{(j)} + s_\star^{(j)})}{\min_j [\vartheta_\star]_{jj}} \frac{\log(p)}{n} + \frac{C_0}{p^{1\wedge\left(\frac{u}{2}-1\right)}}.\end{multline*}
\item  If the variational approximation $Q^{(j)}$ is constructed from some template $\delta^{(\textsf{i},j)}$, then the remainder $\zeta^{(j)}$ is zero if $\delta^{(\textsf{i},j)}\supseteq\delta_\star^{(j)}$. When this is the case we also have $\tilde\Pi_\star^{(j,\infty)} = \Pi_\star^{(j,\infty)}$. This holds for instance if $\delta^{(\textsf{i},j)}$ is the vector with all components equal to $1$ (full-VA). However the full-VA is expensive to compute. In fact, as we illustrate below the full-VA is more expensive to compute than direct MCMC sampling from $\Pi^{(j)}$.  However if $\delta^{(\textsf{i},j)}$ is sparse, for instance if $\delta^{(\textsf{i},j)}$ is the support of the lasso solution -- or some equally well-behaved frequentist estimate -- then the scaling of the  computational cost of $Q^{(j)}$ can be extremely favorable. Hence Corollary implies that extremely fast variational approximation of $\Pi^{(j)}$ with strong theoretical guarantees can be computed in large scale Gaussian graphical models.

\end{enumerate}
\vspace{-0.6cm}
\begin{flushright}
$\square$
\end{flushright}
\end{remark}
\medskip

\subsubsection{Numerical illustration}
\begin{wrapfigure}{r}{0.6\textwidth}
    \centering
    {\includegraphics[scale=0.2]{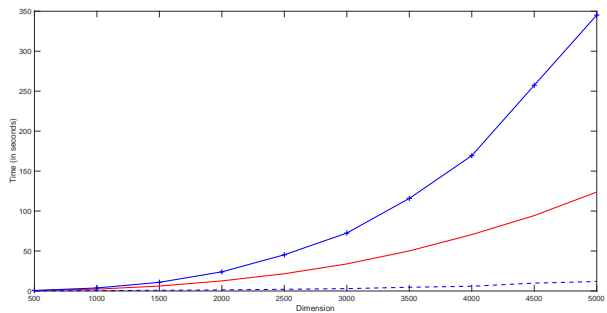}}
    \vspace{-0.8 cm}
 \caption{{\small Costs of: $p$ iterations of Metropolized Gibbs sampler (red solid line);  50 iterations of full-VA (blue+ line); and 50 iterations of midsize-VA with $\|\delta^{(\textsf{i})}\|_0=100$ (blue-dashed line), as   functions of the dimension $p$.}}
\label{fig:lm2}
 \end{wrapfigure}
We perform a simulation study to assess the behavior of the posterior distribution and its variational approximations as described in Corollary \ref{coro:lm}. For simplicity we focus on only one of the regression  problems. We set  $p=1000$, $n\in \{100,500\}$, and we generate $Z=[Y,X]\in\rset^{n\times (p+1)}$ as follows. We first generate the matrix $X$ by simulating the rows of $X$ independently from a Gaussian distribution with correlation $\psi^{|j-i|}$ between components $i$ and $j$, where  $\psi\in\{0,0.8\}$. When $\psi=0$, the resulting matrix $X$ has a low coherence, but the coherence increases when $\psi =0.8$. Using $X$, we general $Y=X\theta_\star +\epsilon/\vartheta_{\star,11}$, with $\vartheta_{\star,11}=1$ that we assume known. We build $\theta_\star$ with $s_\star=10$ non-zeros  components that we fill with draws from the uniform distribution $\pm\mathbf{U}(a,a+1)$, where $a = 4\sqrt{s_\star\log(p)/n}$.

We build $\Pi$ with $\sigma^2=1$, $u=2$, $\rho_1=\sqrt{\log(p)/n}$, and $\rho_0^{-1} = 1/(4n)$. We sample from $\Pi$ using Algorithm \ref{algo:gibbs}. We consider two variational approximation. The full-VA, and a mid-size VA with template $\delta^{(\textsf{i})}$ that contains the support of $\theta_\star$, and such that $\|\delta^{(\textsf{i})}\|_0=100$. We approximate the variational approximations by coordinate ascent variational inference (see e.g. \cite{blei:etal:17}). The details of these algorithms are given in Appendix \ref{append:algo}. We initialize all three algorithms from the lasso solution.  In Figure  \ref{fig:lm2} we plot the computational cost of the three algorithms as $p$ increases. It  shows that the full-VA is actually more expensive than the MCMC sampler. This is  due to the need to form the Cholesky decomposition of a large $p\times p$ matrix at each iteration of the full-VA. In contrast, and as explained in Section \ref{subsec:algo1} the per-iteration cost of Algorithm \ref{algo:gibbs} is of order $O(s_\star p)$. On the other hand, for $p=5,000$ the midsize VA is more than 10 times faster than the MCMC sampler.

Figure \ref{fig:lm} shows the (estimated) posterior distributions for the parameters $\theta_1,\theta_2$ and $\theta_3$ from one MCMC run of $5,000$ iterations and single CAVI-runs of 50 iterations. Here we are comparing the skinny-VA, and the midsize-VA with $\|\delta^{(\textsf{i})}\|_0=100$, for a template $\delta^{(\textsf{i})}$ that contains the support of $\theta_\star$. Since we are working in a high signal-to-noise  ratio setting the results are fairly consistent across replications.  The true signal $\theta_\star$ is such that $\theta_{\star,1}\neq 0$ and $\theta_{\star,2}\neq 0$ while $\theta_{\star,3}=0$. 
Figure \ref{fig:lm} shows that as $n$ increases both VA approximations approximate well the quasi-posterior distribution in the low coherence regime. However in  presence of correlation,  the skinny-VA systematically underestimates the marginal posterior variances when there is correlation between the relevant variables. However,  as suggested by Corollary \ref{coro:lm}, the midsize-VA approximates the whole distribution well.

\begin{figure}[]
\centering
\resizebox{\textwidth}{0.4\textheight}{\begin{tabular}{c}
Linear regression with low coherent design matrix. $p=1000$, $n=100$.\medskip\\
\includegraphics[width=0.8\textwidth,height=0.3\textheight,  trim = 1cm .1cm 0cm 1.cm, clip = true ]{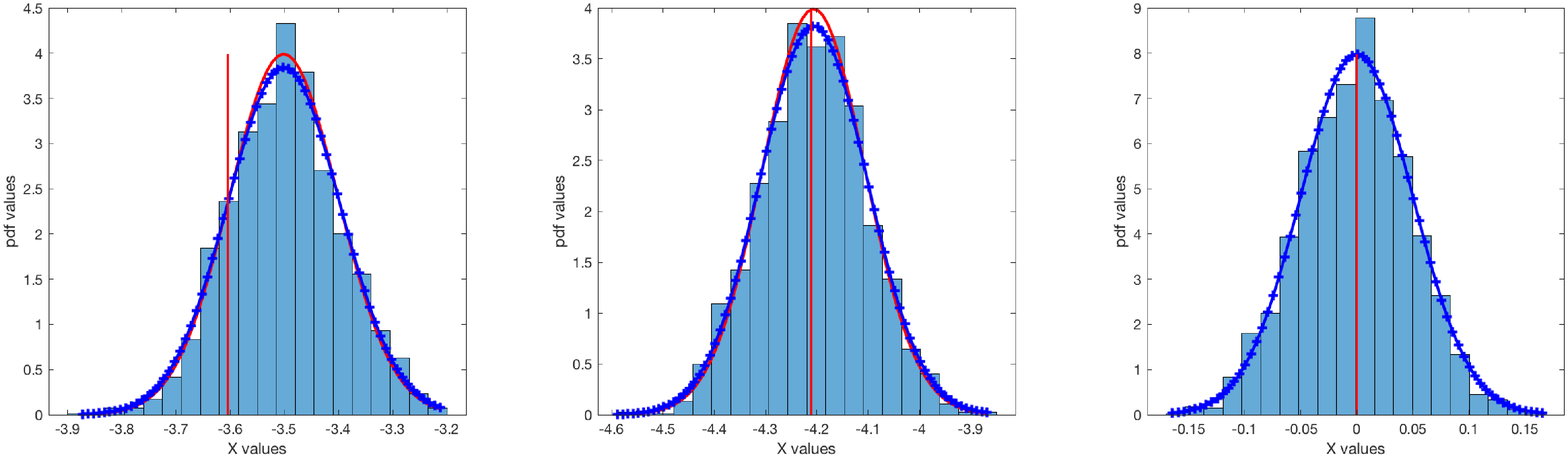}\\
Linear regression with low coherent design matrix. $p=1000$, $n=500$.\medskip\\
\includegraphics[width=0.8\textwidth,height=0.3\textheight ,  trim = 1cm .1cm 0cm 1.cm, clip = true]{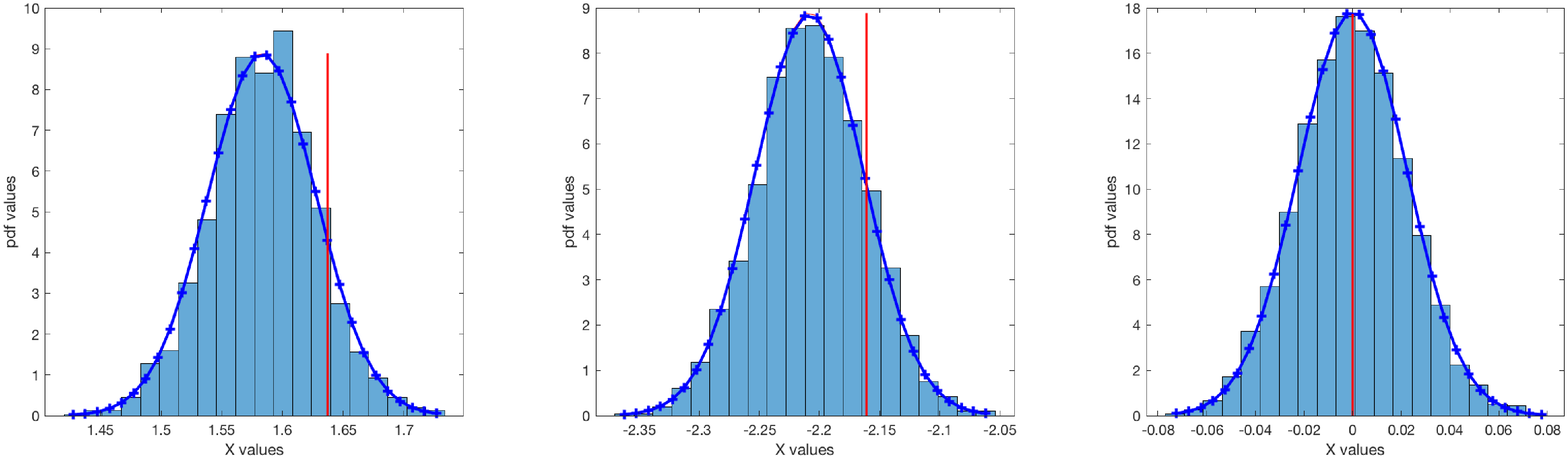}\\
Linear regression with high design matrix . $p=1000$, $n=100$.\medskip\\
\includegraphics[width=0.8\textwidth,height=0.3\textheight ,  trim = 1cm .1cm 0cm 1.cm, clip = true]{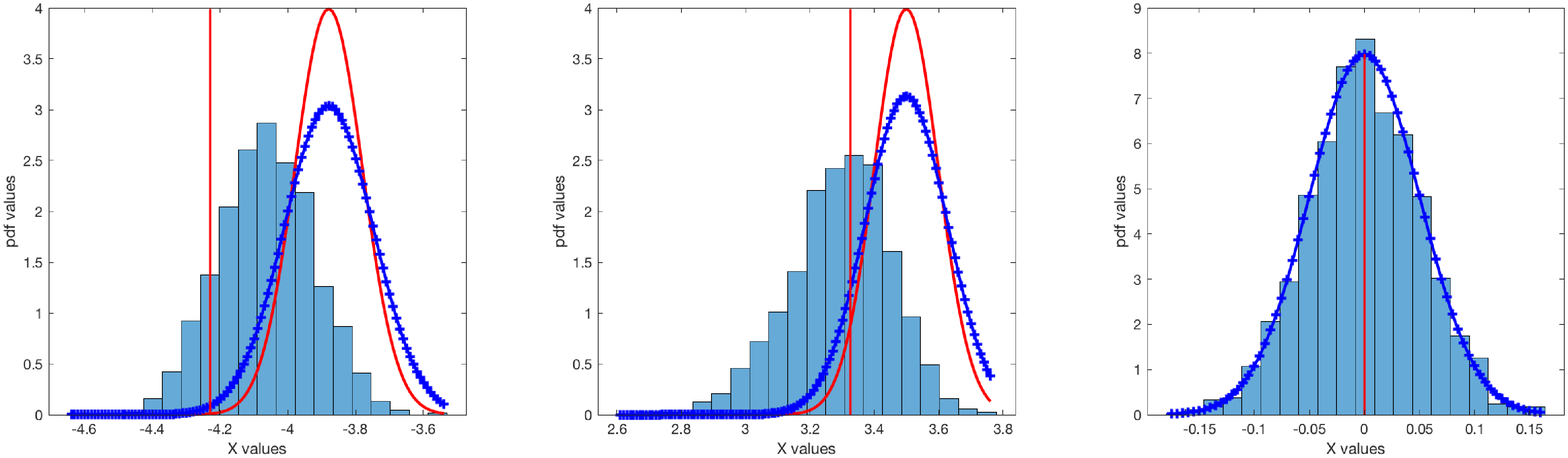}\\
Linear regression with high design matrix. $p=1000$, $n=500$.\medskip\\
\includegraphics[width=0.8\textwidth,height=0.3\textheight ,  trim = 1cm .1cm 0cm 1.cm, clip = true]{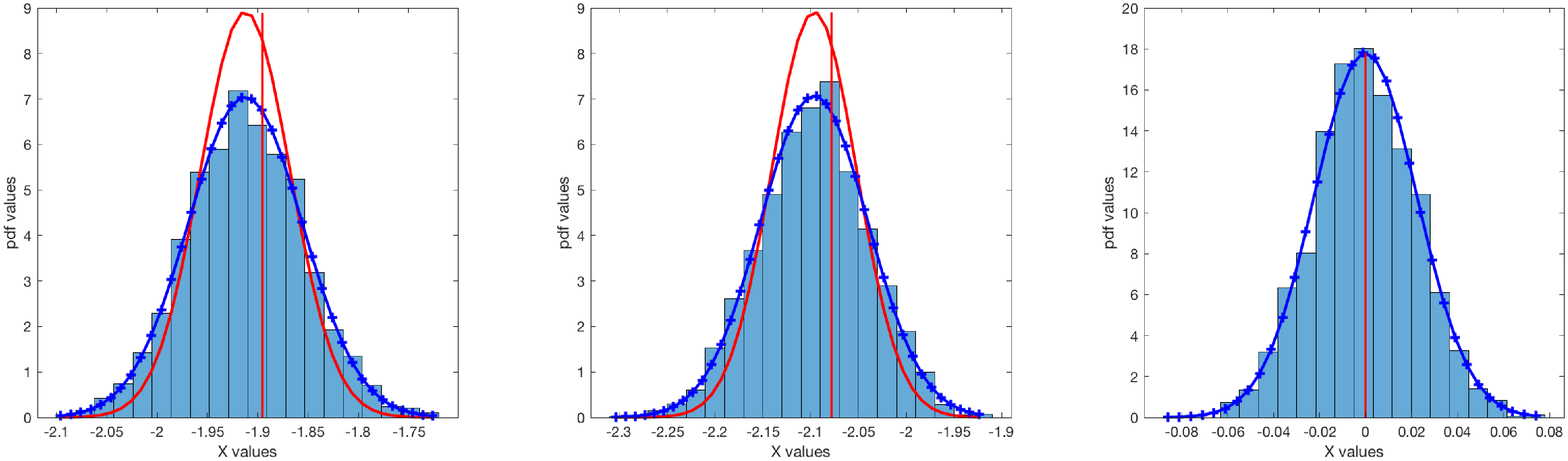}\\
\end{tabular}}
\caption{{\small{Posterior inference for $\beta_1$ (first column), $\beta_2$ (second column) and $\beta_3$ in the linear regression example based on one MCMC run (histogram), one skinny-VA run (continuous red line), and one midesize-VA run ($+$ blue line). Vertical lines locate the true values of the parameters.}} }\label{fig:lm}
\end{figure}

 \subsection{Sparse principal component estimation}\label{sec:spca}
We give another  illustration of the quasi-Bayesian framework with a non-standard example from sparse PCA.  Principal component analysis is a widely used technique for data exploration and data reduction (\cite{jolliffe:1986}). In  order to deal with high-dimensional datasets,  several works have introduced recently various versions of PCA that estimate  sparse principal components  (\cite{jolliffe:etal:03,zou:etal:06,shen:huang:08,lei:vu:15}). Extension of these ideas to a full Bayesian setting has been considered in the literature  but is computationally challenging  (\cite{pati:etal:14,gao:zhou:15,xie:etal:18}). Using the  quasi-Bayesian framework we explore here a fast regression-based approach to sparse PCA that we show works well when the sample size $n$ is close to $p$ and/or the spectral gap is sufficiently large.  We consider the following data generating process.

\begin{assumptionC}\label{H:spca}
The matrix $X\in\rset^{n\times p}$ is such that the rows of $X$ are i.i.d. from the Gaussian distribution $\textbf{N}_p(0,\Sigma)$ on $\rset^p$, with a covariance matrix $\Sigma$ of the form
\[\Sigma = \vartheta\theta_\star\theta_\star' + I_p,\]
for some sparse unit-vector $\theta_\star\in\rset^p$, and some absolute constant $\vartheta>0$. We set $s_\star \eqdef\|\theta_\star\|_0$. \end{assumptionC}
\medskip

Let $X = U\Lambda V'$ be the singular value decomposition (SVD) of $X$. Let $V_1$ be the first column of $V$. It was noted by \cite{zou:etal:06} that  setting $y = \Lambda_{11} U_1$, it holds for all $\lambda>0$ that
\[V_1 = \frac{\hat b}{\|\hat b\|_2},\;\;\mbox{ where }\;\; \hat b\eqdef \argmin_{\beta\in\rset^p} \| y -X\beta\|_2^2 + \lambda\|\beta\|_2^2.\]
This result suggests that one can recover the first principal component $V_1$ by sparse regression of  $y=\Lambda_{11}U_1$ on $X$. To implement this idea in a Bayesian framework we are naturally led  to the quasi-likelihood function 
\[\ell(\theta; X) = -\frac{1}{2\sigma^2}\|y - X\theta\|_2^2,\;\;\theta\in\rset^p,\]
for some constant $\sigma^2>0$. The resulting quasi-posterior distribution on $\Delta\times \rset^p$ is  the same as in (\ref{post:linmod}):
\[\Pi(\delta,\rmd\theta\vert Z)
\propto  e^{-\frac{1}{2\sigma^2}\|y-X\theta_\delta\|_2^2} \omega(\delta) \left(\frac{\rho_{1}}{2\pi}\right)^{\frac{\|\delta\|_0}{2}} \left(\frac{\rho_{0}}{2\pi}\right)^{\frac{p-\|\delta\|_0}{2}} e^{-\frac{\rho_{1}}{2}\|\theta_\delta\|_2^2}e^{-\frac{\rho_{0}}{2}\|\theta-\theta_\delta\|_2^2} \rmd \theta.\]
We analyze this quasi-posterior distribution.  One challenge here is that we do not possess a good understanding of the distribution of the quasi-score function $X'(\Lambda_{11} U_1 - X\theta_\star)/\sigma^2$ due to the intricate nature of the SVD decomposition. Hence Theorem \ref{thm:0} cannot be applied, and thus we do not know whether the quasi-posterior distribution is automatically sparse under the prior H\ref{H2}.  We work around this issue by hard-coding  sparsity directly in the prior as follows.

\begin{assumptionC}\label{H2:bis}
We assume that
\[\omega(\delta)\propto \q^{\|\delta\|_0}(1-\q)^{p-\|\delta\|_0}\textbf{1}_{\Delta_{\bar s}}(\delta),\;\;\;\delta\in\Delta,\]
for some integer $\bar s\geq s_\star$, where $\mathsf{q}\in (0,1)$ is such that $\frac{\mathsf{q}}{1-\mathsf{q}} = \frac{1}{p^{u+1}}$, for some  absolute constant $u>0$. Furthermore we will assume that $p\geq 9$, $p^{u/2}\geq 2e^{2\rho}$.
\end{assumptionC}
\medskip

Since $s_\star$ is not known, how to find $\bar s$ in practice that satisfies $\bar s\geq s_\star$ is not obvious, and would require some judgment from the researcher. However in terms of computations,  using C\ref{H2:bis} instead of H\ref{H2} implies only a minor change to the MCMC sampler in Algorithm \ref{algo:gibbs}\footnote{in STEP 2, if $\delta^{(k)}_j=0$ and $\iota=1$, we propose to do the change only if $\|\delta^{(k)}\|_0\leq \bar s$.}. For $a\in\rset$, $\textsf{sign}(a)=1$ if $a\geq 0$, and $-1$ otherwise.

\begin{corollary}\label{coro:spca}
Assume C\ref{H:spca}, C\ref{H2:bis},  and choose $\sigma^2=\vartheta$,  $\rho = \sqrt{\log(p)/n}$.  Suppose that $\|\theta_\star\|_\infty =O(1)$, as $p\to\infty$. There exist absolute constants $C_0,C$ such that for $n\geq C_0(\frac{p}{\vartheta} + \bar s\log(p))$, we have
\[\lim_{p\to\infty} \PE_\star\left[\textbf{1}_{\{\textsf{sign}(\pscal{V_1}{\theta_\star})=1\}} \Pi\left(\cB_{\theta_\star}\vert X\right) + \textbf{1}_{\{\textsf{sign}(\pscal{V_1}{\theta_\star})=-1\}} \Pi\left(\cB_{-\theta_\star} \vert X\right)\right] =1,\]
where  for $\theta_0\in\{\theta_\star,-\theta_\star\}$,
\[\cB_{\theta_0}\eqdef\bigcup_{\delta\in\Delta_{\bar s}}\{\delta\}\times \left\{\theta\in\rset^p:\; \|\theta_\delta -\theta_0\|_2\leq C\vartheta  \sqrt{\frac{\left(\frac{p}{\vartheta}+\log(p)\right)(\bar s+s_\star)}{n}},\; \|\theta-\theta_\delta\|_2\leq 3\sqrt{\gamma p}\right\}.\]
\end{corollary}
\begin{proof}
See Section \ref{proof:coro:spca}.
\end{proof}

It is well-known that the principal component is identified only up to a sign, which is reflected in Corollary \ref{coro:spca}. The assumption $\sigma^2=\vartheta$ is made for simplicity, since $\vartheta$ is typically unknown. To a certain extent the procedure is robust to a misspecification of $\sigma^2$. 

The contraction rate suggests that the method would perform poorly if the sample size and the spectral gap are  both small, which is confirmed in the simulations. One  important limitation of Corollary \ref{coro:spca} is that the convergence rate does not have the correct dependence on the spectral gap. This is most certainly an artifact of our method of proof. 

\subsubsection{Numerical illustration}
We generate a random matrix $X\in\rset^{n\times p}$  according C\ref{H:spca} with $p=1000$, and $n\in \{100,1000\}$, where $\beta_\star = (0.5,0.5,0,0.5,0.5,0,\ldots,0)'$. We consider two levels of the spectral gap $\vartheta\in\{5,20\}$. As above we set up the prior distribution with $u=2$, $\rho_1 =\sqrt{log(p)/n}$, and $\rho_0^{-1} =1/(4n)$. We use the same MCMC sampler as in the Gaussian graphical model of Section \ref{sec:lin:reg}, that we initialize from the lasso solution, and run the $2000$ iterations. We normalize the MCMC output to have unit-norm (at each iteration). We repeat all computations $100$ times and use the replications to approximate the distribution of the posterior means and posterior variances of the first three components of $\theta$ ($\theta_1,\theta_2$ and $\theta_3$). Using the $100$ replications we  also approximate the distribution  of the error 
\[\int \left\|\frac{\theta\theta'}{\|\theta\|_2^2} -\theta_\star\theta_\star'\right\|_2\Pi(\rmd\theta\vert X),\]
that we call projection approximation error. To assess the quasi-likelihood method advocated here we compare its performance to that of the frequentist estimator of (\cite{zou:etal:06}) as implemented in the \textsf{Matlab} package SpaSM (\cite{sjostrand:etal:18}). We present the results on Figure \ref{fig:spca:1} and \ref{fig:spca:2}. The results supports very well the conclusions of Corollary \ref{coro:spca}.

\medskip

\begin{figure}[]
\centering
\resizebox{\textwidth}{0.4\textheight}{\begin{tabular}{c}
Sparse PCA with $\vartheta = 5$,  $p=1000$, $n=100$.\medskip\\
\includegraphics[width=0.8\textwidth,height=0.3\textheight,  trim = 1cm .1cm 0cm 1.cm, clip = true ]{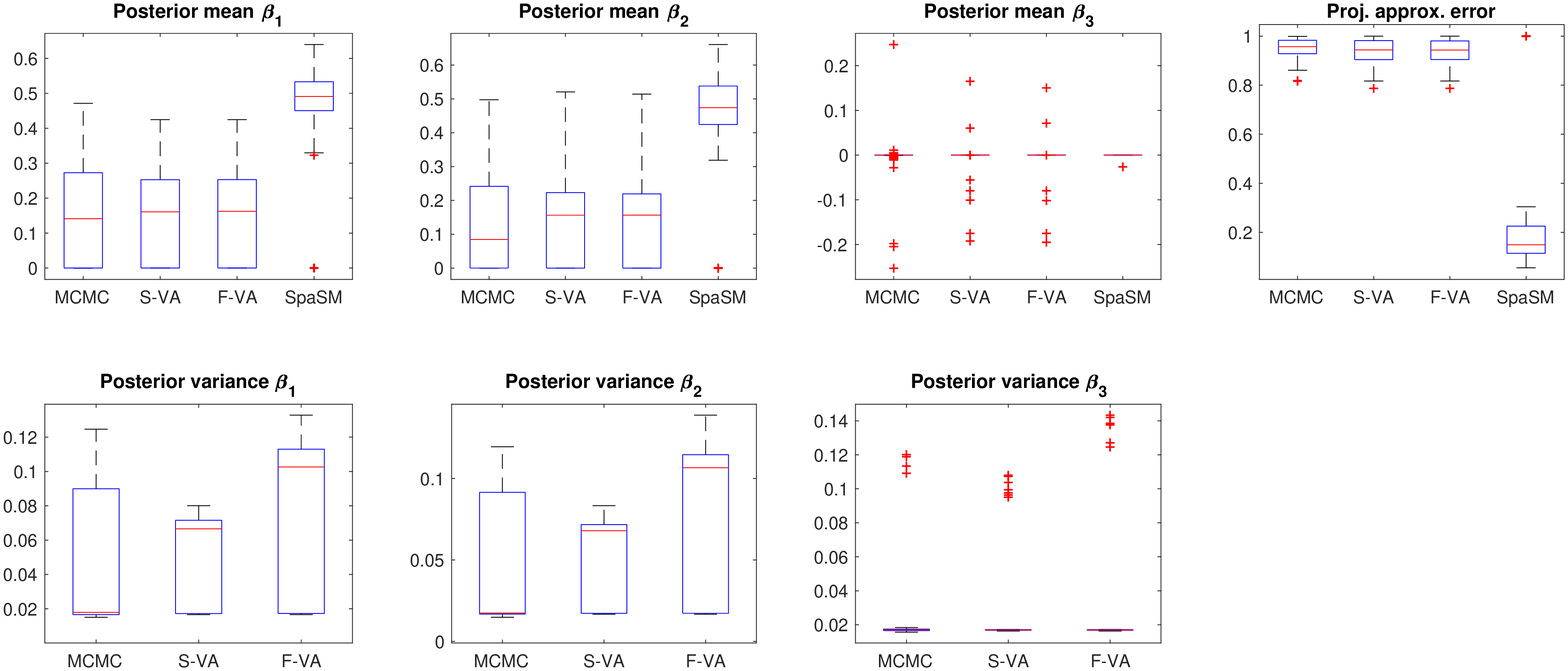}\\
Sparse PCA with $\vartheta = 5$,  $p=1000$, $n=1000$.\medskip\\
\includegraphics[width=0.8\textwidth,height=0.3\textheight ,  trim = 1cm .1cm 0cm 1.cm, clip = true]{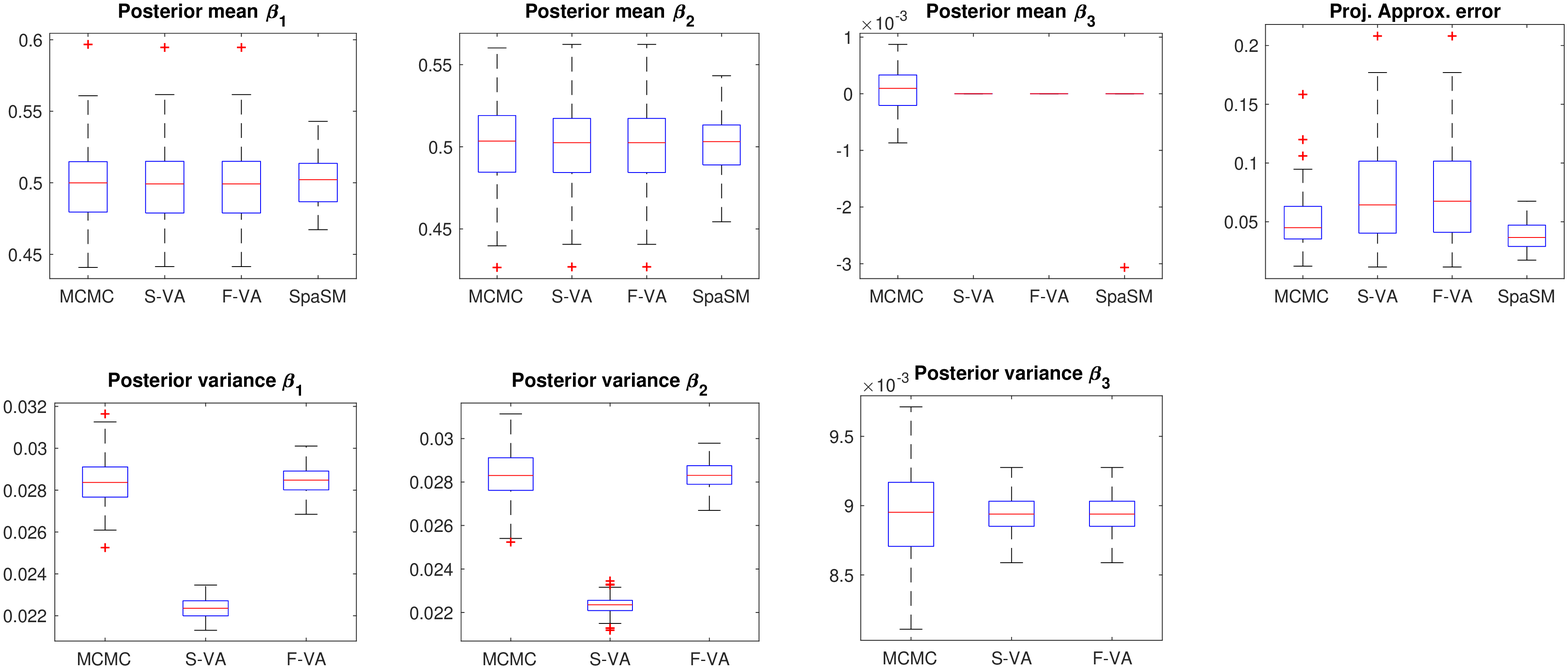}
\end{tabular}}
\caption{{\small{Distributions of posterior means and variances for $\beta_1,\beta_2,\beta_3$, and distribution of the projection approx. error. Estimated from $100$ replications. S-VA is skinny-VA, F-VA is full-VA. We also report similar distributions for the frequentist estimator computed by SpaSM.}} }\label{fig:spca:1}
\end{figure}

\begin{figure}[]
\centering
\resizebox{\textwidth}{0.4\textheight}{\begin{tabular}{c}
Sparse PCA with $\vartheta = 20$,  $p=1000$, $n=100$.\medskip\\
\includegraphics[width=0.8\textwidth,height=0.3\textheight ,  trim = 1cm .1cm 0cm 1.cm, clip = true]{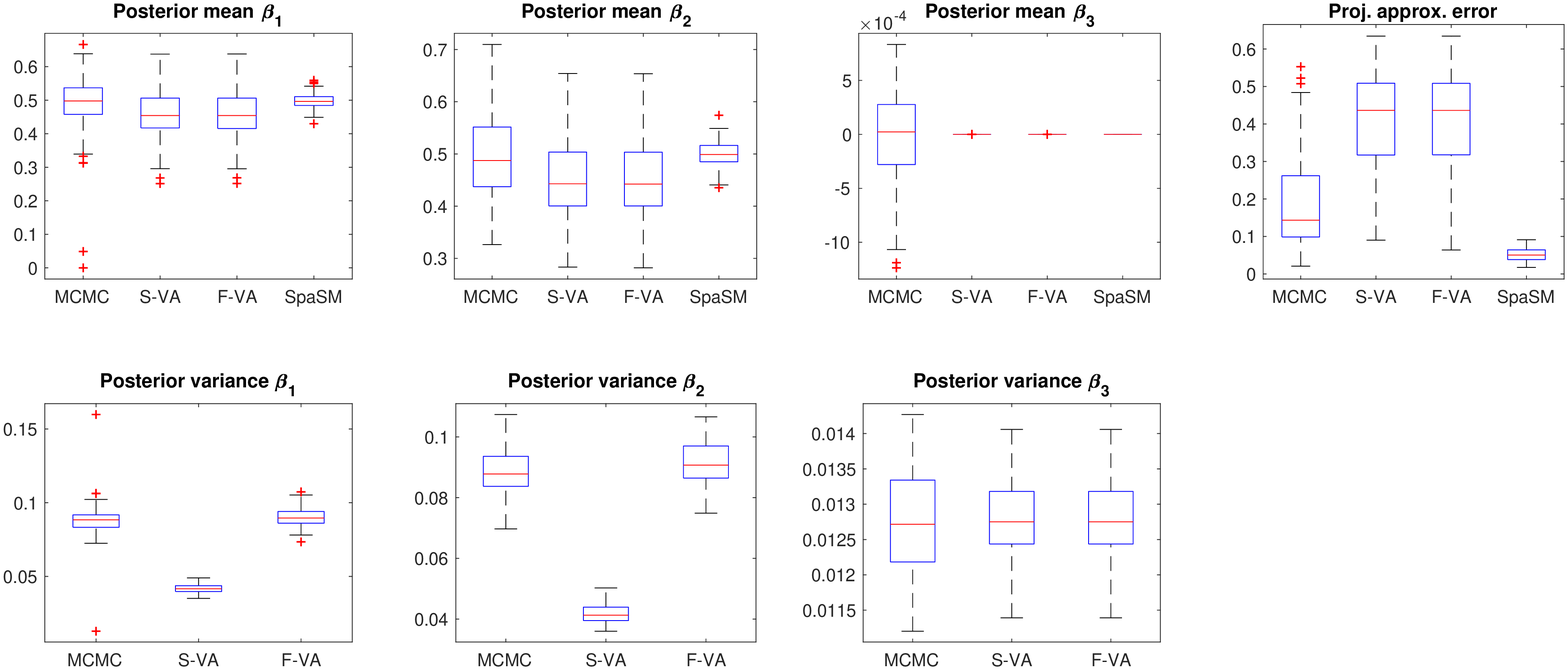}\\
Sparse PCA with $\vartheta = 20$,  $p=1000$, $n=1000$.\medskip\\
\includegraphics[width=0.8\textwidth,height=0.3\textheight ,  trim = 1cm .1cm 0cm 1.cm, clip = true]{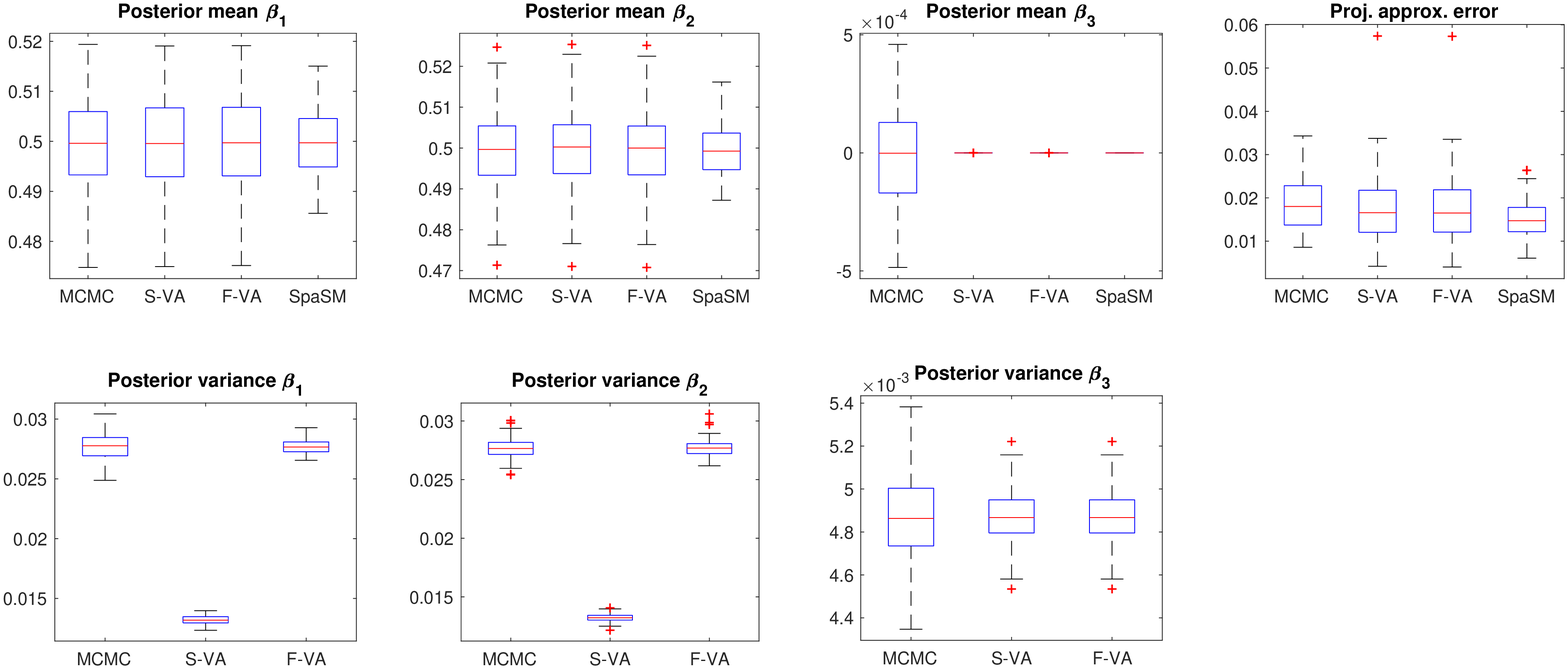}\\
\end{tabular}}
\caption{{\small{Distributions of posterior means and variances for $\beta_1,\beta_2,\beta_3$, and distribution of the projection approx. error. Estimated from $100$ replications. S-VA is skinny-VA, F-VA is full-VA. We also report similar distributions for the frequentist estimator computed by SpaSM.}} }\label{fig:spca:2}
\end{figure}

%
%

\section*{Acknowledgements}
The authors are grateful to Galin Jones, Scott Schmidler, Yuekai Sun, James Johndrow,  and Jonathan Taylor for very helpful discussions that have helped improved on an initial draft of the manuscript. 


\bibliographystyle{ims}
\bibliography{biblio_graph,biblio_mcmc,biblio_optim}

\begin{thebibliography}{49}
\expandafter\ifx\csname natexlab\endcsname\relax\def\natexlab#1{#1}\fi
\expandafter\ifx\csname url\endcsname\relax
  \def\url#1{\texttt{#1}}\fi
\expandafter\ifx\csname urlprefix\endcsname\relax\def\urlprefix{URL }\fi

\bibitem[{{Alquier} and {Ridgway}(2017)}]{alquier:ridgeway:17}
\textsc{{Alquier}, P.} and \textsc{{Ridgway}, J.} (2017).
\newblock {Concentration of tempered posteriors and of their variational
  approximations}.
\newblock \textit{arXiv e-prints}  arXiv:1706.09293.

\bibitem[{Atchade(2017)}]{atchade:15b}
\textsc{Atchade, Y.~A.} (2017).
\newblock On the contraction properties of some high-dimensional
  quasi-posterior distributions.
\newblock \textit{Ann. Statist.} \textbf{45} 2248--2273.

\bibitem[{Atchad\'e(2019)}]{atchade:15:c}
\textsc{Atchad\'e, Y.~F.} (2019).
\newblock Quasi-bayesian estimation of large gaussian graphical models.
\newblock \textit{Journal of Multivariate Analysis} \textbf{173} 656 -- 671.

\bibitem[{{Banerjee} and {Ghosal}(2013)}]{banerjee:ghosal13}
\textsc{{Banerjee}, S.} and \textsc{{Ghosal}, S.} (2013).
\newblock {Posterior convergence rates for estimating large precision matrices
  using graphical models}.
\newblock \textit{ArXiv e-prints} .

\bibitem[{Bhattacharya et~al.(2016)Bhattacharya, Chakraborty and
  Mallick}]{bhattacharya:etal:16}
\textsc{Bhattacharya, A.}, \textsc{Chakraborty, A.} and \textsc{Mallick, B.}
  (2016).
\newblock Fast sampling with gaussian scale mixture priors in high-dimensional
  regression.
\newblock \textit{Biometrika} \textbf{103} 985 -- 991.

\bibitem[{Blei et~al.(2017)Blei, Kucukelbir and McAuliffe}]{blei:etal:17}
\textsc{Blei, D.~M.}, \textsc{Kucukelbir, A.} and \textsc{McAuliffe, J.~D.}
  (2017).
\newblock Variational inference: A review for statisticians.
\newblock \textit{Journal of the American Statistical Association} \textbf{112}
  859--877.

\bibitem[{Boucheron et~al.(2013)Boucheron, Lugosi and
  Massart}]{boucheron:eta:13}
\textsc{Boucheron, S.}, \textsc{Lugosi, G.} and \textsc{Massart, P.} (2013).
\newblock \textit{Concentration inequalities: a nonasymptotic theory of
  independence}.
\newblock Springer Series in Statistics, Oxford University Press, Oxford.

\bibitem[{Castillo et~al.(2015)Castillo, Schmidt-Hieber and van~der
  Vaart}]{castillo:etal:14}
\textsc{Castillo, I.}, \textsc{Schmidt-Hieber, J.} and \textsc{van~der Vaart,
  A.} (2015).
\newblock Bayesian linear regression with sparse priors.
\newblock \textit{Ann. Statist.} \textbf{43} 1986--2018.

\bibitem[{Castillo and van~der Vaart(2012)}]{castillo:etal:12}
\textsc{Castillo, I.} and \textsc{van~der Vaart, A.} (2012).
\newblock Needles and straw in a haystack: Posterior concentration for possibly
  sparse sequences.
\newblock \textit{Ann. Statist.} \textbf{40} 2069--2101.

\bibitem[{Chernozhukov and Hong(2003)}]{chernozhukov:hong03}
\textsc{Chernozhukov, V.} and \textsc{Hong, H.} (2003).
\newblock An {MCMC} approach to classical estimation.
\newblock \textit{J. Econometrics} \textbf{115} 293--346.

\bibitem[{Chernozhukov et~al.(2007)Chernozhukov, Imbens and
  Newey}]{chernozhukov:etal:07}
\textsc{Chernozhukov, V.}, \textsc{Imbens, G.~W.} and \textsc{Newey, W.~K.}
  (2007).
\newblock Instrumental variable estimation of nonseparable models.
\newblock \textit{Journal of Econometrics} \textbf{139} 4 -- 14.

\bibitem[{Dobra et~al.(2011)Dobra, Lenkoski and Rodriguez}]{dobra:11a}
\textsc{Dobra, A.}, \textsc{Lenkoski, A.} and \textsc{Rodriguez, A.} (2011).
\newblock Bayesian inference for general {G}aussian graphical models with
  application to multivariate lattice data.
\newblock \textit{J. Amer. Statist. Assoc.} \textbf{106} 1418--1433.

\bibitem[{Gao and Zhou(2015)}]{gao:zhou:15}
\textsc{Gao, C.} and \textsc{Zhou, H.~H.} (2015).
\newblock Rate-optimal posterior contraction for sparse pca.
\newblock \textit{Ann. Statist.} \textbf{43} 785--818.

\bibitem[{George and McCulloch(1997)}]{george:mcculloch97}
\textsc{George, E.~I.} and \textsc{McCulloch, R.~E.} (1997).
\newblock Approaches to bayesian variable selection.
\newblock \textit{Statist. Sinica} \textbf{7} 339--373.

\bibitem[{Ghosal et~al.(2000)Ghosal, Ghosh and van~der Vaart}]{ghosal:etal:00}
\textsc{Ghosal, S.}, \textsc{Ghosh, J.~K.} and \textsc{van~der Vaart, A.~W.}
  (2000).
\newblock Convergence rates of posterior distributions.
\newblock \textit{Ann. Statist.} \textbf{28} 500--531.

\bibitem[{Horn and Johnson(2012)}]{horn:johnson}
\textsc{Horn, R.~A.} and \textsc{Johnson, C.~R.} (2012).
\newblock \textit{Matrix Analysis}.
\newblock 2nd ed. Cambridge University Press, New York, NY, USA.

\bibitem[{Ichimura(1993)}]{ichumura:93}
\textsc{Ichimura, H.} (1993).
\newblock Semiparametric least squares (sls) and weighted sls estimation of
  single-index models.
\newblock \textit{Journal of Econometrics} \textbf{58} 71 -- 120.

\bibitem[{Jiang and Tanner(2008)}]{jiang:tanner08}
\textsc{Jiang, W.} and \textsc{Tanner, M.~A.} (2008).
\newblock Gibbs posterior for variable selection in high-dimensional
  classification and data mining.
\newblock \textit{Ann. Statist.} \textbf{36} 2207--2231.

\bibitem[{Johnstone and Silverman(2004)}]{johnstone:04}
\textsc{Johnstone, I.~M.} and \textsc{Silverman, B.~W.} (2004).
\newblock Needles and straw in haystacks: Empirical bayes estimates of possibly
  sparse sequences.
\newblock \textit{Ann. Statist.} \textbf{32} 1594--1649.

\bibitem[{Jolliffe(1986)}]{jolliffe:1986}
\textsc{Jolliffe, I.} (1986).
\newblock \textit{Principal Component Analysis}.
\newblock Springer Verlag.

\bibitem[{Jolliffe et~al.(2003)Jolliffe, Trendafilov and
  Uddin}]{jolliffe:etal:03}
\textsc{Jolliffe, I.~T.}, \textsc{Trendafilov, N.~T.} and \textsc{Uddin, M.}
  (2003).
\newblock A modified principal component technique based on the lasso.
\newblock \textit{Journal of Computational and Graphical Statistics}
  \textbf{12} 531--547.

\bibitem[{Kato(2013)}]{kato:13}
\textsc{Kato, K.} (2013).
\newblock Quasi-{B}ayesian analysis of nonparametric instrumental variables
  models.
\newblock \textit{Ann. Statist.} \textbf{41} 2359--2390.

\bibitem[{Khondker et~al.(2013)Khondker, Zhu, Chu, Lin and
  Ibrahim}]{khondkeretal13}
\textsc{Khondker, Z.~S.}, \textsc{Zhu, H.}, \textsc{Chu, H.}, \textsc{Lin, W.}
  and \textsc{Ibrahim, J.~G.} (2013).
\newblock The {B}ayesian covariance lasso.
\newblock \textit{Stat. Interface} \textbf{6} 243--259.

\bibitem[{Kleijn and van~der Vaart(2006)}]{kleijn:vaart:06}
\textsc{Kleijn, B. J.~K.} and \textsc{van~der Vaart, A.~W.} (2006).
\newblock Misspecification in infinite-dimensional {B}ayesian statistics.
\newblock \textit{Ann. Statist.} \textbf{34} 837--877.

\bibitem[{Lei and Vu(2015)}]{lei:vu:15}
\textsc{Lei, J.} and \textsc{Vu, V.~Q.} (2015).
\newblock Sparsistency and agnostic inference in sparse pca.
\newblock \textit{Ann. Statist.} \textbf{43} 299--322.

\bibitem[{Lenkoski and Dobra(2011)}]{dobra:11b}
\textsc{Lenkoski, A.} and \textsc{Dobra, A.} (2011).
\newblock Computational aspects related to inference in {G}aussian graphical
  models with the {G}-{W}ishart prior.
\newblock \textit{J. Comput. Graph. Statist.} \textbf{20} 140--157.
\newblock Supplementary material available online.

\bibitem[{{Li} and {Jiang}(2014)}]{li:jiang:14}
\textsc{{Li}, C.} and \textsc{{Jiang}, W.} (2014).
\newblock {Model Selection for Likelihood-free Bayesian Methods Based on Moment
  Conditions: Theory and Numerical Examples}.
\newblock \textit{ArXiv e-prints} .

\bibitem[{Liao and Jiang(2011)}]{liao:jiang:11}
\textsc{Liao, Y.} and \textsc{Jiang, W.} (2011).
\newblock Posterior consistency of nonparametric conditional moment restricted
  models.
\newblock \textit{Ann. Statist.} \textbf{39} 3003--3031.

\bibitem[{Meinshausen and Buhlmann(2006)}]{meinshausen06}
\textsc{Meinshausen, N.} and \textsc{Buhlmann, P.} (2006).
\newblock High-dimensional graphs with the lasso.
\newblock \textit{Annals of Stat.} \textbf{34} 1436--1462.

\bibitem[{Mitchell and Beauchamp(1988)}]{mitchell:beauchamp:88}
\textsc{Mitchell, T.~J.} and \textsc{Beauchamp, J.~J.} (1988).
\newblock Bayesian variable selection in linear regression.
\newblock \textit{JASA} \textbf{83} 1023--1032.

\bibitem[{Narisetty and He(2014)}]{narisetti:he:14}
\textsc{Narisetty, N.} and \textsc{He, X.} (2014).
\newblock Bayesian variable selection with shrinking and diffusing priors.
\newblock \textit{Ann. Statist.} \textbf{42} 789--817.

\bibitem[{Negahban et~al.(2012)Negahban, Ravikumar, Wainwright and
  Yu}]{negahbanetal10}
\textsc{Negahban, S.~N.}, \textsc{Ravikumar, P.}, \textsc{Wainwright, M.~J.}
  and \textsc{Yu, B.} (2012).
\newblock A unified framework for high-dimensional analysis of $m$-estimators
  with decomposable regularizers.
\newblock \textit{Statistical Science} \textbf{27} 538--557.

\bibitem[{Pati et~al.(2014)Pati, Bhattacharya, Pillai and
  Dunson}]{pati:etal:14}
\textsc{Pati, D.}, \textsc{Bhattacharya, A.}, \textsc{Pillai, N.~S.} and
  \textsc{Dunson, D.} (2014).
\newblock Posterior contraction in sparse bayesian factor models for massive
  covariance matrices.
\newblock \textit{Ann. Statist.} \textbf{42} 1102--1130.
\newline\urlprefix\url{https://doi.org/10.1214/14-AOS1215}

\bibitem[{Peterson et~al.(2015)Peterson, Stingo and
  Vannucci}]{peterson:etal:14}
\textsc{Peterson, C.}, \textsc{Stingo, F.~C.} and \textsc{Vannucci, M.} (2015).
\newblock Bayesian inference of multiple gaussian graphical models.
\newblock \textit{Journal of the American Statistical Association} \textbf{110}
  159--174.

\bibitem[{Pinelis(2018)}]{pinelis:KL:18}
\textsc{Pinelis, I.} (2018).
\newblock Is ${KL}$-divergence ${D(P||Q)}$ strongly convex over ${P}$ in
  infinite dimension?
\newline\urlprefix\url{https://mathoverflow.net/q/307251}

\bibitem[{Raskutti et~al.(2010)Raskutti, Wainwright and Yu}]{raskutti:etal:10}
\textsc{Raskutti, G.}, \textsc{Wainwright, M.~J.} and \textsc{Yu, B.} (2010).
\newblock Restricted eigenvalue properties for correlated gaussian designs.
\newblock \textit{J. Mach. Learn. Res.} \textbf{11} 2241--2259.

\bibitem[{Ravikumar et~al.(2010)Ravikumar, Wainwright and
  Lafferty}]{ravikumaretal10}
\textsc{Ravikumar, P.}, \textsc{Wainwright, M.~J.} and \textsc{Lafferty, J.~D.}
  (2010).
\newblock High-dimensional {I}sing model selection using {$\ell_1$}-regularized
  logistic regression.
\newblock \textit{Ann. Statist.} \textbf{38} 1287--1319.

\bibitem[{Ravikumar et~al.(2011)Ravikumar, Wainwright, Raskutti and
  Yu}]{ravikumaretal11}
\textsc{Ravikumar, P.}, \textsc{Wainwright, M.~J.}, \textsc{Raskutti, G.} and
  \textsc{Yu, B.} (2011).
\newblock High-dimensional covariance estimation by minimizing
  {$\ell_1$}-penalized log-determinant divergence.
\newblock \textit{Electron. J. Stat.} \textbf{5} 935--980.

\bibitem[{Robert and Casella(2004)}]{robertetcasella04}
\textsc{Robert, C.~P.} and \textsc{Casella, G.} (2004).
\newblock \textit{Monte {C}arlo statistical methods}.
\newblock 2nd ed. Springer Texts in Statistics, Springer-Verlag, New York.

\bibitem[{Shen and Huang(2008)}]{shen:huang:08}
\textsc{Shen, H.} and \textsc{Huang, J.~Z.} (2008).
\newblock Sparse principal component analysis via regularized low rank matrix
  approximation.
\newblock \textit{Journal of Multivariate Analysis} \textbf{99} 1015 -- 1034.

\bibitem[{Sj\"ostrand et~al.(2018)Sj\"ostrand, Clemmensen, Larsen, Einarsson
  and Ersbøll}]{sjostrand:etal:18}
\textsc{Sj\"ostrand, K.}, \textsc{Clemmensen, L.}, \textsc{Larsen, R.},
  \textsc{Einarsson, G.} and \textsc{Ersbøll, B.} (2018).
\newblock Spasm: A matlab toolbox for sparse statistical modeling.
\newblock \textit{Journal of Statistical Software, Articles} \textbf{84} 1--37.

\bibitem[{Varin et~al.(2011)Varin, Reid and Firth}]{varin:etal:11}
\textsc{Varin, C.}, \textsc{Reid, N.} and \textsc{Firth, D.} (2011).
\newblock An overview of composite likelihood methods.
\newblock \textit{Statistica Sinica} \textbf{21} 5--42.

\bibitem[{Vershynin(2018)}]{vershynin:18}
\textsc{Vershynin, R.} (2018).
\newblock \textit{High-Dimensional Probability: An Introduction with
  Applications in Data Science}.
\newblock Cambridge Series in Statistical and Probabilistic Mathematics,
  Cambridge University Press.

\bibitem[{Wang and Blei(2018)}]{wang:blei:18}
\textsc{Wang, Y.} and \textsc{Blei, D.~M.} (2018).
\newblock Frequentist consistency of variational bayes.
\newblock \textit{Journal of the American Statistical Association} \textbf{0}
  1--15.

\bibitem[{{Xie} et~al.(2018){Xie}, {Xu}, {Priebe} and {Cape}}]{xie:etal:18}
\textsc{{Xie}, F.}, \textsc{{Xu}, Y.}, \textsc{{Priebe}, C.~E.} and
  \textsc{{Cape}, J.} (2018).
\newblock {Bayesian Estimation of Sparse Spiked Covariance Matrices in High
  Dimensions}.
\newblock \textit{arXiv e-prints}  arXiv:1808.07433.

\bibitem[{Yang and He(2012)}]{yang:he:2012}
\textsc{Yang, W.} and \textsc{He, X.} (2012).
\newblock {B}ayesian empirical likelihood for quantile regression.
\newblock \textit{Ann. Statist.} \textbf{40} 1102--1131.

\bibitem[{Yang et~al.(2016)Yang, Wainwright and Jordan}]{yang:etal:15}
\textsc{Yang, Y.}, \textsc{Wainwright, M.~J.} and \textsc{Jordan, M.~I.}
  (2016).
\newblock On the computational complexity of high-dimensional bayesian variable
  selection.
\newblock \textit{Ann. Statist.} \textbf{44} 2497--2532.

\bibitem[{Yu et~al.(2014)Yu, Wang and Samworth}]{yu:etal:15}
\textsc{Yu, Y.}, \textsc{Wang, T.} and \textsc{Samworth, R.~J.} (2014).
\newblock {A useful variant of the Davis-Kahan theorem for statisticians}.
\newblock \textit{Biometrika} \textbf{102} 315--323.

\bibitem[{Zou et~al.(2006)Zou, Hastie and Tibshirani}]{zou:etal:06}
\textsc{Zou, H.}, \textsc{Hastie, T.} and \textsc{Tibshirani, R.} (2006).
\newblock Sparse principal component analysis.
\newblock \textit{Journal of Computational and Graphical Statistics}
  \textbf{15} 265--286.

\end{thebibliography}

\newpage

\appendix

\section{Proofs of the main results}
\subsection{Some preliminary lemmas}
Let $\mu_\delta(\rmd\theta)$ denote the product measure on $\rset^p$ given by
\[\mu_{\delta}(\rmd \theta)\eqdef \prod_{j=1}^p\mu_{\delta_{j}}(\rmd \theta_{j}),\]
 where  $\mu_{0}(\rmd x)$ is the Dirac mass at $0$, and $\mu_1(\rmd x)$ is the Lebesgue measure on $\rset$. We start with  a useful lower bound on the normalizing constant.

\begin{lemma}\label{lem:control:nc}
Assume H\ref{H1}-H\ref{H2}. For $z\in\Zset$, let $C(z)$
denote the normalizing constant of $\Pi(\cdot\vert z)$.
For $z\in\e_{0}$, we have
\begin{equation}\label{eq:control:nc:check:G}
C(z)  \geq \omega(\delta_{\star}) e^{\ell(\theta_{\star};z)} e^{-\frac{\rho_1}{2}\|\theta_\star\|_2^2}\left(\frac{\rho_1}{\bar\kappa+\rho_1}\right)^{\frac{\|\theta_\star\|_0}{2}}.
\end{equation}
\end{lemma}
\begin{proof}
The proof is very similar to the proof of Lemma 11 of \cite{atchade:15b}. We set 
 \[\bar\omega(\delta)\eqdef \omega(\delta) \left(\frac{\rho_1}{2\pi}\right)^{\frac{\|\delta\|_0}{2}} \left(\frac{\rho_0}{2\pi}\right)^{\frac{p-\|\delta\|_0}{2}}.\]
 Fix  $z\in\e_{0}$. Then $\Pi$ is well-defined, and we  have 
\begin{eqnarray*}
C(z) &  = & \sum_{\delta\in\Delta} \bar\omega(\delta)\int_{\rset^p}e^{-\ell(\theta_\delta;z) -\frac{\rho_1}{2}\|\theta_\delta\|_2^2 -\frac{\rho_0}{2}\|\theta-\theta_\delta\|_2^2}\rmd \theta\\
& \geq & \bar\omega(\delta_{\star})\int_{\rset^p}e^{-\ell(\theta_{\delta_\star};z) -\frac{\rho_1}{2}\|\theta_{\delta_\star}\|_2^2 -\frac{\rho_0}{2}\|\theta-\theta_{\delta_\star}\|_2^2}\rmd \theta \\
& = & \bar\omega(\delta_{\star})(2\pi\rho_0^{-1})^{\frac{p-\|\delta_\star\|_0}{2}}\int_{\rset^p}e^{\ell(u;z)-\frac{\rho_1}{2}\|u\|_2^2}\mu_{\delta_\star}(\rmd u).\end{eqnarray*}
Setting $G\eqdef \nabla\ell(\theta_\star;z)$,  we have for all $u\in\rset^p_{\delta_\star}$ and $z\in\e_0$,
\[\ell(u;z) -\ell(\theta_\star;z) - \pscal{G}{u-\theta_\star} \geq -\frac{\bar\kappa}{2}\|u-\theta_\star\|_2^2,\]
which implies that
\[C(z) \geq \omega(\delta_{\star}) \left(\frac{\rho_1}{2\pi}\right)^{s_\star/2} e^{\ell(\theta_\star;z)-\frac{\rho}{2}\|\theta_\star\|_2^2} \int_{\rset^p}e^{\pscal{G}{u-\theta_\star}-\frac{\bar\kappa}{2}\|u-\theta_\star\|_2^2 + \frac{\rho_1}{2}\|\theta_\star\|_2^2 -\frac{\rho_1}{2}\|u\|_2^2}\mu_{\delta_\star}(\rmd u).\]
For all $u\in\rset^p_{\delta_\star}$, $(1/2)(\|\theta_\star\|_2^2 - \|u\|_2^2) = -\frac{1}{2}\|u-\theta_\star\|_2^2 -\pscal{\theta_\star}{u-\theta_\star}$. Therefore,
\begin{multline*}
\int_{\rset^p}e^{\pscal{G}{u-\theta_\star}-\frac{\bar\kappa}{2}\|u-\theta_\star\|_2^2 + \frac{\rho_1}{2}\|\theta_\star\|_2^2 -\frac{\rho_1}{2}\|u\|_2^2}\mu_{\delta_\star}(\rmd u) \\
= \int_{\rset^p}e^{\pscal{G-\rho_1 \theta_\star}{u-\theta_\star}-\frac{\bar\kappa+\rho_1}{2}\|u-\theta_\star\|_2^2}\mu_{\delta_\star}(\rmd u) = \left(\frac{2\pi}{\bar\kappa + \rho_1}\right)^{\frac{s_\star}{2}}e^{\frac{\bar\kappa+\rho_1}{2}\|G-\rho_1\theta_\star\|_2^2},\end{multline*}
and (\ref{eq:control:nc:check:G})  follows easily.

\end{proof}

Our proofs rely on the existence of some generalized testing procedures that we develop next, following ideas from \cite{atchade:15b}. More specifically we will make use of the following result which follows by combining Lemma 6.1 and Equation (6.1) of \cite{kleijn:vaart:06}.
\begin{lemma}[Kleijn-Van der Vaart (2006)]\label{kv}
Let $(\Xset,\B,\lambda)$ be a measure space with a sigma-finite measure $\lambda$. Let $p$ be a density on $\Xset$, and  $\mathcal{Q}$ a family of integrable real-valued functions on $\Xset$. There exists a measurable $\phi:\;\Xset \to [0,1]$ such that
\[\sup_{q\in\mathcal{Q}}\left[\int \phi p\rmd \lambda + \int (1-\phi)q \rmd \lambda\right]\leq \sup_{q\in\textsf{conv}(\mathcal{Q})} \H(p,q),\]
where $\textsf{conv}(\mathcal{Q})$ is the convex hull of $\mathcal{Q}$, and $\H(q_1,q_2)\eqdef\int \sqrt{q_1q_2}\rmd \lambda$.
\end{lemma}

We introduce the quasi-likelihood
\[f_\theta(z) \eqdef e^{\ell(\theta;z)},\;\;\theta\in\rset^p,\;z\in\Zset.\]
For $\theta_1\in\rset^p$, we recall that
\[\L_{\theta_1}(\theta;z)\eqdef \ell(\theta;z) - \ell(\theta_1;z) -\pscal{\nabla\ell(\theta_1;z)}{\theta-\theta_1},\;\theta\in\rset^p.\]

We develop the test in a slightly more general setting. More specifically , in order to handle the PCA example we will allow the mode of  $\ell(\cdot;z)$ to depend on  $z$.

Let $\delta_\star$ be some sparse element $\Delta$. Let $\Theta_\star$ be a finite nonempty subset of $\rset^p_{\delta_\star}$ (the set of possible contraction points). Let $\bar\rho>0$ be a constant, $\bar s\geq 1$ an integer, and $\r$ a rate function. For each $\theta_\star\in\Theta_\star$,  we define
\begin{multline*}
\e_{\textsf{t},\theta_\star}\eqdef\left\{z\in\Zset:\; \|\nabla \log f_{\theta_\star}(z)\|_\infty\leq \frac{\bar\rho}{2},\;\;\right.\\
\left. \mbox{ and for all } \delta\in\Delta_{\bar s},\;\theta\in\rset^p_\delta,\;  \L_{\theta_\star}(\theta;z) \leq -\frac{1}{2}\r(\|\theta-\theta_\star\|_2)\right\},
\end{multline*}
which  roughly represents the set of data points for which $\Pi(\cdot\vert z)$ could contract towards $\theta_\star$.

\begin{lemma}\label{test}
Set $s_\star\eqdef\|\delta_\star\|_0$, and
\[\epsilon \eqdef \inf\left\{z>0:\; \r(x) -2\bar \rho(s_\star + \bar s)^{1/2} x\geq 0,\;\mbox{ for all } x\geq z\right\}.\] 
Let $f_\star$ be a density on $\Zset$, and $M>2$ a constant. There exists a measurable function $\phi:\Zset\to [0,1]$ such that
\[ \int_\Zset \phi(z)f_\star(z)\rmd z \leq  \frac{2|\Theta_\star|(9p)^{\bar s}e^{-\frac{M}{8}\bar \rho (s_\star + \bar s)^{1/2}\epsilon}}{1-e^{-\frac{M}{8}\bar \rho (s_\star + \bar s)^{1/2}\epsilon}},\]
where $|\Theta_\star|$ denotes the cardinality of $\Theta_\star$. Furthermore, for any $\delta\in\Delta_{\bar s}$, any $\theta\in\rset^p_\delta$ such that  $\|\theta-\theta_\star\|_2 >jM\epsilon$ for some $j\geq 1$, and some $\theta_\star\in\Theta_\star$, we have
\[\int_{\e_{\textsf{t},\theta_\star}}(1-\phi(z)) \frac{f_{\theta}(z)}{f_{\theta_\star}(z)}f_{\star}(z)\rmd z \leq  e^{-\frac{1}{8}\r\left(\frac{jM\epsilon}{2}\right)}.\]
\end{lemma}
\begin{proof}

Define
\[\bar q_{\theta_\star,u}(z) \eqdef \frac{f_{u}(z)}{f_{\theta_\star}(z)}f_{\star}(z)\textbf{1}_{\e_{\textsf{t},\theta_\star}}(z),\;\;\;\theta_\star\in\Theta_\star,\;\;u\in\rset^p,\;\;z\in\Zset.\]
Using the properties of the event $\e_{\textsf{t},\theta_\star}$, we note that for $\delta\in\Delta_{\bar s}$, and  $u\in\rset^p_\delta$ we have
\begin{equation}\label{bound:ratio:test}
\int_{\Zset}\bar q_{\theta_\star,u}(z)\rmd z = \int_{\e_{\textsf{t},\theta_\star}} e^{\pscal{\nabla \ell(\theta_\star;z)}{u-\theta_\star} +\L_{\theta_\star}(u;z)}f_{\star}(z)\rmd z \leq e^{\frac{\bar\rho}{2}\|u-\theta_\star\|_1}<\infty.\end{equation}
Fix $\eta\geq  2\epsilon$ arbitrary. Fix $\theta_\star\in\Theta_\star$, $\delta\in\Delta_{\bar s}$, and fix  $\theta\in\rset^p_\delta$ such that  $\|\theta-\theta_\star\|_2>\eta$. Let 
\[\mathcal{P}=\mathcal{P}_{\theta_\star,\delta,\theta}\eqdef \left\{\bar q_{\theta_\star,u}:\; u\in\rset^p_\delta,\;\|u-\theta\|_2\leq \frac{\eta}{2}\right\}.\]
According to Lemma \ref{kv}, applied with  $p = f_{\star}$, and $\mathcal{Q}=\mathcal{P}$,  there exists a test function $\phi_{\theta_\star,\delta,\theta}$ (that we will write simply as $\phi$ for convenience) such that
\begin{equation}\label{lem:test:eq1}
\sup_{q\in\mathcal{P}}\left[\int \phi  f_{\star} + \int (1-\phi) q \right]\leq \sup_{q\in\textsf{conv}(\mathcal{P})} \int_{\Zset}\sqrt{f_\star(z) q(z)}\rmd z.\end{equation}
Any $q\in \textsf{conv}(\mathcal{P})$ can be written as $q = \sum_j\alpha_j \bar q_{\theta_\star,u_j}$, where $\sum_j\alpha_j=1$, $u_j\in\rset^p_\delta$, $\|u_j-\theta\|_2\leq \eta/2$. Notice that this implies that $\|u_j-\theta_\star\|_2> \eta/2\geq \epsilon$. Therefore, by Jensen's inequality, the first inequality of (\ref{bound:ratio:test}), and the properties of the set $\e_{\textsf{t},\theta_\star}$, we get
\begin{eqnarray*}
\int_{\Zset}\sqrt{f_\star(z) q(z)}\rmd z 
& \leq &\sqrt{\sum_{j} \alpha_j \int_{\e_{\textsf{t},\theta_\star}} \frac{f_{u_j}(z)}{f_{\theta_\star}(z)} f_{\star}(z)\rmd z}\nonumber\\
& \leq & \sqrt{\sum_{j} \alpha_j  e^{\frac{\bar\rho}{2}\|u_j-\theta_\star\|_1 -\frac{1}{2}\r(\|u_j-\theta_\star\|_2)}},\nonumber\\
& \leq & \sqrt{\sum_{j} \alpha_j  e^{-\frac{1}{4}\r(\|u_j-\theta_\star\|_2)}}\\
& \leq & e^{-\frac{1}{8}\r\left(\frac{\eta}{2}\right)}.
\end{eqnarray*}
Consequently, (\ref{lem:test:eq1}) yields
\begin{equation}\label{lem:test:eq2}
\sup_{q\in\mathcal{P}}\left[\int \phi  f_{\star} + \int (1-\phi)  q \right] \leq e^{-\frac{1}{8}\r\left(\frac{\eta}{2}\right)}.
\end{equation}

For $M>2$, write $\cup_{\theta_\star}\cup_\delta\{\theta\in\rset^p_\delta:\;\|\theta-\theta_\star\|_2 > M\epsilon\}$ as $\cup_{\theta_\star}\cup_\delta\cup_{j\geq 1} \A_{\epsilon}(\theta_\star,\delta,j)$, where the unions in $\delta$ are taken over all $\delta$ such that $\|\delta\|_0 \leq \bar s$, and 
 \[\A_{\epsilon}(\theta_\star,\delta,j) \eqdef \left\{\theta\in\rset^p_\delta:\;jM\epsilon <\|\theta-\theta_\star\|_2 \leq (j+1)M\epsilon\right\}.\]
For $\A_{\epsilon}(\theta_\star,\delta,j)\neq\emptyset$, let $\mathcal{S}(\theta_\star,\delta,j)$ be a maximally $(jM\epsilon/2)$-separated point in $\A_{\epsilon}(\theta_\star,\delta,j)$. It is easily checked that the cardinality of $\mathcal{S}(\theta_\star,\delta,j)$ is upper bounded by $9^{\|\delta\|_0}\leq  9^{\bar s}$ (see for instance \cite{ghosal:etal:00}~Example 7.1 for the arguments). 
For $\theta\in\mathcal{S}(\theta_\star,\delta,j)$, let $\phi$ denote the test function obtained above with  $\eta = jM\epsilon$. From  (\ref{lem:test:eq2}), this test  satisfies 
\begin{equation}\label{lem:test:eq3}
\sup_{u\in\rset^p_\delta,\;\|u-\theta\|_2\leq \frac{jM\epsilon}{2}} \left[ \int_{\Zset}\phi(z) f_\star(z)\rmd z + \int_{\Zset}(1-\phi(z)) \bar q_{\theta_\star,u}(z)\rmd z \right]\leq   e^{-\frac{1}{8}\r\left(\frac{jM\epsilon}{2}\right)} .\end{equation}
We then set
\[\bar \phi = \max_{\theta_\star\in\Theta_\star}\;\max_{\delta:\;\|\delta\|_0\leq \bar s}\;\;\sup_{j\geq 1}\;\max_{\theta\in\mathcal{S}(\theta_\star,\delta,j)}\; \phi.\]
It then follows that
\begin{multline*}
\int_\Zset \bar \phi(z) f_\star(z)\rmd z \leq \sum_{\theta_\star}\sum_{k=0}^{\bar s}\;\sum_{\delta:\;\|\delta\|_0=k}\; \sum_{j\geq 1} \;\sum_{\theta \in\mathcal{S}(\theta_\star,\delta,j)} \int_{\Zset} \phi(z)f_\star(z)\rmd z \\
\leq |\Theta_\star| \sum_{k=0}^{\bar s}{ p\choose k}9^{k} \sum_{j\geq 1}e^{-\frac{1}{8}\r\left(\frac{jM\epsilon}{2}\right)} \leq 2|\Theta_\star| (9p)^{\bar s}  \sum_{j\geq 1}e^{-\frac{1}{8}\r\left(\frac{jM\epsilon}{2}\right)}.\end{multline*}
Since $jM\epsilon/2\geq \epsilon$, we can say that $\r(jM\epsilon/2) \geq 2\bar\rho (s_\star + \bar s)^{1/2}(jM\epsilon/2)$. Hence
\[\sum_{j\geq 1}e^{-\frac{1}{8}\r\left(\frac{jM \epsilon}{2}\right)} \leq \frac{e^{-\frac{M}{8}\bar \rho (s_\star + \bar s)^{1/2}\epsilon}}{1-e^{-\frac{M}{8}\bar \rho (s_\star + \bar s)^{1/2}\epsilon}}.\]

And if for some $\delta$, such that $\|\delta\|_0\leq \bar s$, some $\theta_\star\in\Theta_\star$,  and some $\theta\in\rset^p_\delta$ we have $\|\theta-\theta_\star\|_2>jM\epsilon$, then $\theta$ resides within $(iM\epsilon)/2$ of some point $\theta_0\in\mathcal{S}(\theta_\star,\delta,i)$ for some $i\geq j$. Hence, by  (\ref{lem:test:eq3}),
\[\int_{\Zset}(1-\bar\phi(z)) \bar q_{\theta_\star,\theta}(z)\rmd z \leq\int_{\Zset}(1-\phi(z)) \bar q_{\theta_\star,\theta}(z)\rmd z \leq e^{-\frac{1}{8}\r\left(\frac{iM\epsilon}{2}\right)} \leq e^{-\frac{1}{8}\r\left(\frac{jM\epsilon}{2}\right)}.\]
This ends the proof.
\end{proof}

\medskip

\subsection{Proof Theorem \ref{thm:0}}
\label{sec:proof:thm:0}
Let $f:\;\Delta\times\rset^p\to [0,\infty)$ be some arbitrary measurable function. Take $\e\subseteq\e_0$. By the control on the normalizing constant obtained in Lemma \ref{lem:control:nc},  we have
\begin{multline*}
\textbf{1}_{\e}(z)  \int f\rmd \Pi(\cdot\vert z) \leq  \left(1 +\frac{\bar\kappa}{\rho_1}\right)^{\frac{s_\star}{2}}  \\
\times\sum_{\delta\in\Delta}\;\frac{\omega(\delta)}{\omega(\delta_{\star})}\left(\frac{\rho_1}{2\pi}\right)^{\frac{\|\delta\|_0}{2}}\textbf{1}_{\e}(z)\int_{\rset^p}f(\delta,u)\frac{e^{\ell(u;z) -\frac{\rho_1}{2}\|u\|_2^2}}{e^{\ell(\theta_\star;z)-\frac{\rho_1}{2}\|\theta_\star\|_2^2}}\mu_\delta(\rmd u).
\end{multline*}

We write
\[\ell(u;z) - \ell(\theta_\star;z) = \L_{\theta_\star}(u;z) + \pscal{\nabla\ell(\theta_\star;z)}{u-\theta_\star}.\]
Therefore, since for $z\in\e\subseteq\e_0$, $\|\nabla\ell(\theta_\star;z)\|_\infty\leq \bar\rho/2$, it follows that for $z\in\e$
\[\ell(u;z) - \ell(\theta_\star;z) \leq \L_{\theta_\star}(u;z) + \left(1-\frac{\rho_1}{\bar\rho}\right)\pscal{\nabla\ell(\theta_\star;z)}{u-\theta_\star} + \frac{\rho_1}{2}\|u-\theta_\star\|_1.\]
We deduce from the above and Fubini's theorem that
\begin{multline}\label{control:prob:3}
\PE_\star\left[\textbf{1}_{\e}(Z) \int f\rmd\Pi(\cdot\vert Z) \right] \leq \left(1 +\frac{\bar\kappa}{\rho_1}\right)^{\frac{s_\star}{2}} \sum_{\delta\in\Delta}\;\frac{\omega(\delta)}{\omega(\delta_{\star})}\left(\frac{\rho_1}{2\pi}\right)^{\frac{\|\delta\|_0}{2}}\\\times \int_{\rset^p}f(\delta,u)e^{\frac{\rho_1}{2}\left(\|\theta_\star\|_2^2-\|u\|_2^2\right) +\frac{\rho_1}{2}\|u-\theta_\star\|_1}\PE_\star\left[\textbf{1}_{\e}(Z)e^{\L(u;Z)  + \left(1-\frac{\rho_1}{\bar\rho}\right)\pscal{\nabla\ell(\theta_\star;Z)}{u-\theta_\star}} \right]\mu_\delta(\rmd u).
\end{multline}

Set  $\textsf{d}(u)\eqdef -\rho_1\|u\|_1 + \rho_1\|\theta_\star\|_1 +(\rho_1/2)\|u-\theta_\star\|_1$, $u\in\rset^p$. Given (\ref{eq:curvature:cond:thm0}), we claim that 
\begin{equation}\label{thm0:claim}
e^{\textsf{d}(u)}\PE_\star\left[\textbf{1}_{\e}(Z)e^{\L(u;Z)  + \left(1-\frac{\rho_1}{\bar\rho}\right)\pscal{\nabla\ell(\theta_\star;Z)}{u-\theta_\star}} \right] \leq e^{\frac{\mathsf{a}_0}{2}} e^{-\frac{\rho_1}{4}\|u-\theta_\star\|_1} ,\;\;\;u\in\rset^p,\end{equation}
where $\mathsf{a}_0 = -\min_{x>0}[\r_0(x) -4\rho_1 s_\star^{1/2}]$. The proof of this statement is  essentially the same as in \cite{castillo:etal:14}~Theorem 1. 
We give the details for completeness. Indeed,
\begin{eqnarray*}
\textsf{d}(u) &=& \frac{\rho_1}{2}\|\delta_\star\cdot(u-\theta_\star)\|_1 +\frac{\rho_1}{2}\|\delta_\star^c\cdot u\|_1 -\rho_1\|\delta_\star\cdot u\|_1 -\rho_1\|\delta_\star^c\cdot u\|_1 +\rho_1\|\theta_\star\|_1\\
& \leq & -\frac{\rho_1}{2}\|\delta_\star^c\cdot(u-\theta_\star)\|_1 + \frac{3\rho_1}{2}\|\delta_\star\cdot(u-\theta_\star)\|_1.
\end{eqnarray*}
If $\|\delta_\star^c\cdot(u-\theta_\star)\|_1 > 7\|\delta_\star\cdot(u-\theta_\star)\|_1$, we easily deduce that $\textsf{d}(u) \leq -\frac{\rho_1}{4}\|u-\theta_\star\|_1$.  This bound together with (\ref{eq:curvature:cond:thm0}) shows that the claim holds true when $\|\delta_\star^c\cdot(u-\theta_\star)\|_1 > 7\|\delta_\star\cdot(u-\theta_\star)\|_1$. If $\|\delta_\star^c\cdot(u-\theta_\star)\|_1 \leq 7\|\delta_\star\cdot(u-\theta_\star)\|_1$, then again by (\ref{eq:curvature:cond:thm0}), and the bound on $\textsf{d}(u)$ obtained above, we deduce that the logarithm of the left-hand side of (\ref{thm0:claim}) is upper bounded by
\begin{multline*}
-\frac{\rho_1}{2}\|\delta_\star^c\cdot(u-\theta_\star)\|_1 + \frac{3\rho_1}{2}\|\delta_\star\cdot(u-\theta_\star)\|_1 -\frac{1}{2}\r_0(\|\delta_\star\cdot(u-\theta_\star)\|_2) \\
\leq  -\frac{\rho_1}{2}\|u-\theta_\star\|_1 +2\rho_1 s_\star^{1/2}\|\delta_\star\cdot(u-\theta_\star)\|_2 -\frac{1}{2}\r_0(\|\delta_\star\cdot(u-\theta_\star)\|_2)\\
 \leq -\frac{\rho_1}{2}\|u-\theta_\star\|_1 -\frac{1}{2}\left[\r_0(\|\delta_\star\cdot(u-\theta_\star)\|_2) -4\rho_1 s_\star^{1/2}\|\delta_\star\cdot(u-\theta_\star)\|_2\right] \\
 \leq -\frac{\rho_1}{2}\|u-\theta_\star\|_1 + \frac{\mathsf{a_0}}{2},
\end{multline*}
which also gives  the stated claim. 
Hence (\ref{control:prob:3}) becomes
\begin{multline}\label{control:prob:4}
\PE_\star\left[\textbf{1}_{\e}(Z) \int f\rmd\Pi(\cdot\vert Z) \right]   \leq \left(1 +\frac{\bar\kappa}{\rho_1}\right)^{\frac{s_\star}{2}}e^{\frac{\mathsf{a_0}}{2}}  \sum_{\delta\in\Delta}\; \frac{\omega(\delta)}{\omega(\delta_{\star})}\left(\frac{\rho_1}{2\pi}\right)^{\frac{\|\delta\|_0}{2}}\\\times \int_{\rset^p} f(\delta,u) e^{\frac{\rho_1}{2}\left(\|\theta_\star\|_2^2 - \|u\|_2^2\right)-\rho_1\left(\|\theta_\star\|_1 - \|u\|_1\right)} e^{-\frac{\rho_1}{4}\|u-\theta_\star\|_1}\mu_\delta(\rmd u).
\end{multline}
The integral in the last display is bounded from above by
\begin{multline*}
\int_{\rset^p} f(\delta,u) e^{-\frac{\rho_1}{2}\|u-\theta_\star\|_2^2 +\rho_1\|\theta_\star\|_2\|u-\theta_\star\|_2 + \frac{3\rho_1}{4}\|u-\theta_\star\|_1}\mu_\delta(\rmd u)\\
\leq  e^{2\rho_1\|\theta_\star\|_2^2}e^{2\rho_1\|\delta\|_0} \int_{\rset^p} f(\delta,u)e^{-\frac{\rho_1}{4}\|u-\theta_\star\|_2^2}\mu_\delta(\rmd u),\end{multline*}
using some simple algebraic majoration. Then (\ref{control:prob:4}) becomes
 \begin{multline}\label{control:prob:42}
\PE_\star\left[\textbf{1}_{\e}(Z) \int f\rmd\Pi(\cdot\vert Z) \right]   \leq \left(1 +\frac{\bar\kappa}{\rho_1}\right)^{\frac{s_\star}{2}}e^{\frac{\mathsf{a_0}}{2} +2\rho_1\|\theta_\star\|_2^2} \\
\times\sum_{\delta\in\Delta}\; \frac{\omega(\delta)}{\omega(\delta_{\star})}(\sqrt{2}e^{2\rho_1})^{\|\delta\|_0} \left(\frac{\rho_1}{4\pi}\right)^{\frac{\|\delta\|_0}{2}}\int_{\rset^p} f(\delta,u)e^{-\frac{\rho_1}{4}\|u-\theta_\star\|_2^2}\mu_\delta(\rmd u).
\end{multline}

In the special case where $f(\delta,u) = \textbf{1}_{\{\|\delta\|_0\geq s_\star + k\}}$ for some $k\geq 0$, we have
\[
\PE_\star\left[\textbf{1}_{\e}(Z) \Pi(\|\delta\|_0\geq s_\star+k\vert Z)\right]  \leq \left(1 +\frac{\bar\kappa}{\rho_1}\right)^{\frac{s_\star}{2}}  e^{\frac{\mathsf{a_0}}{2} +2\rho_1\|\theta_\star\|_2^2} \sum_{\delta:\;\|\delta\|_0\geq s_\star + k} \frac{\omega(\delta)}{\omega_{\delta_\star}} \left(\sqrt{2}e^{2\rho_1}\right)^{\|\delta\|_0}.\]
By H\ref{H2}, we have
\begin{multline*}
\sum_{\delta:\;\|\delta\|_0 \geq  s_\star + k} \frac{\omega(\delta)}{\omega(\delta_{\star})} \left(\sqrt{2}e^{2\rho_1}\right)^{\|\delta\|_0}  = \sum_{j=s_\star + k }^p {p\choose j}\left(\frac{\textsf{q}}{1-\textsf{q}}\right)^{j-s_\star}\left(\sqrt{2}e^{2\rho_1}\right)^j\\
\leq {p\choose s_\star}\left(\sqrt{2}e^{2\rho_1}\right)^{s_\star} \sum_{j=s_\star+k}^p \left(\frac{\sqrt{2}e^{2\rho_1}}{p^u}\right)^{j-s_\star},
\end{multline*}
using the fact that $\frac{\textsf{q}}{1-\textsf{q}} = \frac{1}{p^{u+1}}$, and ${p\choose j}\leq p^{j-s_\star} {p \choose s_\star}$. Hence for $p^{u/2}\geq 2e^{2\rho_1}$ we get
\[\sum_{\delta:\;\|\delta\|_0 \geq  s_\star + k} \frac{\omega(\delta)}{\omega(\delta_{\star})} \left(\sqrt{2}e^{2\rho_1}\right)^{\|\delta\|_0}  \leq 2 {p\choose s_\star}\left(\sqrt{2}e^{2\rho_1}\right)^{s_\star} \frac{1}{p^{\frac{uk}{2}}} \leq 2e^{s_\star(\frac{1}{2}+2\rho_1) +s_\star\log(p)- \frac{uk}{2}\log(p)}.\]
Hence  we conclude that
\begin{multline*}
\PE_\star\left[\textbf{1}_{\e}(Z) \Pi(\|\delta\|_0\geq s_\star+k\vert Z)\right] \\
\leq 2e^{s_\star \left(\frac{1}{2} +2\rho_1 +\log(p)\right)  + \frac{s_\star}{2}\log\left(1 +\frac{\bar\kappa}{\rho_1}\right)}  e^{\frac{\mathsf{a_0}}{2} +2\rho_1\|\theta_\star\|_2^2} e^{-\frac{uk}{2}\log(p)}\\
\leq 2 e^{(1+c_0)s_\star\log(p)} e^{-\frac{uk}{2}\log(p)},\end{multline*}
using (\ref{eq:cond:thm:0}). Setting $k = (2/u)(1+c_0)s_\star + j$ for some $j\geq 1$ yields the stated result. This completes the proof.
\vspace{-0.4cm}
\begin{flushright}
$\square$
\end{flushright}

\medskip

\subsection{Proof of Theorem \ref{thm1}} \label{sec:proof:thm1}
We write $\e_1$ instead of $\e_1(\bar s)$, and take $\e\subseteq\e_1$. We note that  $\cB^c =  \{\delta\in\Delta:\;\|\delta\|_0>\bar s\} \cup \F_{1} \cup \F_{2}$, where
\[\F_{1}\eqdef \bigcup_{\delta\in \Delta_{\bar s}} \{\delta\}\times\left\{\theta\in\rset^p:\; \|\theta_\delta-\theta_\star\|_2> C\epsilon\right\},\]
\and
\begin{multline*}\;\;\; \F_{2}\eqdef \bigcup_{\delta\in\Delta_{\bar s}} \{\delta\}\times\left\{\theta\in\rset^p:\;\|\theta_\delta-\theta_\star\|_2\leq C\epsilon,\;\;\mbox{ and } \;\;\|\theta-\theta_\delta\|_2>\epsilon_1 \right\},
\end{multline*}
where $\epsilon_1=\sqrt{(1+C_1)\rho_0^{-1} p}$. Therefore  we have 
\begin{multline}\label{eq:proof:thm:contrac:eq0}
\textbf{1}_{\e}(Z)\Pi(\cB^c\vert Z)  =  \textbf{1}_{\e}(Z)\Pi(\|\delta\|_0>\bar s\vert Z) 
+ \textbf{1}_{\e}(Z)\Pi(\F_{1}\vert Z) + \textbf{1}_{\e}(Z)\Pi(\F_{2}\vert Z).\end{multline}

Let $\phi$ denote the test function asserted by Lemma \ref{test} with  $M\leftarrow C$, $\Theta_\star =\{\theta_\star\}$. We can then write 
\begin{equation}\label{eq:proof:thm:contrac:eq2}
\PE_\star\left[\textbf{1}_{\e}(Z)\Pi(\F_{1}\vert Z)\right] \leq \PE_\star\left(\phi(Z)\right)+ \PE_\star\left[\textbf{1}_{\e}(Z)\left(1-\phi(Z)\right) \Pi(\F_{1}\vert Z)\right].\end{equation}
Lemma \ref{test} gives
\begin{equation}\label{eq:proof:thm:contrac:eq3}
\PE_\star\left(\phi(Z)\right) \leq \frac{2(9p)^{\bar s}e^{-\frac{C}{8}\bar \rho_1 (s_\star + \bar s)^{1/2}\epsilon}}{1-e^{-\frac{C}{8}\bar \rho_1 (s_\star + \bar s)^{1/2}\epsilon}}\leq 4 e^{-\frac{C}{32}\bar \rho_1 (s_\star + \bar s)^{1/2}\epsilon},\end{equation}
 for $(C/16)\bar\rho(\bar s+s_\star)^{1/2} \epsilon \geq 2\bar s\log(p)$. By Lemma \ref{lem:control:nc}, we have
\begin{eqnarray*}
\textbf{1}_{\e}(Z)\Pi(\F_{1}\vert Z) & \leq & \textbf{1}_{\e}(Z)\left(1 +\frac{\bar\kappa}{\rho_1}\right)^{s_\star/2} \\
&&\times \sum_{\delta\in\Delta_{\bar s}}\frac{\omega(\delta)}{\omega(\delta_{\star})}\left(\frac{\rho_1}{2\pi}\right)^{\|\delta\|_0/2}\int_{\F_\epsilon^{(\delta)}}\frac{e^{\ell(\theta;Z) - \frac{\rho_1}{2}\|\theta\|_2^2}}{e^{\ell(\theta_\star;Z)-\frac{\rho_1}{2}\|\theta_\star\|_2^2}}\mu_\delta(\rmd \theta),
\end{eqnarray*}
where $\F_{\epsilon}^{(\delta)}\eqdef\{\theta\in\rset^p:\; \|\theta_\delta-\theta_\star\|_2> C\epsilon\}$. We use this last display together with  Fubini's theorem, to conclude that
\begin{multline}\label{eq:proof:thm:contrac:eq41}
\PE_\star\left[\textbf{1}_{\e}(Z)\left(1-\phi(Z)\right) \Pi(\F_{1}\vert Z)\right]\\  \left(1 +\frac{\bar{\kappa}}{\rho_1}\right)^{s_\star/2}\sum_{\delta\in\Delta_{\bar s}}\frac{\omega(\delta)}{\omega(\delta_{\star})}\left(\frac{\rho_1}{2\pi}\right)^{\|\delta\|_0/2} \\
 \times\int_{\F_\epsilon^{(\delta)}} \PE_\star\left[(1-\phi(Z))\frac{e^{\ell(\theta;Z) }}{e^{\ell(\theta_\star;Z)}}\textbf{1}_{\e}(Z)\right] \frac{e^{-\frac{\rho_1}{2}\|\theta\|_2^2}}{e^{-\frac{\rho_1}{2}\|\theta_\star\|_2^2}} \mu_\delta(\rmd \theta).
\end{multline}
We write  $\F_\epsilon^{(\delta)} = \cup_{j\geq 1} \F^{(\delta)}_{j,\epsilon}$, where $\F_{j,\epsilon}^{(\delta)}\eqdef\{\theta\in\rset^p:\; jC\epsilon < \|\theta_\delta-\theta_\star\|_2\leq (j+1)C\epsilon\}$. Using this and Lemma \ref{test}, we have
\begin{multline}\label{eq:proof:thm:contrac:eq411}
\int_{\F_{j,\epsilon}^{(\delta)}}\PE_\star\left[(1-\phi(Z))\frac{e^{\ell(\theta;Z) }}{e^{\ell(\theta_\star;Z)}}\textbf{1}_{\e}(Z)\right] \frac{e^{-\frac{\rho_1}{2}\|\theta\|_2^2}}{e^{-\frac{\rho_1}{2}\|\theta_\star\|_2^2}} \mu_\delta(\rmd \theta)\\
\leq e^{-\frac{1}{8}\r\left(\frac{jC\epsilon}{2}\right)}\int_{\F_{j,\epsilon}^{(\delta)}}\frac{e^{-\frac{\rho_1}{2}\|\theta\|_2^2}}{e^{-\frac{\rho_1}{2}\|\theta_\star\|_2^2}}\mu_\delta(\rmd \theta).\end{multline}
We note that  $\rho_1\|\theta_\star\|_2^2 - \rho_1\|\theta\|_2^2  =-\rho_1\|\theta-\theta_\star\|_2^2 -2\rho_1\pscal{\theta_\star}{\theta-\theta_\star} \leq -\rho_1\|\theta-\theta_\star\|_2^2 + 2\rho_1\|\theta_\star\|_\infty\| \theta-\theta_\star\|_1$. Therefore, for $\theta\in\rset^p_\delta\cap \F_{j,\epsilon}^{(\delta)}$, $\rho_1\|\theta_\star\|_2^2 - \rho_1\|\theta\|_2^2   \leq -\rho_1\|\theta-\theta_\star\|_2^2  +2\rho_1 \|\theta_\star\|_\infty (\bar s+ s_\star)^{1/2}(j+1)C\epsilon$.  We deduce that  the right-hand size of (\ref{eq:proof:thm:contrac:eq411}) is upper-bounded by 
\[e^{-\frac{1}{8}\r\left(\frac{jC\epsilon}{2}\right)} e^{4\rho_1\|\theta_\star\|_\infty (\bar s+s_\star)^{1/2}\left(\frac{jC\epsilon}{2}\right)}\left(\frac{ 2\pi}{\rho_1}\right)^{\|\delta\|_0/2} \leq e^{-\frac{1}{16}\r\left(\frac{jC\epsilon}{2}\right)}\left(\frac{ 2\pi}{\rho_1}\right)^{\|\delta\|_0/2},\] 
using the condition $\bar\rho \geq 32\rho\|\theta_\star\|_\infty$. Combined with  (\ref{eq:proof:thm:contrac:eq411}) and (\ref{eq:proof:thm:contrac:eq41}) the last inequality implies that 
\begin{multline}\label{eq:proof:thm:contrac:eq42}
\PE_\star\left[\textbf{1}_{\e}(Z)\left(1-\phi(Z)\right) \Pi(\F_{1}\vert Z)\right] \leq   \left(1 +\frac{\bar{\kappa}}{\rho_1}\right)^{s_\star/2}\left(\sum_{\delta\in\Delta_{\bar s}}\frac{\omega(\delta)}{\omega(\delta_{\star})}\right)\sum_{j\geq 1}e^{-\frac{1}{16}\r\left(\frac{jC\epsilon}{2}\right)}\\
\leq \left(1 +\frac{\bar{\kappa}}{\rho_1}\right)^{s_\star/2}\left(\sum_{\delta\in\Delta_{\bar s}}\frac{\omega(\delta)}{\omega(\delta_{\star})}\right)\frac{e^{-\frac{C}{16}\bar \rho_1 (s_\star + \bar s)^{1/2}\epsilon}}{1-e^{-\frac{C}{16}\bar \rho_1 (s_\star + \bar s)^{1/2}\epsilon}}.
\end{multline}
We note ${p\choose s}\leq p^s$, so that
\begin{multline*}
\sum_{\delta\in\Delta_{\bar s}}\frac{\omega(\delta)}{\omega(\delta_{\star})} = \left(\frac{1-\q}{\q}\right)^{s_\star}\sum_{\delta\in\Delta_{\bar s}}\left(\frac{\q}{1-\q}\right)^{\|\delta\|_0} = p^{s_\star(1+u)}\sum_{s=0}^{\bar s}{p\choose s}\left(\frac{1}{p^{1+u}}\right)^{s}  \leq 2p^{s_\star(1+u)},
\end{multline*} 
provided that $p^u\geq 2$. It follows that 
\begin{multline}\label{eq:proof:thm:contrac:eq4}
\PE_\star\left[\textbf{1}_{\e}(Z)(1-\phi(Z))\Pi( \F_1\vert Z)\right]\\ \leq  2p^{s_\star(1+u)}e^{\frac{s_\star}{2}\log\left(1+\frac{\bar\kappa}{\rho_1}\right)}
  \frac{e^{-\frac{C}{16}\bar \rho_1 (s_\star + \bar s)^{1/2}\epsilon}}{1-e^{-\frac{C}{16}\bar \rho_1 (s_\star + \bar s)^{1/2}\epsilon}}  \leq 4 e^{-\frac{C}{32}\bar \rho_1 (s_\star + \bar s)^{1/2}\epsilon},
\end{multline}
provided that $(C/32)\bar\rho(s_\star +\bar s)^{1/2}\epsilon \geq s_\star(1+u)\log\left(p+\frac{p\bar\kappa}{\rho_1}\right)$.

Let $\F_{2}^{(\delta)}\eqdef\{\theta\in\rset^p:\;\|\theta_\delta-\theta_\star\|_2\leq C\epsilon,\;\mbox{ and }\|\theta-\theta_\delta\|_2>\epsilon_1\}$, so that
\[
\textbf{1}_{\e}(Z)\Pi(\F_{2}\vert Z) = \textbf{1}_{\e}(Z)\sum_{\delta\in\Delta_{\bar s}} \Pi(\delta\vert Z) \Pi(\F_{2}^{(\delta)}\vert \delta,Z),\]
and $\Pi(\F_{2}^{(\delta)}\vert \delta,Z) \leq  \PP[\|V_\delta\|_2>\epsilon_1]$, 
where $V_\delta =(V_1,\ldots,V_{p-\|\delta\|_0}) \stackrel{i.i.d.}{\sim} \textbf{N}(0,\rho_0^{-1})$. By Gaussian tails bounds  we get $\Pi(\F_{2}^{(\delta)}\vert \delta,Z) \leq 2e^{-p}$, for any constant $C_1\geq 3$. We conclude that
\begin{equation}\label{eq:proof:thm:contrac:eq1}
\textbf{1}_{\e}(Z)\Pi(\F_{2}\vert Z) \leq  \frac{1}{p^{\bar s}},
\end{equation}
for all $p$ large enough. The theorem follows by collecting the bounds (\ref{eq:proof:thm:contrac:eq1}), (\ref{eq:proof:thm:contrac:eq4}), (\ref{eq:proof:thm:contrac:eq3}), (\ref{eq:proof:thm:contrac:eq2}), and (\ref{eq:proof:thm:contrac:eq0}).
\begin{flushright}
\vspace{-0.4cm}
$\square$
\end{flushright}

\medskip

\subsection{Proof of Theorem \ref{thm:sel}}\label{sec:proof:thm:sel}
We write $\e_1$ (resp. $\e_2$) instead of $\e_1(\bar s)$ (resp. $\e_2(\bar s)$), and we fix $\e\subseteq\e_2$. First we derive a contraction rate for the frequentist estimator $\hat\theta_\delta$. To that end we note that for  $\delta\in\A_{\bar s}$, and $z\in\e_{0}$,  $\|\nabla\ell^{[\delta]}([\theta_\star]_\delta;z)\|_\infty \leq \bar\rho/2$. Furthermore, the curvature assumption on $\ell$ in $\e_{1}$ implies that
\[0\geq -\ell^{([\delta]}(\hat\theta_\delta;z)  + \ell^{([\delta]}([\theta_\star]_\delta;z) \geq \pscal{-\nabla\ell^{[\delta]}([\theta_\star]_\delta;z)}{\hat\theta_\delta - [\theta_\star]_\delta} + \frac{1}{2}\r(\|\hat\theta_\delta - [\theta_\star]_\delta\|_2).\]
Using this and the definition of $\epsilon$, it follows that for  $\delta\in\A_{\bar s}$,
\begin{equation}\label{eq:bound:freq:est}
\textbf{1}_{\e_{1}}(z) \|\hat\theta_\delta -[\theta_\star]_\delta\|_2\leq  \epsilon.\end{equation}
Set $\A_+ \eqdef \A_{\bar s}\setminus\A_{s_\star +j}$, and recall that $\cB_j = \cup_{\delta\in\A_{s_\star +j}}\{\delta\}\times \cB^{(\delta)}$. Therefore we have
\[\Pi(\cB_j\vert z)  +\Pi\left(\cup_{\delta\in\A_+}\{\delta\}\times \cB^{(\delta)}\vert z\right) + \Pi(\cB^c\vert z) = 1,\]
so that
\begin{equation}\label{proof:mod:sel:eq0}
 \textbf{1}_{\e}(z)\left(1-\Pi(\cB_j\vert z) \right)  =    \textbf{1}_{\e}(z) \Pi(\cB^c\vert z) + \textbf{1}_{\e}(z)\Pi\left(\cup_{\delta\in\A_+}\{\delta\}\times \cB^{(\delta)}\vert z\right).
\end{equation}
Hence it remains only to upper bound the last term on the right-hand side of the last display. By definition we have
\[\Pi\left(\cup_{\delta\in\A_+}\{\delta\}\times \cB^{(\delta)}\vert z\right) =  \Pi(\delta_\star\times \cB^{(\delta_\star)}\vert z)  \sum_{\delta\in\A_+} \frac{ \Pi(\delta\times \cB^{(\delta)}\vert z) }{ \Pi(\delta_\star\times \cB^{(\delta_\star)}\vert z) },\]
and
\begin{equation}\label{proof:mod:sel:eq1}
\frac{ \Pi(\delta\times \cB^{(\delta)}\vert z) }{ \Pi(\delta_\star\times \cB^{(\delta_\star)}\vert z)} = \frac{\omega(\delta)}{\omega(\delta_{\star})} \left(\frac{\rho_1}{\rho_0}\right)^{\frac{\|\delta\|_0-s_\star}{2}}\frac{\int_{\cB^{(\delta)}} e^{\ell(\theta_\delta;z) -\frac{\rho_1}{2}\|\theta_\delta\|_2^2 -\frac{\rho_0}{2}\|\theta-\theta_\delta\|_2^2}\rmd\theta}{\int_{\cB^{(\delta_\star)}} e^{\ell(\theta_{\delta_\star};z) -\frac{\rho_1}{2}\|\theta_{\delta_\star}\|_2^2 -\frac{\rho_0}{2}\|\theta-\theta_{\delta_\star}\|_2^2}\rmd\theta}.\end{equation}
By integrating out the non-selected components ($\theta-\theta_\delta$), we note that the integral in the numerator of the last display is bounded from above by
\[ (2\pi\rho_0^{-1})^{(p-\|\delta\|_0)/2} \int_{\{\theta\in\rset^p:\;\|\theta-\theta_\star\|_2\leq C\epsilon\}}e^{\ell(\theta;z) -\frac{\rho_1}{2}\|\theta\|_2^2}\mu_\delta(\rmd\theta),\]
whereas the integral in the denominator is lower bounded by
\begin{multline*}
(2\pi\rho_0^{-1})^{(p-s_\star)/2} \PP\left(\sqrt{\rho_0^{-1}}\|V\|_2  \leq C_1\epsilon_1\right) \int_{\{\theta\in\rset^p:\;\|\theta-\theta_\star\|_2 \leq C\epsilon\}}e^{\ell(\theta;z) -\frac{\rho_1}{2}\|\theta\|_2^2}\mu_{\delta_\star}(\rmd\theta)\\
\geq \frac{1}{2}(2\pi\rho_0^{-1})^{(p-s_\star)/2} \int_{\{\theta\in\rset^p:\;\|\theta-\theta_\star\|_2\leq C\epsilon\}}e^{\ell(\theta;z) -\frac{\rho_1}{2}\|\theta\|_2^2}\mu_{\delta_\star}(\rmd\theta),\end{multline*}
where $V = (V_1,\ldots,V_{p-s_\star})$ is a random vector with  i.i.d. standard normal components. These observations together with (\ref{proof:mod:sel:eq1}) lead to
\begin{equation*}
\frac{ \Pi(\delta\times \cB^{(\delta)}\vert z) }{ \Pi(\delta_\star\times \cB^{(\delta_\star)}\vert z)} \leq 
\frac{2\omega(\delta)}{\omega(\delta_{\star})} \left(\frac{\rho_1}{2\pi}\right)^{\frac{\|\delta\|_0-s_\star}{2}}\frac{\int_{\{\theta\in\rset^p:\;\|\theta-\theta_\star\|_2\leq C\epsilon\}}e^{\ell(\theta;z) -\frac{\rho_1}{2}\|\theta\|_2^2}\mu_\delta(\rmd\theta)}{\int_{\{\theta\in\rset^p:\;\|\theta-\theta_\star\|_2\leq C\epsilon\}}e^{\ell(\theta;z) -\frac{\rho_1}{2}\|\theta\|_2^2}\mu_{\delta_\star}(\rmd\theta)}.
\end{equation*}
For $\theta\in\rset^p_\delta$, $\delta\in\A_{\bar s}$, and $\|\theta-\theta_\star\|_2\leq C\epsilon$, it is easily checked that
\[-C\|\theta_\star\|_\infty\rho_1\bar s^{1/2}\epsilon \leq  \frac{\rho_1}{2}\left(\|\theta_\star\|_2^2 - \|\theta\|_2^2\right) \leq C\|\theta_\star\|_\infty\rho_1\bar s^{1/2}\epsilon,\]
and by the definition of $\varpi$, and noting from (\ref{eq:bound:freq:est})  that $\|[\theta]_\delta-\hat\theta_\delta\|_2\leq \|[\theta]_\delta-[\theta_\star]_\delta\|_2 + \|\hat\theta_\delta - [\theta_\star]_\delta\|_2 \leq (C+1)\epsilon$, we have
\begin{multline*}
\left|\ell^{[\delta]}(\theta;z) - \ell^{[\delta]}(\hat\theta_\delta;z) -\underbrace{\pscal{\nabla\ell^{[\delta]}(\hat\theta_\delta;z)}{[\theta]_\delta-\hat\theta_\delta}}_{=0} +\frac{1}{2}([\theta]_\delta-\hat\theta_\delta)'\mathcal{I}_\delta([\theta]_\delta-\hat\theta_\delta)\right| \\
\leq \frac{\varpi(\delta,(C+1)\epsilon;z)}{6}\bar s^{3/2}\|[\theta]_\delta-\hat\theta_\delta\|_2^3\leq \bar s^{3/2} \frac{\mathsf{a}_2}{6}((C+1)\epsilon)^3.\end{multline*}
We conclude that
\begin{multline*}
\frac{ \Pi(\delta\times \cB^{(\delta)}\vert z) }{ \Pi(\delta_\star\times \cB^{(\delta_\star)}\vert z)} \leq 2e^{C_0(\rho_1\|\theta_\star\|_\infty\bar s^{1/2}\epsilon + \mathsf{a}_2\bar s^{3/2}\epsilon^3)} \\
\times \frac{\omega(\delta)}{\omega(\delta_{\star})} \left(\frac{\rho_1}{2\pi}\right)^{\frac{\|\delta\|_0-s_\star}{2}}\frac{e^{\ell^{[\delta]}(\hat\theta_\delta;z)}}{e^{\ell^{[\delta_\star]}(\hat\theta_{\delta_\star};z)}}  \frac{\sqrt{\det\left(2\pi \mathcal{I}_\delta^{-1}\right)}}{\sqrt{\det\left(2\pi \mathcal{I}_{\delta_\star}^{-1}\right)} \mathbf{N}(\hat\theta_{\delta_\star};\mathcal{I}_{\delta_\star}^{-1})(\cB_{\delta_\star})},
\end{multline*}
for some absolute constant $C_0$, where $\cB_\delta =\{u\in\rset^{\|\delta\|}:\; \|u-[\theta_\star]_\delta\|_2\leq C\epsilon\}$, and $\mathbf{N}(\hat\theta_{\delta};\mathcal{I}^{-1}_{\delta})(A)$ denotes the probability of  $A$ under the Gaussian distribution  $\mathbf{N}(\hat\theta_{\delta};\mathcal{I}^{-1}_{\delta})$. For $z\in\e_1$, using the assumption $(C-1)\epsilon \underline{\kappa}^{1/2} \geq 2(s_\star^{1/2}+1)$, and for $z\in\e_1$, we have $\mathbf{N}(\hat\theta_{\delta_\star};\mathcal{I}_{\delta_\star}^{-1})(\cB_{\delta_\star})\geq 1/2$. We conclude that
\begin{multline}\label{proof:mod:sel:eq2}
\textbf{1}_{\e_{1}}(z)\frac{ \Pi(\delta\times \cB^{(\delta)}\vert z) }{ \Pi(\delta_\star\times \cB^{(\delta_\star)}\vert z)} \leq  4e^{C_0(\rho_1\|\theta_\star\|_\infty\bar s^{1/2}\epsilon + \mathsf{a}_2\bar s^{3/2} \epsilon^3)} \frac{\omega(\delta)}{\omega(\delta_{\star})} \left(\rho_1\right)^{\frac{\|\delta\|_0-s_\star}{2}}\frac{e^{\ell(\hat\theta_\delta;z)}}{e^{\ell(\hat\theta_{\delta_\star};z)}} \sqrt{\frac{\det(\mathcal{I}_{\delta_\star})}{\det(\mathcal{I}_\delta)}}. \end{multline}
%

For $z\in\e_{2}$, and $\|\delta\|_0=s_\star + j$,  we have
\[\ell(\hat\theta_\delta;z) - \ell(\hat\theta_{\delta_\star};z) \leq \frac{j u}{2}\log(p).\]
Recall that $\mathcal{I}_\delta = -\nabla^{(2)}\ell^{[\delta]}(\hat\theta_\delta;z)$. Hence we can write
 \[\frac{\det(\mathcal{I}_{\delta_\star})}{\det(\mathcal{I}_\delta)}  = \frac{\det\left(-\nabla^{(2)}\ell^{[\delta_\star]}(\hat\theta_{\delta_\star};z)\right)}{\det\left(-\nabla^{(2)}\ell^{[\delta]}(\hat\theta_{\delta_\star};z)\right)}\times \frac{\det\left(-\nabla^{(2)}\ell^{[\delta]}(\hat\theta_{\delta_\star};z)\right)}{\det\left(-\nabla^{(2)}\ell^{[\delta]}(\hat\theta_\delta;z)\right)}.\]
 The Cauchy interlacing property (Lemma \ref{lem:interlacing}) implies that the first term on the right hand side of the  last display is upper bounded by $(1/\underline{\kappa})^j$. To bound the second term,  we first note that by convexity of the function $-\log\det$, for any pair of symmetric positive definite matrices $A,B$ of same size, it holds $|\log\det(A) - \log\det(B)| \leq \max(\|A^{-1}\|_{\textsf{F}},\|B^{-1}\|_{\textsf{F}})\|A-B\|_{\textsf{F}}$, where $\|M\|_{\textsf{F}}$ denotes the Frobenius norm of $M$.  Hence, if a symmetric positive definite matrix $A(\theta)$ depends smoothly on a parameter $\theta$, then we have $|\log\det(A(\theta)) -\log\det(A(\theta_0))| \leq \sup_{u\in\Theta} \|A(u)^{-1}\|_{\textsf{F}} \;\|\nabla A(\bar\theta)\cdot(\theta-\theta_0)\|_{\textsf{F}}$, for some $\bar\theta$ on the segment between $\theta$ and $\theta_0$. We use this together with the definition of $\mathsf{a}_2$,  to conclude that the second term on the right hand of the last equation is upper bounded by $e^{\frac{2\mathsf{a}_2\bar s^{3}\epsilon}{\underline{\kappa}}}$. Hence
 \[\frac{\det(\mathcal{I}_{\delta_\star})}{\det(\mathcal{I}_\delta)}   \leq  \left(\frac{1}{\underline{\kappa}}\right)^j e^{\frac{2\mathsf{a}_2\bar s^{3}\epsilon}{\underline{\kappa}}}.\]
Using these bounds,  we obtain from (\ref{proof:mod:sel:eq2}),
\begin{equation}\label{proof:mod:sel:eq3}
\textbf{1}_{\e}(z)\frac{ \Pi(\delta\times \cB^{(\delta)}\vert z) }{ \Pi(\delta_\star\times \cB^{(\delta_\star)}\vert z)} \leq  4e^{C_0(\rho_1\|\theta_\star\|_\infty\bar s^{1/2}\epsilon + \mathsf{a}_2\bar s^{3/2} (\epsilon^3+ \frac{\bar s^{1/2}\epsilon}{\underline{\kappa}}))}\left(\sqrt{\frac{\rho_1}{\underline{\kappa}}}\frac{1}{p^{1 +\frac{u}{2}}}\right)^j. \end{equation}
Using (\ref{proof:mod:sel:eq3}) and summing over $\delta\in\A_+$, it follows that 
\begin{multline*}
\textbf{1}_{\e}(z)\Pi\left(\cup_{\delta\in\A_+}\{\delta\}\times \cB^{(\delta)}\vert z\right) \\
\leq  4e^{C_0(\rho_1\|\theta_\star\|_\infty\bar s^{1/2}\epsilon + \mathsf{a}_2\bar s^{3/2} (\epsilon^3+ \frac{\bar s^{1/2}\epsilon}{\underline{\kappa}}))} \sum_{j=k+1}^{\bar s-s_\star} \;\;\sum_{\delta\supseteq\delta_\star,\;\|\delta\|_0 = s_\star +j} \left(\sqrt{\frac{\rho_1}{\underline{\kappa}}}\frac{1}{p^{1 +\frac{u}{2}}}\right)^j,\\
 \leq  8e^{C_0(\rho_1\|\theta_\star\|_\infty\bar s^{1/2}\epsilon + \mathsf{a}_2\bar s^{3/2} (\epsilon^3+ \frac{\bar s^{1/2}\epsilon}{\underline{\kappa}}))}  \left(\sqrt{\frac{\rho_1}{\underline{\kappa}}}\frac{1}{p^{\frac{u}{2}}}\right)^{k+1},
\end{multline*}
provided  that $p^{u/2} \sqrt{\underline{\kappa}/\rho_1}\geq 2$.  This bound and (\ref{proof:mod:sel:eq0}) yields the stated bound.

\begin{remark}
By tracing the steps in the proof of (\ref{proof:mod:sel:eq3}), it can be checked that the following lower bound also holds.
\begin{equation}\label{proof:mod:sel:eq4}
\textbf{1}_{\e_1}(z)\frac{ \Pi(\delta\times \cB^{(\delta)}\vert z) }{ \Pi(\delta_\star\times \cB^{(\delta_\star)}\vert z)} \geq  \frac{1}{4}e^{-C_0(\rho_1\|\theta_\star\|_\infty\bar s^{1/2}\epsilon + \mathsf{a}_2\bar s^{3/2} (\epsilon^3+ \frac{\bar s^{1/2}\epsilon}{\underline{\kappa}}))}\left(\sqrt{\frac{\rho_1}{\bar \kappa}}\frac{1}{p^{u+1}}\right)^j. \end{equation}
\end{remark}
\vspace{-0.6cm}
\begin{flushright}
$\square$
\end{flushright}

\medskip

\subsection{Proof of Theorem \ref{thm:KL}}\label{sec:proof:thm:KL}
We start with the following  general observation.  Let $\pi$, $q$, and $\mu$ be three probability measures on some measurable space such that $\mu(\rmd x) = \frac{e^{f(x)} \pi(\rmd x)\textbf{1}_A(x)}{\int_A e^{f(u)}\pi(\rmd u)}$ for some measurable $\rset$-valued function $f$, and a measurable set $A$ such that $\pi(A)\geq 1/2$. Furthermore, suppose that the support of $q$ is $A$.
Then
\[\int \log\left(\frac{\rmd\mu}{\rmd\pi}\right)\rmd q =\int_A f\rmd q - \log\left(\int_A e^f\rmd\pi\right).\]
By Jensen's inequality we have
\[-\log\left(\int_A e^{f}\rmd\pi\right) \leq -\log(\pi(A)) -\int_A f \frac{\rmd\pi}{\pi(A)}.\]
Since $-\log(1-x)\leq 2x$ for $x\in [0,1/2]$, we have $-\log(\pi(A)) \leq 2\pi(A^c)$, and we conclude that
\begin{eqnarray}\label{bound:KL}
\int \log\left(\frac{\rmd\mu}{\rmd\pi}\right)\rmd q & \leq & \left|\int_A f\rmd q - \int_A f\rmd \pi\right|  +  2\pi(A^c)\left(1+ \int_A |f|\rmd \pi \right)\nonumber\\
& \leq & \int_A|f|\rmd q + 2\int_A|f|\rmd\pi + 2\pi(A^c).\end{eqnarray}
When $q=\mu$, (\ref{bound:KL}) writes
\begin{equation}\label{bound:KL:2}
\KL{\mu}{\pi} \leq  \int_A|f|\rmd \mu + 2\int_A|f|\rmd\pi + 2\pi(A^c).\end{equation}
Let us now apply (\ref{bound:KL}) and (\ref{bound:KL:2}). Fix  $z\in\e$. In order to use these bounds, we first note that the density of $\Pi_\star^{(\infty)}$   with respect to $\Pi$ that can be written as
\begin{equation}\label{eq:density:bvm}
\frac{\rmd\Pi^{(\infty)}_\star}{\rmd\Pi}(\delta,\theta\vert z) = \frac{e^{-R(\delta,\theta;z)}\textbf{1}_{ \{\delta_\star\}\times \rset^p}(\delta,\theta)}{\int_{ \{\delta_\star\}\times \rset^p} e^{-R(\delta,\theta;z)}\Pi(\rmd\delta,\rmd\theta\vert z)},\end{equation}
 where
 \begin{eqnarray*}
 R(\delta,\theta;z) &  \eqdef &  \ell(\theta_\delta;z)-\frac{\rho_1}{2}\|\theta_\delta\|_2^2 - \ell(\hat\theta_\delta;z)+ \frac{\rho_1}{2}\|\hat\theta_\delta\|_2^2 + \frac{1}{2}([\theta]_\delta-\hat\theta_\delta)'\mathcal{I}_{\delta}([\theta]_\delta-\hat\theta_\delta),\\
 & = & -\frac{\rho_1}{2}\|\theta_\delta\|_2^2 +  \frac{\rho_1}{2}\|\hat\theta_\delta\|_2^2 + \frac{1}{6}\nabla^{(3)}\ell^{[\delta]}(\bar \theta_\delta;z)\cdot\left([\theta]_\delta-\hat\theta_\delta,[\theta]_\delta-\hat\theta_\delta,[\theta]_\delta-\hat\theta_\delta\right),\end{eqnarray*}
 for some element $\bar\theta_\delta$ on the segment between $[\theta]_\delta$ and $\hat\theta_\delta$. The second equality follows from Taylor expansion and $\nabla\ell^{[\delta]}(\hat\theta_\delta;z)=0$. That second expression of $R$ shows that for $z\in\e$, $\delta\in\A_{\bar s}$, and $\theta\in\cB^{(\delta)}$, 
 \begin{equation}\label{bound:R1}
 \left| R(\delta,\theta)\right| \leq C_0\rho_1 \bar s^{1/2} \epsilon + C_0 \mathsf{a}_2\bar s^{3/2}\epsilon^3,\end{equation}
for some absolute constant $C_0$.  However, in general when $\theta\notin\cB^{(\delta)}$, $R(\delta,\theta)$ is quadratic in $\theta$ under the assumptions of the theorem. Indeed, using $\nabla\ell^{[\delta]}(\hat\theta_\delta;z)=0$,  we can write that $\ell(\theta_\delta;z) - \ell^{[\delta]}(\hat\theta_\delta;z) = -(1/2)([\theta]_\delta-\hat\theta_\delta)' [-\nabla^{(2)} \ell^{[\delta]} (\bar \theta_\delta;z)]([\theta]_\delta-\hat\theta_\delta)$, for some element $\bar\theta_\delta$ on the segment between $[\theta]_\delta$ and $\hat\theta_\delta$. Hence, for $\theta\in\rset^p$
\begin{multline}\label{bound:R2}
|R(\delta,\theta)| \leq \frac{\rho_1}{2}\left| \|\theta_\delta\|_2^2 -\|\hat\theta_\delta\|_2^2\right| \\
+\frac{1}{2}\left|([\theta]_\delta-\hat\theta_\delta)' [-\nabla^{(2)} \ell^{[\delta]} (\bar \theta_\delta;z)([\theta]_\delta-\hat\theta_\delta) -  ([\theta]_\delta-\hat\theta_\delta)'\mathcal{I}_{\delta}([\theta]_\delta-\hat\theta_\delta)\right| \\
\leq \frac{\rho_1 +\bar \kappa}{2}\|[\theta]_\delta-\hat\theta_\delta\|_2^2 +\rho_1 \|\hat\theta_\delta\|_2\|[\theta]_\delta- \hat\theta_\delta\|_2 \\
\leq (\rho_1 +\bar \kappa) \|[\theta]_\delta-\hat\theta_\delta\|_2^2 + \frac{\rho_1^2(\epsilon + \|\theta_\star\|_2)^2}{2(\rho_1+\bar \kappa)},\end{multline}
where the second inequality uses (\ref{def:kappa:bar:bis}), and the third inequality follows from some basic algebra, and (\ref{eq:bound:freq:est}). 

Let  $R$ be some arbitrary probability measure on $\Delta\times \rset^p$ with support $\{\delta_\star\}\times\rset^p$.  We make use of (\ref{bound:KL}) with $q=R$, $\mu=\Pi_\star^{(\infty)}$, $\pi=\Pi$, and $A = \{\delta_\star\}\times \rset^p$.  We then split the integrals over $\{\delta_\star\}\times\rset^p$ into $\{\delta_\star\}\times\cB^{(\delta_\star)}$ and $\{\delta_\star\}\times(\rset^p\setminus \cB^{(\delta_\star)})$, together with (\ref{bound:R1}) and (\ref{bound:R2}) to get
\begin{multline}\label{bound:KL:eq1}
\textbf{1}_{\e}(z) \int \log\left(\frac{\rmd \Pi_\star^{(\infty)}}{\rmd\Pi}\right)\rmd R \leq  2\textbf{1}_{\e}(z)\left(1-\Pi(\delta_\star \vert z)\right)\\
+ C_0 \left(\rho_1\bar s^{1/2}\epsilon + \mathsf{a}_2\bar s^{3/2}\epsilon^3\right)+\frac{3\rho_1^2(\epsilon + \|\theta_\star\|_2)^2}{2(\rho_1+\bar \kappa)}\\
 + (\rho_1+\bar\kappa)\textbf{1}_{\e}(z)\int_{ \{\delta_\star\}\times \rset^p\setminus\cB^{(\delta_\star)}} \|[\theta]_\delta - \hat\theta_\delta\|_2^2R(\rmd\delta,\rmd\theta)\\
+ 2(\rho_1+\bar\kappa) \textbf{1}_{\e}(z)\int_{ \{\delta_\star\}\times \rset^p\setminus\cB^{(\delta_\star)}} \|[\theta]_\delta - \hat\theta_\delta\|_2^2\Pi(\rmd\delta,\rmd\theta\vert Z).\end{multline}
By (\ref{def:kappa:min:KL}), (\ref{def:kappa:bar:bis}) and  Lemma \ref{bound:l2:dev}, the  last integral in the last display is bounded from above by
\[(C-1)^2\epsilon^2\left(\frac{\rho_1+\bar\kappa}{\rho_1+\underline{\kappa}}\right)^{\frac{s_\star}{2}} e^{-\frac{(C-1)^2\epsilon^2\underline{\kappa}}{32}} + 2e^{-p},\]
provided that $\underline{\kappa}(C-1)\epsilon \geq 4\max(\sqrt{s_\star\underline{\kappa}},\rho_1(\epsilon+ s_\star^{1/2}\|\theta_\star\|_\infty))$. We conclude that
 \begin{multline}\label{bound:KL:eq2}
\textbf{1}_{\e}(z) \int \log\left(\frac{\rmd \Pi_\star^{(\infty)}}{\rmd\Pi}\right)\rmd R \leq   C_0 \left(\rho_1\bar s^{1/2}\epsilon + \mathsf{a}_2\bar s^{3/2}\epsilon^3\right)
+\frac{3\rho^2(\epsilon + \|\theta_\star\|_2)^2}{2(\rho_1+\bar \kappa)} \\
+C_0(\rho_1 +\bar\kappa)\epsilon^2\left(\frac{\rho_1+\bar\kappa}{\rho_1+\underline{\kappa}}\right)^{\frac{s_\star}{2}} e^{-\frac{(C-1)^2\epsilon^2\underline{\kappa}}{32}} + 2(\rho_1 +\bar\kappa)e^{-p} +2 \textbf{1}_{\e}(z) (1-\Pi(\delta_\star \vert z)) \\
+  (\rho_1+\bar\kappa)\textbf{1}_{\e}(z) \int_{ \{\delta_\star\}\times \rset^p\setminus\cB^{(\delta_\star)}} \|[\theta]_\delta - \hat\theta_\delta\|_2^2R(\rmd\delta,\rmd\theta).\end{multline}
In the particular case where $R=\Pi_\star^{(\infty)}$, Lemma \ref{bound:l2:dev} gives
\begin{equation}\label{bound:moment:Q}
\int_{ \{\delta_\star\}\times \rset^p\setminus\cB^{(\delta_\star)}} \|[\theta]_\delta - \hat\theta_\delta\|_2^2R(\rmd\delta,\rmd\theta)\leq (C-1)^2\epsilon^2\left(\frac{\bar\kappa}{\underline{\kappa}}\right)^{\frac{s_\star}{2}} e^{-\frac{(C-1)^2\epsilon^2\underline{\kappa}}{32}}.\end{equation}
The result follows by plugging the last inequality in (\ref{bound:KL:eq2}). We note that the last display also holds true if $R=\tilde\Pi_\star^{(\infty)}$.
\vspace{-0.6cm}
\begin{flushright}
$\square$
\end{flushright}

\subsection{Proof of Theorem \ref{thm:VA}}\label{sec:proof:thm:VA}
We introduce 
\[\tilde Q (\delta,\rmd \theta) \propto \tilde Q(\delta)  e^{-\frac{1}{2}(\theta - \hat\theta_\star)' \left(\S\cdot \bar{\mathcal{I}}\right) (\theta-\hat\theta_\star)}  \rmd\theta,\]
for some arbitrary distribution $\tilde Q$ on $\Delta$ of the form $\tilde Q(\delta) = \prod_{j=1}^p \alpha_j^{\delta_j}(1-\alpha_j)^{1-\delta_j}$, where $\alpha_j=\alpha$ if $\delta_{\star j}=1$, and $\alpha_j=1-\alpha$ otherwise, for some $\alpha\in(0,1)$.  Note that $\tilde Q\in\mathcal{Q}$,  and  $\|\tilde Q - \tilde\Pi^{(\infty)}_\star\|_\tv\to 0$, as $\alpha\to 1$. 

The strong convexity of the KL-divergence (Lemma \ref{lem:strong:conv:kl}) allows us to write, for any $t\in (0,1)$,
\[t\KL{Q}{\Pi} + (1-t) \KL{\tilde Q}{\Pi} \geq \KL{tQ + (1-t)\tilde Q}{\Pi} + \frac{t(1-t)}{2}\|\tilde Q-Q\|_\tv^2.\]
This implies that
\[\frac{t(1-t)}{2}\|\tilde Q-Q\|_\tv^2 \leq \KL{\tilde Q}{\Pi} + t\left(\KL{Q}{\Pi} - \KL{\tilde Q}{\Pi}\right)\leq \KL{\tilde Q}{\Pi},\]
where the second inequality uses the fact that $\tilde Q\in\mathcal{Q}$, and $Q$ is the minimizer of the KL-divergence over that family. Hence with $t=1/2$ we have
\begin{eqnarray*}
\|Q - \tilde\Pi_\star^{(\infty)}\|_\tv^2 & \leq & 2 \|Q - \tilde Q\|_\tv^2 +  2\|\tilde Q-\tilde\Pi_\star^{(\infty)}\|_\tv^2  \\
 & \leq & 16 \KL{\tilde Q}{\Pi} + 2\|\tilde Q-\tilde\Pi_\star^{(\infty)}\|_\tv^2,\end{eqnarray*}
 where the second inequality uses the bound on $\|\tilde Q-Q\|_\tv^2$ obtained above.
\begin{multline*}
\KL{\tilde Q}{\Pi} = \int\log\left(\frac{\rmd \tilde Q}{\rmd \Pi}\right)\rmd \tilde Q \\
= \int_{(\delta_\star\times\rset^p)^c} \log\left(\frac{\rmd \tilde Q}{\rmd \Pi}\right)\rmd \tilde Q +  \int_{\delta_\star\times\rset^p} \log\left(\frac{\rmd \tilde Q}{\rmd \Pi}\right)\rmd \tilde Q.\end{multline*}
We note that $\tilde\Pi_\star^{(\infty)}$ is precisely the restriction of  $\tilde Q$ on $\{\delta_\star\}\times \rset^p$. Therefore,  on $\{\delta_\star\}\times \rset^p$, the density $\frac{\rmd \tilde Q}{\rmd \Pi}$ can be written as
\[\frac{\rmd \tilde Q}{\rmd \Pi} = \tilde Q(\{\delta_\star\}\times \rset^p) \frac{\rmd \tilde\Pi_\star^{(\infty)}}{\rmd \Pi_\star^{(\infty)}} \frac{\rmd \Pi_\star^{(\infty)}}{\rmd \Pi} .\]
Hence
\begin{multline*}
\int_{\delta_\star\times\rset^p} \log\left(\frac{\rmd \tilde Q}{\rmd \Pi}\right)\rmd \tilde Q \leq 
\KL{\tilde\Pi_\star^{(\infty)}}{\Pi_\star^{(\infty)}} + \tilde Q(\delta_\star) \int_{\delta_\star\times\rset^p} \log\left(\frac{\rmd \Pi^{(\infty)}_\star}{\rmd \Pi}\right)\rmd\tilde\Pi_\star^{(\infty)}.
\end{multline*}
On the other hand,
\begin{multline}\label{bound:VA:eq5}
\int_{(\delta_\star\times\rset^p)^c} \log\left(\frac{\rmd \tilde Q}{\rmd \Pi}\right)\rmd \tilde Q \\
= \sum_{\delta\neq\delta_\star}\tilde Q(\delta)\left[\log\left(\frac{\tilde Q(\delta)}{\Pi(\delta\vert z)}\right) + \int \log\left(\frac{\tilde Q(\theta)}{\Pi(\theta\vert \delta,z)}\right)\tilde Q(\theta)\rmd \theta\right]\\
\leq \left(1- \tilde Q(\delta_\star)\right) \max_{\delta\in\Delta} \left[-\log(\Pi(\delta\vert z)) + \int \log\left(\frac{\tilde Q(\theta)}{\Pi(\theta\vert \delta,z)}\right)\tilde Q(\theta)\rmd \theta\right].\end{multline}
Collecting all the terms we obtain
\begin{multline*}
\|Q - \tilde\Pi_\star^{(\infty)}\|_\tv^2 \leq 16 \KL{\tilde\Pi_\star^{(\infty)}}{\Pi_\star^{(\infty)}} +2\|\tilde Q-\tilde\Pi_\star^{(\infty)}\|_\tv^2 \\
+ 16 \tilde Q(\delta_\star) \int_{\delta_\star\times\rset^p} \log\left(\frac{\rmd \Pi^{(\infty)}_\star}{\rmd \Pi}\right)\rmd\tilde\Pi_\star^{(\infty)} \\
+16\left(1- \tilde Q(\delta_\star)\right) \max_{\delta\in\Delta} \left[-\log(\Pi(\delta\vert z)) + \int \log\left(\frac{\tilde Q(\theta)}{\Pi(\theta\vert \delta,z)}\right)\tilde Q(\theta)\rmd \theta\right].
\end{multline*}
Letting $\alpha\to 1$ on both sides  yields
\[
\|Q - \Pi_\star^{(\infty)}\|_\tv^2 \leq 16 \KL{\tilde\Pi_\star^{(\infty)}}{\Pi_\star^{(\infty)}} + 16 \int_{\delta_\star\times\rset^p} \log\left(\frac{\rmd \Pi^{(\infty)}_\star}{\rmd \Pi}\right)\rmd\tilde\Pi_\star^{(\infty)}.
\]
Using Lemma \ref{lem:KL:Gaussian}, we have
\[\KL{\tilde\Pi_\star^{(\infty)}}{\Pi_\star^{(\infty)}}  = \frac{\zeta}{2},\]
where $\zeta = \log\left(\frac{\det(\bar{\mathcal{I}})}{\det(\S\cdot\bar{\mathcal{I}})}\right)  + \textsf{Tr}\left(\bar{\mathcal{I}}^{-1}(\S\cdot \bar{\mathcal{I}})\right) - p$. Hence the theorem.
\vspace{-0.6cm}
\begin{flushright}
$\square$
\end{flushright}

\subsection{Proof of Corollary \ref{coro:lm}}\label{proof:coro:lm}
\paragraph{\underline{\tt On the event $\G$}}\;\; We first constructed the event $\G$.  Let $\tau_\Sigma \eqdef \max_j \Sigma_{jj}$. For $c_1=5$, $c_2=1/4$, and $c_3=9$, for $j=1,\ldots,p+1$, we set $\G \eqdef \bigcap_{j=1}^{p+1}\H^{(j)}$, where
\begin{multline*}
\H^{(j)}\eqdef\left\{ Z\in\rset^{n\times (p+1)}:\; \max_{1\leq k\leq p,\;k\neq j} \left|\frac{\|Z_k\|_2^2}{n}-\Sigma_{jj}\right| \leq c_1\tau_\Sigma\right.\\
\left. \mbox{ for all }v\in\rset^p:\;\frac{\|X^{(j)}v\|_2}{\sqrt{n}} \geq c_2\|\Sigma^{1/2}v\|_2 -c_3\tau_\Sigma\sqrt{\frac{\log(p)}{n}}\|v\|_1\right\}.\end{multline*}
When B\ref{H:lin:mod} holds, by Theorem  1 of \cite{raskutti:etal:10} and Lemma 1 of \cite{ravikumaretal11} there exist absolute positive constant $c_4,c_5$ such that 
\[\PP(Z\notin \G)\leq 4(p+1)e^{-n/128} + c_4(p+1)e^{-c_5 n} \to 0,\]
as $p\to\infty$, provided that $n\geq (256/\min(1,128 c_5))\log(p)$.  In what follows we will assume that $n$ satisfies
\begin{equation}\label{ss:cond:lm}
n\geq \frac{256}{\min(1,128 c_5)}\log(p),\;\;\mbox{ and } n\geq \left(\frac{16 c_3\tau_\Sigma}{c_2\lambda_{\textsf{min}}^{1/2}(\Sigma)}\right)^2 \left[\max_j 2s^{(j)}_\star \left(1+\frac{6}{u}\right)+\frac{4}{u}\right] \log(p).\end{equation}

\paragraph{\underline{\tt Problem set up and posterior sparsity}}\;\; 
For any $j$ we can partition $Z$ as $Z = [Y^{(j)},X^{(j)}]$, and under B\ref{H:lin:mod},
\begin{equation}\label{cond:model:lm}
Y^{(j)} = X^{(j)}\theta_\star^{(j)} + \frac{1}{\sqrt{[\vartheta_\star]_{jj}}} V^{(j)},\;\;\mbox{ where } V^{(j)}\vert X^{(j)}\sim \textbf{N}_n(0,I_n).\end{equation}
The quasi-likelihood of the $j$-th regression is $\ell^{(j)}(u;z) = (1/2\sigma_j^2)\|Y^{(j)} - X^{(j)}u\|_2^2$. The resulting quasi-posterior distribution $\Pi^{(j)}(\cdot\vert Z)$ on $\Delta\times\rset^p$ fits squarely in the framework developed in the paper, and we will successively apply to it the different general theorems obtained above. However to keep the notation simple, and when there is no risk of confusion, we shall omit the index $j$ from the various quantities. For instance we will $Y$ instead of $Y^{(j)}$, $X$ instead of $X^{(j)}$, etc... 

From the expression of the quasi-likelihood, we have
\[\nabla \ell(\theta_\star;Z) = \frac{1}{\sigma^2}X'(Y-X \theta_\star),\]
and
\[\L_{\theta_\star}(u;Z) = -\frac{n}{2\sigma^2} (u-\theta_\star)'\left(\frac{X'X}{n}\right)(u-\theta_\star),\;\;u\in\rset^p,\]
which does not depend on $Y$. Let us first apply  Theorem \ref{thm:0}. We set 
\begin{multline*}
\G_1 \eqdef \H\bigcap \left\{Z = [Y^{(j)},X^{(j)}]\in\rset^{n\times (p+1)}:\;\right.\\
\left. \max_{1\leq k\leq p,\;k\neq j}\left|\pscal{X_k}{Y^{(j)}-X^{(j)}\theta^{(j)}_\star}\right|\leq \sqrt{\frac{6\tau_\Sigma}{[\vartheta_{\star}]_{jj}}(1+c_1)n\log(p)}\right\}.\end{multline*}
We  set
\[\bar\rho =\frac{2}{\sigma^2_j}\sqrt{\frac{6\tau_\Sigma}{[\vartheta_{\star}]_{jj}} (1+c_1)n\log(p)},\;\;\;\bar\kappa  = (n/\sigma^2)(1+c_1)s_\star^{(j)}\tau_\Sigma.\]
We stress again that these quantities and events are specific to the $j$-th regression. 
From the expressions of $\nabla\ell(\theta_\star;z)$, and $\L_{\theta_\star}(\theta;z)$, it is straightforward to check that $\G_1\subseteq\e_0$ if we define $\e_0$ in H\ref{H1} by taking $\bar\rho$ and $\bar\kappa$ as above.  We also note that by the choice of $\rho_1$ and the conditions $\|\theta_\star\|_\infty= O(1)$,   we have $32\|\theta_\star\|_\infty\rho_1 \leq \bar\rho$ for all $p$ large enough.  To apply Theorem \ref{thm:0}, it only remains to check  (\ref{eq:curvature:cond:thm0}).  With $\G_1$ and $\mathcal{L}_{\theta_\star}$ as defined above, we have
\begin{multline}\label{res:conc:lm}
\PE_\star\left[\textbf{1}_{\G_1}(Z)e^{\L_{\theta_\star}(u;Z)  + \left(1-\frac{\rho_1}{\bar\rho}\right)\pscal{\nabla\ell(\theta_\star;Z)}{u-\theta_\star}}\right] \\
\leq \PE_\star\left[\textbf{1}_{\H}(X)e^{-\frac{n}{2\sigma^2} (u-\theta_\star)'\left(\frac{X'X}{n}\right)(u-\theta_\star)} \PE_\star\left(e^{\frac{1}{\sigma^2}\left(1-\frac{\rho_1}{\bar\rho}\right)(Y-X\theta_\star)'X(u-\theta_\star)}\vert X\right)\right]\\
=\PE_\star\left[\textbf{1}_{\H}(X)e^{-\frac{n}{2\sigma^2}\left(1- \frac{\left(1-\frac{\rho_1}{\bar\rho}\right)^2}{\sigma^2\vartheta_{\star,11}}\right) (u-\theta_\star)'\left(\frac{X'X}{n}\right)(u-\theta_\star)}\right],\end{multline}
where the equality uses the moment generating function of the conditionally Gaussian random variable $V$. For $u\in\rset^p$ such that $\|\delta_\star^c\cdot(u-\theta_\star)\|_1 \leq 7\|\delta_\star\cdot(u-\theta_\star)\|_1$, and for $Z\in\G$, we have
\[\frac{1}{\sqrt{n}}\|X(u-\theta_\star)\|_2 \geq c_2\lambda_{\textsf{min}}(\Sigma)^{1/2}\|u-\theta_\star\|_2 -8c_3s_\star^{1/2}\tau_\Sigma\sqrt{\frac{\log(p)}{n}} \|(\delta_\star\cdot(u-\theta_\star)\|_2.\]
It follows that
\[(u-\theta_\star)'\left(\frac{X'X}{n}\right)(u-\theta_\star) \geq \frac{c_2^2}{4}\lambda_{\textsf{min}}(\Sigma) \|\delta_\star\cdot(u-\theta_\star)\|_2^2,\]
if the sample size $n$ satisfies 
\[n\geq \left(\frac{16 c_3\tau_\Sigma}{c_2\lambda_{\textsf{min}}^{1/2}(\Sigma)}\right)^2 s_\star \log(p).\]
Therefore, 
Since $ \sigma^2[\vartheta_{\star}]_{jj}\geq 1$, we conclude from (\ref{res:conc:lm}) that (\ref{eq:curvature:cond:thm0}) holds with 
\[r_0(x) = \frac{nc_2^2\lambda_{\textsf{min}}(\Sigma)}{4\sigma^2}\left(1- \left(1-\frac{\rho_1}{\bar\rho}\right)^2\right)   x^2 \geq \frac{nc_2^2\lambda_{\textsf{min}}(\Sigma)}{4\sigma^2}\frac{\rho_1}{\bar\rho}  x^2,\]
and hence 
\[\mathsf{a}_0 = \frac{64s_\star\sigma^2\rho_1\bar\rho}{nc_2^2\lambda_{\textsf{min}}(\Sigma)} \leq C_0,\] 
for some absolute constant $C_0$, as $p\to\infty$, given the choice of $n$, $\rho_1$ and $\bar\rho$. The condition (\ref{eq:cond:thm:0}) is easily seen to hold for $c_0=2$. Theorem \ref{thm:0} then gives 
\begin{equation}\label{eq:sparse:lm}
\PE_\star\left[\textbf{1}_{\G_1}(Z) \Pi\left(\|\delta\|_0> s_\star\left(1+ \frac{6}{u}\right) + \frac{4}{u}\vert Z\right)\right] \leq \frac{2}{p^2}.
\end{equation}
Since $Y = X\theta_\star + \frac{1}{\sqrt{[\vartheta_{\star}]_{jj}}}V$, where $V\vert X\sim\textbf{N}(0,I_n)$, by a standard union bound argument, and Gaussian tail bounds 
\begin{multline*}
\textbf{1}_{\H}(X)\PP(Z\notin \G_1\vert X) \\
=\textbf{1}_{\H}(X)\PP\left(\max_{1\leq k\leq p+1,\;k\neq j}\;|\pscal{X_k}{V}|>\sqrt{6\tau_\Sigma(1+c_1)n\log(p)}\; \vert X\right)\leq \frac{2}{p^2}.\end{multline*}
Therefore, (\ref{eq:sparse:lm}) becomes 
\begin{equation}\label{eq:sparse:lm:2}
\PE_\star\left[\textbf{1}_{\H}(X) \Pi\left(\|\delta\|_0> s_\star\left(1+ \frac{6}{u}\right) +  \frac{4}{u}\vert Z\right)\right] \leq \frac{4}{p^2}.
\end{equation}

 \paragraph{\underline{\tt Contraction and rate}}\;\; 
Set $\bar s = s_\star\left(1+ \frac{6}{u}\right) + \frac{4}{u}$.  We now apply Theorem \ref{thm1} to $\Pi^{(j)}$. With similar calculations as above, for $\|\delta\|_0\leq \bar s$, and $u\in\rset^p_\delta$, 
\[\mathcal{L}_{\theta_\star}(u;z) \leq - \frac{n c_2^2\lambda_{\textsf{min}}(\Sigma)}{8\sigma^2} \|u-\theta_\star\|_2^2,\]
provided that the sample size $n$ satisfies (\ref{ss:cond:lm}) which shows that $\G_1\subseteq\e_1(\bar s)$ with the rate function $\r(x) = x^2n c_2^2\lambda_{\textsf{min}}(\Sigma)/(4\sigma^2)$. The contraction rate $\epsilon$ then becomes
\[\epsilon = \frac{4\sigma^2\bar\rho (\bar s+s_\star)^{1/2}}{nc_2^2\lambda_{\textsf{min}}(\Sigma)} = \frac{8\sqrt{2(1+c_1)}}{c_2^2} \frac{\tau_{\Sigma}^{1/2}}{\lambda_{\textsf{min}}(\Sigma) [\vartheta_{\star}]_{jj}^{1/2}} \sqrt{\frac{(\bar s+s_\star)\log(p)}{n}}.\]
The condition (\ref{tech:cond:thm1}) holds by choosing the absolute constant $C\geq 3$ large enough so that $C(1+c_1)\tau_\Sigma\geq (1+u)c_2^2\lambda_{\textsf{min}}(\Sigma) \sigma^2[\vartheta_{\star}]_{jj}$.  Theorem \ref{thm1} then gives
\begin{equation}\label{eq:thm:1:lm}
\PE_\star\left[\textbf{1}_{\H}(X)\Pi\left(\cB^c\vert Z\right)\right]\leq   \PE_\star\left[\textbf{1}_{\G_1}(Z)\Pi\left(\cB^c\vert Z\right)\right]  + \PE_\star\left[\textbf{1}_{\H}(X)\PP(Z\notin \G_1\vert X) \right] \leq \frac{C_0}{p^{2}}.\end{equation}

 \paragraph{\underline{\tt Model selection consistency}}\;\; 
We now apply Theorem \ref{thm:sel} to $\Pi^{(j)}$ With $\bar s= \bar s^{(j)}$ as above, set
\begin{multline*}
\G_2 \eqdef \G_1\ \bigcap_{k=1}^{\bar s-s_\star} \left\{Z = [Y,X]\in\rset^{n\times (p+1)}:\;\right.\\
\left. \max_{\delta\supseteq\delta_\star,\;\|\delta\|_0=s_\star +k}\; (Y-X\theta_\star)'\mathcal{P}_{\delta\setminus\delta_\star}(Y-X\theta_\star) \leq \sigma^2 ku\log(p)\right\},\end{multline*}
where for $\delta\supseteq\delta_\star$, $\mathcal{P}_{\delta\setminus\delta_\star}$ is the orthogonal projector on the sub-space of $\textsf{span}(X_\delta)$ that is orthogonal to $\textsf{span}(X_{\delta_\star})$, where the notation $\textsf{span}(X_\delta)$ denotes the linear space spanned by the columns of $X_\delta$.  We note that $\G_2\subseteq\e_{2}(\bar s)$. Indeed, for $\delta\in\A_{\bar s}$,  and $X\in\H$, the matrix $X_\delta$ is full-rank column. Hence if $X_\delta =Q_{(\delta)} R_{(\delta)}$ is the QR decomposition of $X_\delta$, then
\[\ell^{[\delta]}(\hat\theta_\delta;Z) - \ell^{[\delta_\star]}(\hat\theta_\star;Z) = \frac{1}{2\sigma^2}\|Q_{(\delta\setminus\delta_\star)}'(Y-X\theta_\star)\|_2^2 = \frac{1}{2\sigma^2}(Y-X\theta_\star)'\mathcal{P}_{\delta\setminus\delta_\star} (Y-X\theta_\star).\]
It then follows  that $\G_2\subseteq\e_{2}(\bar s)$. Furthermore, since $\ell$ is quadratic,  (\ref{def:kappa:min}) holds with $\underline{\kappa} = n c_2^2\lambda_{\textsf{min}}(\Sigma)/(4\sigma^2)$, and (\ref{def:kappa:bar:bis}) holds with $\bar\kappa  = (n/\sigma^2)(1+c_1)s_\star^{(j)}\tau_\Sigma$, provided that the sample size condition (\ref{ss:cond:lm}) holds.  Theorem \ref{thm:sel} (applied $\mathsf{a}_2=0$), and (\ref{eq:thm:1:lm})  give for all $k\geq 0$,
\begin{multline}\label{eq:thm:sel:lm:2}
\PE_\star\left[\textbf{1}_{\G_2}(Z)\Pi\left(\cB_k^c\vert Z\right)\right] \leq C_0\left(\sqrt{\frac{\rho_1}{\underline{\kappa}}}\frac{1}{p^{u/2}}\right)^{k+1} +\PE_\star\left[\textsf{1}_{\G_1}(Z)\Pi(\cB^c\vert Z)\right] \\
\leq C_0\left(\sqrt{\frac{\rho_1}{\underline{\kappa}}}\frac{1}{p^{u/2}}\right)^{k+1} + \frac{C_0}{p^{2}}.
\end{multline}
To replace $\G_2$ by $\H$, we write
\[\PE_\star\left[\textbf{1}_{\H}(X)\Pi\left(\cB_k^c\vert Z\right)\right] \leq \PE_\star\left[\textbf{1}_{\G_2}(Z)\Pi\left(\cB_k^c\vert Z\right)\right]   + \PP_\star\left[X\in \H,Z\notin\G_2\right]  .\]
Given $\delta\in\A_{s_\star +k}$, by the Hanson-Wright inequality (Lemma \ref{lem:HW}), 
\begin{multline*}
\textbf{1}_{\H}(X) \PP\left((Y-X\theta_\star)'\mathcal{P}_{\delta\setminus\delta_\star}(Y-X\theta_\star)> \sigma^2 ku\log(p)\vert X\right)  \\
=  \textbf{1}_{\H}(X)\PP\left(V'\mathcal{P}_{\delta\setminus\delta_\star} V > \sigma^2[\vartheta_{\star}]_{jj} ku\log(p)\vert X\right)\leq \frac{1}{p^{\frac{\sigma^2[\vartheta_{\star}]_{jj}uk}{4}}},\end{multline*}
for all $p$ large enough.  Hence by union bound, for $\sigma^2[\vartheta_{\star}]_{jj} u\geq 8$,
\[\textbf{1}_{\H}(X)\PP(Z\notin\G_2\vert X) \leq  \textbf{1}_{\H}(X)\PP(Z\notin\G_1\vert X) + \sum_{k\geq 1} \frac{1}{p^{\frac{\sigma^2[\vartheta_{\star}]_{jj}uk}{4}}} \leq \frac{4}{p^2} .\]
We conclude that for all $k\geq 0$,
\begin{equation}\label{eq:thm:sel:lm:3}
\PE_\star\left[\textbf{1}_{\H}(X)\Pi\left(\cB_k^c\vert Z\right)\right] 
\leq C_0\left(\sqrt{\frac{\rho_1}{\underline{\kappa}}}\frac{1}{p^{u/2}}\right)^{k+1}  +  \frac{C_0}{p^{2}}.
\end{equation} 
 
 \paragraph{\underline{\tt Bernstein-von Mises approximation and variational approximations}}\;\; 
Taking $k=0$ in (\ref{eq:thm:sel:lm:3}) together with Theorem \ref{thm:KL} gives
\begin{equation*}
\PE_\star\left[ \textbf{1}_{\G}(Z) \max_{1\leq j\leq p+1} \KL{\Pi_\star^{(j,\infty)}}{\Pi^{(j)}}\right] \leq \frac{C_0\max_j(\bar s^{(j)} + s_\star^{(j)})}{\min_j [\vartheta_\star]_{jj}} \frac{\log(p)}{n} + \frac{C_0}{p^{\frac{u}{2}-1}} + \frac{C_0}{p}, \end{equation*}
for some absolute constant $C_0$, assuming that $\sigma^2[\vartheta_{\star}]_{jj}u \geq 16$, and $u>2$.
 Finally we apply (\ref{thm:VA:eq}) and (\ref{bound:moment:Q}) applied with $R=\tilde\Pi_\star^{(\infty)}$  to get the stated controls on the variational approximations.  This ends the proof.
\vspace{-0.6cm}
\begin{flushright}
$\square$
\end{flushright}

\medskip
\subsection{Proof of  Corollary \ref{coro:spca}}\label{proof:coro:spca}
The proof follows the same steps as in the proof of Theorem \ref{thm1}. Let 
\begin{multline*}
\bar\rho =\frac{8C_0\vartheta}{\sigma^2}  \sqrt{n\left(\frac{p}{\vartheta}+\log(p)\right)},\;\
\bar\kappa = \frac{c_1n}{\sigma^2},\;\;\r(x) = \frac{c_2n}{\sigma^2} x^2, \;\; \\
\mbox{ and } \;\; \epsilon = \frac{8C_0\vartheta}{c_2} \sqrt{\frac{\frac{p}{\vartheta}+\log(p)}{n}\left(\bar s+s_\star\right)}, \end{multline*}
for some absolute constants $C_0,c_1,c_2$, that we specify later.
For $\theta_0\in\{\theta_\star,-\theta_\star\}$, let $\cB_{\theta_0}$ be the set $\cB$ defined in (\ref{def:set:B}) but with $\theta_\star$ replaced by $\theta_0$, $\epsilon$ as above, and for some absolute constant $C,C_1$. Similarly let $\e_{0,\theta_0}$ (resp. $\e_{1,\theta_0}(\bar s)$) be the set $\e_0$ (resp. $\e_1(\bar s)$) but with $\theta_\star$ replaced by $\theta_0$, and $\bar\kappa,\bar\rho$ as above and the rate function $\r$ as above. Also for absolute constant $C \geq 3$, set
\begin{multline*}
\F_{1,\theta_0}\eqdef \bigcup_{\delta\in \Delta_{\bar s}} \{\delta\}\times\left\{\theta\in\rset^p:\; \|\theta_\delta-\theta_0\|_2> C\epsilon\right\},\;\;\;\\
\F_{2,\theta_0}\eqdef \bigcup_{\delta\in\Delta_{\bar s}} \{\delta\}\times\left\{\theta\in\rset^p:\;\|\theta_\delta-\theta_0\|_2\leq C\epsilon,\;\;\mbox{ and } \;\;\|\theta-\theta_\delta\|_2>\epsilon_1 \right\}.\end{multline*}
From the definitions we can write $\Delta\times\rset^p=\{\delta:\;\|\delta\|_0>\bar s\} \cup \F_{1,\theta_0}\cup\F_{2,\theta_0}\cup \cB_{\theta_0}$. Using this and $\Pi(\|\delta\|_0>\bar s\vert X)=0$, it follows that
\[\Pi\left(\cB_{\theta_0}\vert X\right) = 1- \Pi\left(\F_{1,\theta_0}\vert X\right) - \Pi\left(\F_{2,\theta_0}\vert X\right). \]
Hence it suffices to show that for $\varepsilon\in\{-1,1\}$,
\[\lim_{p\to\infty} \PE_\star\left[\textbf{1}_{\{\textsf{sign}(\pscal{V_1}{\theta_\star})=\varepsilon\}} \left(\Pi\left(\F_{1,\varepsilon\theta_\star}\vert X\right) + \Pi\left(\F_{2,\varepsilon\theta_\star}\vert X\right)\right) \right]=0.\]
We have
\begin{multline}\label{eq:proof:coro:spca:eq1}
\PE_\star\left[\textbf{1}_{\{\textsf{sign}(\pscal{V_1}{\theta_\star})=\varepsilon\}} \left(\Pi\left(\F_{1,\varepsilon\theta_\star}\vert X\right) + \Pi\left(\F_{2,\varepsilon\theta_\star}\vert X\right)\right) \right]\\
 \leq  \PP_\star\left(X\notin\e_{1,\varepsilon\theta_\star}(\bar s),\textsf{sign}(\pscal{V_1}{\theta_\star})=\varepsilon\right) \\
+ \PE_\star\left[\textbf{1}_{\e_{1,\varepsilon\theta_\star}(\bar s)} (X)  \left(\Pi\left(\F_{1,\varepsilon\theta_\star}\vert X\right) + \Pi\left(\F_{2,\varepsilon\theta_\star}\vert X\right)\right) \right].
\end{multline}
With the same argument as in the proof of Theorem \ref{thm1}, we have 
\[\PE_\star\left[\textbf{1}_{\e_{1,\varepsilon\theta_\star}(\bar s)} (X) \Pi\left(\F_{2,\varepsilon\theta_\star}\vert X\right) \right]\leq 4e^{-p}.\]
We use the test constructed in Lemma \ref{test} with $\Theta_\star=\{\theta_\star,-\theta_\star\}$, and $M=C$ to write
\begin{multline*}
\PE_\star\left[\textbf{1}_{\e_{1,\varepsilon\theta_\star}(\bar s)} (X)  \Pi\left(\F_{1,\varepsilon\theta_\star}\vert X\right) \right]\leq \E_\star[\phi(X)] \\
+\PE_\star\left[\textbf{1}_{\e_{1,\varepsilon\theta_\star}(\bar s)} (X)\left(1-\phi(X)\right)\Pi\left(\F_{1,\varepsilon\theta_\star}\vert X\right) \right],\end{multline*}
and
\[\PE_\star[\phi(X)] \leq \frac{4(9p)^{\bar s}e^{-\frac{C}{8}\bar \rho_1(\bar s+s_\star)^{1/2}\epsilon}}{1-e^{-\frac{C}{8}\bar \rho_1(\bar s+s_\star)^{1/2}\epsilon}}\to 0,\]
as $p\to\infty$, by appropriately choosing the absolute constant $C$. 
The same argument leading to (\ref{eq:proof:thm:contrac:eq4}) applies to the second term on the right hand side of the last display, and we deduce that
\[\lim_{p\to\infty} \PE_\star\left[\textbf{1}_{\e_{1,\varepsilon\theta_\star}(\bar s)} (X)\left(1-\phi(X)\right)\Pi\left(\F_{1,\varepsilon\theta_\star}\vert X\right) \right]=0.\]
Collecting these limiting behaviors we conclude from (\ref{eq:proof:coro:spca:eq1}) that
\begin{multline*}
 \lim_{p\to\infty}  \PE_\star\left[\textbf{1}_{\{\textsf{sign}(\pscal{V_1}{\theta_\star})=\varepsilon\}} \left(\Pi\left(\F_{1,\varepsilon\theta_\star}\vert X\right) + \Pi\left(\F_{2,\varepsilon\theta_\star}\vert X\right)\right) \right]\\
 \leq  \lim_{p\to\infty} \PP_\star\left(X\notin\e_{1,\varepsilon\theta_\star}(\bar s),\textsf{sign}(\pscal{V_1}{\theta_\star})=\varepsilon\right).
\end{multline*}
Hence it suffices to show that with $\bar\kappa$, $\bar\rho$, and the rate function $\r$ as above we have $ \PP_\star\left(X\notin\e_{1,\varepsilon\theta_\star}(\bar s) \vert \textsf{sign}(\pscal{V_1}{\theta_\star})=\varepsilon\right)\to 0$, as $p\to\infty$.

For $\theta_0\in\{\theta_\star,-\theta_\star\}$, and $\theta\in\rset^p_\delta$, for any $\delta\in\Delta_{\bar s}$,
\[\mathcal{L}_{\theta_0}(\theta;X) = -\frac{n}{\sigma^2}(\theta - \theta_0)'\left(\frac{X'X}{n}\right)(\theta-\theta_0).\]
Lemma 1 of \cite{ravikumaretal11}, and Theorem 1 of \cite{raskutti:etal:10} then show that the function $\theta\mapsto \mathcal{L}_{\theta_0}(\theta;X)$ satisfies the requirements of $\e_{1,\varepsilon\theta_\star}(\bar s)$ with high probability, provided that the sample size $n$ satisfies  $n\geq C_0 (\bar s+ s_\star) \log(p)$, for some absolute constant $C_0$. Hence it remains only to show that 
\begin{equation}\label{prob:grad:spca}
\lim_{p\to\infty} \PP_\star\left(\|\nabla\ell(\varepsilon\theta_\star;X)\|_\infty>\frac{\bar \rho}{2} ,\textsf{sign}(\pscal{V_1}{\theta_\star})=\varepsilon\right) = 0,\end{equation}
where $\bar \rho$ is as defined at the beginning of the proof. The largest eigenvalue of $\Sigma$ is $1+\vartheta$ with corresponding eigenvector $\theta_\star$. Hence, by the Davis-Kahan's theorem (Corollary 1 \cite{yu:etal:15}), on $\{\textsf{sign}(\pscal{V_1}{\theta_\star})=\varepsilon\}$, \begin{equation}\label{eq:dk}
\|V_1 -\varepsilon\theta_\star\|_2\leq \frac{4}{\vartheta}\left\|\frac{X'X}{n} -\Sigma\right\|_2.\end{equation}
Noting that $y=\Lambda_{11} U_1 = XV_1$, we have for $\theta_0\in\{\theta_\star,-\theta_\star\}$,
\begin{multline*}
\nabla\ell(\theta_0;X) = \frac{1}{\sigma^2}X'(y-X\theta_0) = \frac{1}{\sigma^2}X'X(V_1 -\theta_0) \\
= \frac{1}{\sigma^2}(X'X-n\Sigma)(V_1 -\theta_0) + \frac{n}{\sigma^2}\Sigma(V_1 -\theta_0).\end{multline*}
Hence
\[\|\nabla\ell(\theta_0;X)\|_\infty\leq \frac{n}{\sigma^2}\left(\left\|\frac{X'X}{n}-\Sigma\right\|_2 + \left(1+\|\theta_\star\|_\infty\vartheta\right)\right)\|V_1-\theta_0\|_2.\]
This bound together with the Davis-Kahan's theorem (\ref{eq:dk}) yields that on $\{\textsf{sign}(\pscal{V_1}{\theta_\star})=\varepsilon\}$, we have
\begin{equation}\label{eq:bound:cov:spca}
\|\nabla\ell(\varepsilon\theta_\star;X)\|_\infty\leq \frac{4n}{\sigma^2\vartheta}\left[ \left\|\frac{X'X}{n}-\Sigma\right\|_2 + \left(1+\|\theta_\star\|_\infty\vartheta\right)\right] \left\|\frac{X'X}{n} - \Sigma \right\|_2.\end{equation}
Note then that if the covariance $X'X/n$ satisfies
\begin{equation}\label{dev:cov:intrinsic:dim}
\left\|\frac{X'X}{n} - \Sigma\right\|_2 \leq C_0 \left[\sqrt{\frac{\frac{p}{\vartheta}+\log(p)}{n}} +\frac{\frac{p}{\vartheta}+\log(p)}{n} \right](\vartheta+1),\end{equation}
for some absolute constant $C_0$, then for $n\geq C_0(\frac{p}{\vartheta} + \log(p))$, we get $\|(X'X)/n - \Sigma\|_2 \leq C_0\vartheta$, and in that case (\ref{eq:bound:cov:spca}) gives 
\[\|\nabla\ell(\varepsilon\theta_\star;X)\|_\infty\leq \frac{4nC_0}{\sigma^2} \left\|\frac{X'X}{n} - \Sigma \right\|_2 \leq \frac{4C_0\vartheta}{\sigma^2}  \sqrt{n\left(\frac{p}{\vartheta}+\log(p)\right)} =\frac{\bar \rho_1}{2},\]
for some absolute constant $C_0$. This means that the probability on the right hand side of (\ref{prob:grad:spca}) is upper bounded by the probability that (\ref{dev:cov:intrinsic:dim}) fails.  The matrix $\Sigma$ has the property that $\textsf{Tr}(\Sigma)/\|\Sigma\|_2 =(p+\vartheta)/(1+\vartheta) \leq 1+ (p/\vartheta)$. Using this and by deviation bound for Gaussian distribution with covariance matrix with  low intrinsic dimension (see e.g. \cite{vershynin:18}~Theorem 9.2.4), (\ref{dev:cov:intrinsic:dim})  holds that with probability at least $1-1/p$.  Hence the results.
\vspace{-0.6cm}
\begin{flushright}
$\square$
\end{flushright}

\section{Some technical results}
%

We make use of the following expression of the KL-divergence between two Gaussian distributions.

\begin{lemma}\label{lem:KL:Gaussian}
For $i=1,2$ let $\pi_i$ denote the probability distribution of the Gaussian distribution $\textbf{N}(\mu_i,\Sigma_i)$. We have
\[\KL{\pi_1}{\pi_2} = \frac{1}{2}(\mu_2-\mu_1)'\Sigma_2^{-1}(\mu_2-\mu_1) + \frac{1}{2}\log\left(\frac{\det(\Sigma_2)}{\det(\Sigma_1)}\right) +\frac{1}{2}\textsf{Tr}(\Sigma_2^{-1}\Sigma_1) -\frac{p}{2}.\]
\end{lemma}

The following lemma follows readily from standard Gaussian deviation bounds. We omit the details.
\begin{lemma}\label{bound:l2:dev}
Suppose that a $\rset^p$-valued random variable $X$ has density $f(x)\propto e^{-\ell(x) -\rho\|x\|_2^2/2}$, for  a twice differentiable function $\ell$ such that $m I_p \preceq \nabla^{(2)} \ell \preceq M I_p$, for some constants $0<m\leq M$, and $\rho>0$. Let $\mu$ denote the mode of $\ell$. For all $t\geq 4\max\left(\frac{\rho}{\rho + m}\|\mu\|_2,\sqrt{\frac{p}{\rho + m}}\right)$ we have
\begin{multline*}
\PP\left(\|X-\mu\|_2>t\right) \leq \left(\frac{M+\rho}{m+\rho}\right)^{\frac{p}{2}} e^{-\frac{t^2(m+\rho)}{16}},\;   \\
\mbox{ and } \;\; \PE\left( \|X-\mu\|_2^2 \textbf{1}_{\{\|X-\mu\|_2>t\}}\right) \leq t^2 \left(\frac{M+\rho}{m+\rho}\right)^{\frac{p}{2}} e^{-\frac{t^2(m+\rho)}{32}}.\end{multline*}
\end{lemma}
\begin{proof}
By Taylor expansion of $\ell$ around $\mu$:
\[-\frac{M}{2}\|x-\mu\|_2^2 - \frac{\rho}{2}\|x\|_2^2 \leq \ell(\mu)-\ell(x) -\frac{\rho}{2}\|x\|_2^2 \leq -\frac{m}{2}\|x-\mu\|_2^2 - \frac{\rho}{2}\|x\|_2^2,\;\;\; x\in\rset^p.\]
This implies that
\[\int_{\rset^p} e^{\ell(\mu) -\ell(x) -\frac{\rho}{2}\|x\|_2^2}\rmd x \geq  e^{-\frac{M\rho}{2(M+\rho)}\|\mu\|_2^2}\left(\frac{2\pi}{\rho + M}\right)^{p/2}.\]
Therefore, for any $t>0$,
\begin{eqnarray*}
\PP\left(\|X-\mu\|_2>t\right) & \leq & e^{\frac{M\rho}{2(M+\rho)}\|\mu\|_2^2}\left(\frac{\rho+M}{\rho + m}\right)^{p/2} \PP\left(\left\|\frac{Z}{\sqrt{\rho+m}} - \frac{\rho\mu}{\rho+m} \right\|_2>t\right),\\
& \leq & e^{\frac{\rho}{2}\|\mu\|_2^2}\left(\frac{\rho+M}{\rho + m}\right)^{p/2} e^{-\frac{1}{2}\left(t\sqrt{m+\rho} -\frac{\rho\|\mu\|_2}{\sqrt{m+\rho}} -\sqrt{p}\right)^2}.\end{eqnarray*}
where $Z\sim\textbf{N}_p(0,I_p)$.  For $t\geq 4\max(\rho\|\mu\|_2/(\rho+m),\sqrt{\frac{p}{m+\rho}})$, this yields
 \[\PP\left(\|X-\mu\|_2>t\right) \leq \left(\frac{\rho+M}{\rho + m}\right)^{p/2} e^{-\frac{t^2(m+\rho)}{16}}.\]
By Holder's inequality
\[\PE\left( \|X-\mu\|_2^2 \textbf{1}_{\{\|X-\mu\|_2>t\}}\right) \leq \PE^{1/2}( \|X-\mu\|_2^4) \PP^{1/2}\left(\|X-\mu\|_2>t\right).\]
With the same calculations as above,
\begin{eqnarray*}
\PE( \|X-\mu\|_2^4) &  \leq & e^{\frac{\rho}{2}\|\mu\|_2^2}\left(\frac{\rho+M}{\rho + m}\right)^{p/2} \PE\left(\left\|\frac{Z}{\sqrt{\rho+m}} - \frac{\rho\mu}{\rho+m} \right\|_2^4\right),\\
 & \leq & 8 e^{\frac{\rho}{2}\|\mu\|_2^2}\left(\frac{\rho+M}{\rho + m}\right)^{p/2} \left(\frac{3p^2}{(m+\rho)^2} + \frac{\rho^4\|\mu\|_2^4}{(m+\rho)^4}\right)\\
& \leq & e^{\frac{\rho}{2}\|\mu\|_2^2}\left(\frac{\rho+M}{\rho + m}\right)^{p/2} \frac{t^4}{8},
\end{eqnarray*}
using the assumption $t\geq 4\max(\frac{\rho}{\rho + m}\|\mu\|_2,\sqrt{\frac{p}{m+\rho}})$, which implies the second inequality.
\end{proof}

The next results establishes the strong convexity of the KL divergence. The proof is due to I. Pinelis (\cite{pinelis:KL:18}). We reproduce it here for completeness.
\begin{lemma}\label{lem:strong:conv:kl}
Let $P_0,P_1$ be two probability measures that are absolutely continuous with respect to a probability measure $Q$, on some measure space $\Xset$. For any $t\in (0,1)$, we have
\[t \KL{P_1}{Q} + (1-t)\KL{P_0}{Q} \geq \KL{tP_1 + (1-t)P_0}{Q} + \frac{t(1-t)}{2}\|P_1-P_0\|_\tv^2.\]
\end{lemma}
\begin{proof}
For $j=0,1$, set $f_j = \rmd P_j/\rmd Q$. For $t\in [0,1]$, set $f_t = t f_1 + (1-t)f_0$, and $P_t(\rmd u) = f_t(u)Q(\rmd u)$. Set $h(x)=x\log(x)$, $x\geq 0$.   By Taylor expansion with integral remainder, for $j\in\{0,1\}$, $t\in [0,1]$, and $x\in\Xset$,  we have
\begin{multline*}
h(f_j(u)) = h(f_t(u)) + \left(f_j(u)-f_t(u)\right) h'(f_t(u)) \\
+\left(f_j(u)-f_t(u)\right)^2\int_0^1h^{''}\left((1-\alpha)f_t(u)+\alpha f_j(u)\right)(1-\alpha)\rmd\alpha.\end{multline*}
 $h'(x)=\log(x)-1$, and $h^{''}(x)=1/x$, so that
 \begin{multline}\label{KL:sc:eq1}
 th(f_1(u)) + (1-t)h(f_0(u)) -h(f_t(u)=t(1-t)\left(f_1(u)-f_0(u)\right)^2\\
 \times\int_0^1\left[\frac{t}{(1-\alpha)f_t(u)+\alpha f_0(u)} + \frac{1-t}{(1-\alpha)f_t(u)+\alpha f_1(u)}\right](1-\alpha)\rmd\alpha.\end{multline}
 We can write $(1-\alpha)f_t(u)+\alpha f_0(u) = f_{s_0(\alpha,t)}(u)$, where $s_0(\alpha,t) = (1-\alpha)t$. Similarly, $(1-\alpha)f_t(u)+\alpha f_1(u) = f_{s_1(\alpha,t)}$, where $s_1(\alpha,t) = \alpha + t(1-\alpha)$. Using these expressions, and integrating both sides of (\ref{KL:sc:eq1}) gives
 \begin{multline*}
 t \KL{P_1}{Q} + (1-t)\KL{P_0}{Q} -\KL{P_t}{Q} \\
 = t(1-t)\int_0^1(1-\alpha)\left[ t\int\frac{(f_1(u)-f_0(u))^2}{f_{s_0(\alpha,t)}(u)}Q(\rmd u) + (1-t) \int\frac{(f_1(u)-f_0(u))^2}{f_{s_1(\alpha,t)}(u)}Q(\rmd u)\right]\rmd\alpha.\end{multline*}
 For any $s\in (0,1)$,
 \begin{multline*}
 \int\frac{(f_1(u)-f_0(u))^2}{f_s(u)}Q(\rmd u) = \frac{1}{(1-s)^2}\int\frac{(f_1(u)-f_s(u))^2}{f_s(u)}Q(\rmd u)\\
 =\frac{1}{(1-s)^2}\int\left(\frac{f_1(u)}{f_s(u)} -1\right)^2 f_s(u)Q(\rmd u)\geq \frac{1}{(1-s)^2}\left[\int\left|\frac{f_1(u)}{f_s(u)} -1\right| Q_s(\rmd u)\right]^2 \\
 =\frac{1}{(1-s)^2}\|P_s-P_1\|_\tv^2 = \|P_1-P_0\|_\tv^2.
 \end{multline*}
 We conclude that
 \begin{multline*} t \KL{P_1}{Q} + (1-t)\KL{P_0}{Q} -\KL{P_t}{Q}\\
 \geq t(1-t)\|P_1-P_0\|_\tv^2\int_0^1\alpha(1-\alpha)\rmd\alpha = \frac{t(1-t)}{2}\|P_1-P_0\|_\tv^2,\end{multline*}
as claimed. 
\end{proof}

The following deviation bound is known as the Hanson-Wright inequality. This version is taken from (\cite{vershynin:18}).

\begin{lemma}\label{lem:HW}
Let $X = (X_1,\ldots,X_n)$ be a random vector with independent mean zero components. 
Suppose that  there exists $\sigma>0$ such that for all unit-vector $u\in\rset^n$, and all $t\geq 0$, $\PP(|\pscal{u}{X}|>t) \leq 2e^{-t^2/(2\sigma^2)}$. Then for all $t\geq 6$, it holds
\begin{equation}\label{eq:HW}
\PP\left[X'AX > (4+t)\sigma^2n\lambda_{\textsf{max}}(A)\right] \leq e^{-\frac{ct n}{6}},\end{equation}
for some absolute constant $c$. In the particular case where $X\sim\textbf{N}_n(0,I_n)$, $\sigma=1$, and we can take $c=3$.
\end{lemma}

We will also need the following lemma on determinants of sub-matrices.
\begin{lemma}\label{lem:interlacing}
If symmetric positive definite matrices $A,M$ and $D\in\rset^{q\times q}$ are such that $M=\left(\begin{array}{cc}A & B\\B'&D\end{array}\right)$, then 
\[\det(A) \lambda_{\textsf{min}}(M)^q \leq \det(M) \leq \det(A) \lambda_{\textsf{max}}(M)^q.\]
\end{lemma}
\begin{proof}
This follows from Cauchy's interlacing property for eigenvalues. See for instance \cite{horn:johnson} Theorem 4.3.17.
\end{proof}

\section{Algorithms for linear regression models}
\label{append:algo}
Both algorithms are initialized from the lasso solution and its support. The VA also needs an initial value of the matrix $C$ which we take as $(c/n) I_p$, with $c= 0.001$.

\begin{algorithm}[Gibbs sampler for (\ref{post:linmod})]\label{algo:gibbs}
At the $k$-th iteration, given $(\delta^{(k)},\theta^{(k)})$:
\begin{enumerate}
\item For all $j$ such that $\delta^{(k)}_j=0$, draw $\theta_j^{(k+1)}\sim \textbf{N}(0,\rho_0^{-1})$. Then draw jointly $[\theta^{(k+1)}]_\delta\sim \textbf{N}(m^{(k)},\Sigma^{(k)})$, where
\[m^{(k)} = \left(X_{\delta^{(k)}}'X_{\delta^{(k)}} + \sigma^2\rho_1 I_{\|\delta^{(k)}\|_0}\right)^{-1}X_{\delta^{(k)}}'z,\;\; \Sigma^{(k)} =\sigma^2\left(X_{\delta^{(k)}}'X_{\delta^{(k)}} + \sigma^2\rho_1 I_{\|\delta^{(k)}\|_0}\right)^{-1} .\]
\item \begin{enumerate}
\item Given $\theta^{(k+1)}=\theta$, set $\delta^{(k+1)}=\delta^{(k)}$, and repeat for $j=1,\ldots,p$.  Draw $\iota\sim\textbf{Ber}(0.5)$. If $\delta^{(k)}_j =0$, and $\iota=1$, with probability  $\min(1,A_j)/2$ change $\delta^{(k+1)}_j$ to $\iota$. If $\delta^{(k)}_j =1$, and $\iota=0$, with probability $\min(1,A_j^{-1})/2$, change $\delta^{(k+1)}_j$ to $\iota$; where
\[ A_j  = \frac{\textsf{q}}{1-\textsf{q}}\sqrt{\frac{\rho_1}{\rho_0}} e^{-\left(\rho_1-\rho_0\right)\frac{\theta_j^2}{2}}e^{-\frac{\theta_j^2}{2\sigma^2}\|X_j\|_2^2 +\frac{\theta_j}{\sigma^2}\left(\pscal{X_j}{Y}-\sum_{i:\;\delta^{(k+1)}_i=1,\;i\neq j} \theta_i\pscal{X_j}{X_i}\right)}.\] 
\end{enumerate}
\end{enumerate}
\end{algorithm}

\medskip
%
%

\begin{algorithm}[Midsize VA approximation for (\ref{post:linmod}) using template $\delta^{(\textsf{i})}$]
Given $\alpha^{(k)}, \mu^{(k)}$, and $C^{(k)}$
\begin{enumerate}
\item 
\begin{enumerate}
\item Set $\bar\alpha = \alpha^{(k)}$. For $j=1,\ldots,p$  update $\bar\alpha_j$ as  $\bar\alpha_j=\frac{1}{1+R_j}$, where
\[R_j =\frac{1-\textsf{q}}{\textsf{q}}\sqrt{\frac{\rho_0}{\rho_1}}e^{\left(\rho_1-\rho_0\right)\frac{\widehat{\theta_j^2}}{2}}e^{\frac{1}{2\sigma^2}\left[\widehat{\theta_j^2}\|X_j\|_2^2 -2\mu_j^{(k)}\pscal{X_j}{y-\sum_{i\neq j}\mu^{(k)}_i\bar\alpha_{i}X_i} + S_j\right]}, \]
where $\widehat{\theta_j^2} = (\mu_j^{(k)})^2 + C_{jj}^{(k)}$, and $S_j = 2\sum_{i\neq j} \bar\alpha_iC_{ij}\pscal{X_j}{X_i}$.
\item Set $\alpha^{(k+1)}=\bar\alpha$.
\end{enumerate}
\item 
\begin{enumerate}
\item For each $j$ such that $\delta^{(\textsf{i})}_j=0$, set
\[C_{jj}^{(k+1)} = \frac{1}{\left(\rho_1 + \frac{\|X_j\|_2^2}{\sigma^2}\right)\alpha_j^{(k+1)} + \rho_0(1-\alpha_j^{(k+1)}) },\]
and 
\[\mu_j = \frac{C_{jj}^{(k+1)}}{\sigma^2}\alpha_j^{(k+1)}\pscal{X_j}{y-\sum_{i\neq j}\alpha_i^{(k+1)}\bar\mu_iX_i}.\]
\item If $\|\delta^{(\textsf{i})}\|_0>0$ do the following. Set $\tilde y= y -\sum_{j:\delta^{(\textsf{i})}_j=0} \alpha^{(k+1)}_j\mu^{(k+1)}_j X_j$. Form the matrix  $M\in\rset^{p\times p}$ such that $M_{ij}= \alpha_i^{(k+1)}\|X_i\|_2^2$, if $i=j$, and $M_{ij}= \alpha_i^{(k+1)}\alpha_j^{(k+1)}\pscal{X_i}{X_j}$ if $i\neq j$. Let $\Lambda\in\rset^{p\times p}$ be the diagonal matrix such that $\Lambda_{jj} = \alpha_j^{(k+1)}\rho_1 + \rho_0(1-\alpha_j^{(k+1)})$. Then we update $C^{(k)}$ to 
\[[C^{(k+1)}]_{\delta^{(\textsf{i})},\delta^{(\textsf{i})}} = \left(\left[\Lambda + \frac{1}{\sigma^2}M\right]_{\delta^{(\textsf{i})},\delta^{(\textsf{i})}}\right)^{-1},\]
and we update $\mu^{(k)}$ to 
\[[\mu^{(k+1)}]_{\delta^{(\textsf{i})}} = \left([C^{(k+1)}]_{\delta^{(\textsf{i})},\delta^{(\textsf{i})}}\right)\left[\mathsf{diag}(\alpha^{(k+1)})\right]_{\delta^{(\textsf{i})},\delta^{(\textsf{i})}} X_{\delta^{(\textsf{i})}}' \tilde y,\]
where $\mathsf{diag}(\alpha^{(k+1)})$ is the diagonal matrix with diagonal given by $\alpha^{(k+1)}$.
\end{enumerate}
\end{enumerate}
\end{algorithm}
\begin{remark}
Setting $\delta^{(\textsf{i})}={\bf 0}_p$ in the algorithm above yields the mean field variational approximation (skinny-VA). And taking $\delta^{(\textsf{i})}$ as the vector will all components equal to $1$ yields the full variational approximation (full-VA).
\end{remark}

\end{document}